\documentclass[10pt,a4paper]{article}

\usepackage{amssymb}
\usepackage{bm}
\usepackage{graphicx}
\usepackage[centertags]{amsmath}
\usepackage{amsfonts}
\usepackage{amsthm}
\usepackage{graphicx}
\usepackage{slashed}
\usepackage{float}

\usepackage[usenames,dvipsnames,svgnames,table]{xcolor}
\usepackage{hyperref}
\hypersetup{linktoc=page}

\numberwithin{equation}{section}
\graphicspath{/Users/yannis/Documents/}

\newtheorem{thm}{Theorem}

\newtheorem{prop}[thm]{Proposition}

\newtheorem{conj}{Conjecture}

\theoremstyle{definition}
\newtheorem{rem}{Remark}

\newcommand{\rr}{\mathbb{R}}

\newcommand{\la}{\lambda}
\newcommand{\ee}{\varepsilon}
\newcommand{\ph}{\varphi}

\newcommand{\si}{\Sigma}
\newcommand{\sis}{\widetilde{\Sigma}}
\newcommand{\rrr}{\mathcal{R}}
\newcommand{\rrrr}{\widetilde{\mathcal{R}}}
\newcommand{\mi}{\mathring{\mu}}
\newcommand{\ho}{\mathcal{H}^{+}}
\newcommand{\nnn}{\mathcal{N}}
\newcommand{\so}{\bar{S}}
\newcommand{\aaa}{\alpha}

\newcommand{\meg}{\geqslant}
\newcommand{\mik}{\leqslant}

\begin{document}
\title{Nonlinear Wave Equations With Null Condition On Extremal Reissner-Nordstr\"{o}m Spacetimes I: Spherical Symmetry}
\author{Yannis Angelopoulos\thanks{Department of Mathematics, University of Toronto, Bahen Centre, 40 St. George Street,  Toronto, Ontario, M5S 2E4, Canada}}
\date{}
\maketitle

\begin{abstract}
We study spherically symmetric solutions of semilinear wave equations in the case where the nonlinearity satisfies the null condition on extremal Reissner--Nordstr\"{o}m black hole spacetimes. We show that solutions which arise from sufficiently small compactly supported smooth data prescribed on a Cauchy hypersfurace $\widetilde{\Sigma}_0$ crossing the future event horizon $\mathcal{H}^{+}$ are globally well-posed in the domain of outer communications up to and including $\mathcal{H}^{+}$. Our method allows us to close all bootstrap estimates under very weak decay results (compatible with those known for the linear case). Moreover we establish a certain number of non-decay and blow-up results along the horizon $\mathcal{H}^{+}$ which generalize known instability results for the linear case. Our results apply to spherically symmetric wave maps for a wide class of target spaces.
\end{abstract}
\tableofcontents
\section{Introduction}
The analysis of linear and non-linear wave equations on black hole spacetimes has recently been the object of intense study.  Following this thread of research in this paper we  focus on the the study of spherically symmetric solutions of some nonlinear wave equation on the extremal Reissner--Nordstr\"{o}m black hole spacetime.

The extremal Reissner--Nordstr\"{o}m solution is an $1$--parameter family of solutions (the parameter being the mass $M > 0$) to the Einstein--Maxwell equations and it has the following form:
\begin{equation}\label{rn1}
 g = -\left(1- \frac{M}{r} \right)^2 dt^2 + \left(1- \frac{M}{r} \right)^{-2} dr^2 + r^2 \gamma_{\mathbb{S}^2},
 \end{equation}
where $\gamma_{\mathbb{S}^2} $ is the standard metric on the 2-sphere $\mathbb{S}^2$. This spacetime has attracted the interest of both mathematicians and physicists, especially during the last 5 years. 

We will study the following equation:
\begin{equation}\label{nw}
\left\{\begin{aligned}
       \Box_g \psi = A(\psi ) g^{\alpha \beta} \partial_{\alpha} \psi \partial_{\beta} \psi  + O (|\psi|^l ) \\
       \psi |_{\Sigma_0 } = \ee \psi_0, \quad n_{\widetilde{\Sigma}_0} \psi |_{\Sigma_0 } = \ee \psi_1, \
       \end{aligned} \right.
\end{equation}
where the foliation $\{ \widetilde{\Sigma}_{\tau}\}_{\tau \meg 0}$ (and hence the initial hypersurface $\widetilde{\Sigma}_0$) will be defined later,  $\psi_0$ is smooth compactly supported, $\ee > 0$ is such that $0\leq \ee\leq \ee'$ where $\ee'$ is some small constant to be chosen later, $l$ is some constant to be specified later as well, and $A$ is a function that is bounded along with all its derivatives:
\begin{equation}\label{boundedgeom}
|A^{(k)} (\psi ) | \mik a_k \mbox{ for all $k\in \mathbb{N}$.} 
\end{equation}
Note here that we can also allow the condition 
\begin{equation}\label{wmsphere}
|A (\psi )| \mik a_0 |\psi | + O (|\psi|^2 ).
\end{equation}
A special case of equation \eqref{nw} is the very well known and very well studied Wave Maps equation. This can be seen by dropping the cubic and higher order terms, and taking into consideration only the quadratic nonlinear term
$$ A(\psi ) g^{\alpha \beta} \partial_{\alpha} \psi \partial_{\beta} \psi, $$
which satisfies the so called null condition. In its extrinsic formulation, $A(\psi)$ is the second fundamental form of the target manifold. The condition of boundedness of $A$ along with all its derivatives \eqref{boundedgeom} translates to the geometric condition that the target manifold for the Wave Maps equation has bounded geometry. The second more relaxed condition on $A$ (condition \eqref{wmsphere}) covers the case of the target being the sphere. Our techniques probably allow for more relaxed conditions, but we won't deal with this here. We will actually work with \eqref{nw} without the $l$-th or higher order terms (which are easier to treat as it will be demonstrated in Section \ref{othernon}) and just with the assumptions \eqref{boundedgeom}, since for our results the assumption \eqref{wmsphere} makes things even simpler.

We will assume that our initial data $(\ee \psi_0 , \ee \psi_1 )$ are spherically symmetric, which will give us that our solution will be spherically symmetric as well, due to the form of the nonlinearity.

Our main results in this paper can be summarized as follows (see Section \ref{Main} for the precise statements):

1) Spherically symmetric solutions of \eqref{nw} that come from sufficiently small compactly supported smooth data \textit{are unique and  exist globally } in the domain of outer communications up to and including the event horizon $\mathcal{H}^{+}$.

2) The solution itself \textit{decays} with respect to a parameter that moves across the future null infinity $\mathcal{I}^{+}$ and the future event horizon $\mathcal{H}^{+}$. The first order derivatives with respect to $v$ and $r$ (in the Eddington--Finkelstein coordinate system -- see Section \ref{gernfm}) on the other hand remain \textit{bounded}, and their size is comparable always to the one of the initial data.

3) On the future event horizon $\mathcal{H}^{+}$ we demonstrate the existence of an \textit{almost conservation law} (see Section \ref{acwothh}). More specifically we have that on $\mathcal{H}^{+}$ the $Y \psi$ derivative (where $Y = \partial_r$ in the Eddington--Finkelstein system of coordinates, see Section \ref{gernfm} again) is almost constant and hence generically\textit{ does not decay}. This coupled with our dispersive estimates implies that second and higher order derivatives with respect to $r$ \textit{blow up} asymptotically on $\mathcal{H}^{+}$ (see Section \ref{asbufhd}).

The instabilities observed in this nonlinear problem are analogous to the ones that were previously observed in the linear case (see Section \ref{ernnn}). The behaviour of solutions to the same nonlinear problem in the non-extremal setting is drastically different as none of these phenomena occur, while the global well-posedness proof presents significantly fewer difficulties.

\subsection{Linear Waves On Black Hole Spacetimes}
From the point of view of mathematical general relativity, the study of the linear wave equation
$$ \Box_g \psi = 0 $$
on a black hole spacetime is seen as a rather crude linearization of the problem of nonlinear black hole stability (which involves of course the Einstein equations -- see the seminal work of Christodoulou and Klainerman \cite{christab} for the proof of the nonlinear stability of Minkowski spacetime). This linear problem has been intensively studied, especially during the last 20 years. We refer to the works of Dafermos and Rodnianski \cite{enadio}, \cite{redshift}, \cite{newmethod}, Blue and Soffer \cite{blu1}, Marzuola, Metcalfe, Tataru and Tohaneanu \cite{tataru1}, Aretakis \cite{A1}, \cite{A2}, \cite{aretakis3}, Schlue \cite{volker1}, Dafermos, Rodnianski and Shlapentokh--Rothman \cite{tessera} to mention a few. For a more complete set of references and for a nice introduction to the topic we refer to the lecture notes of Dafermos and Rodnianski \cite{lecturesMD}, \cite{tria}. 

In these works, results of decay and boundedness for various energies have been obtained. Such estimates pave the way for the study of nonlinear problems on black hole spacetimes. 

\subsection{Semilinear Waves On Minkowski And Black Hole Spacetimes}
There is a long history on the study of nonlinear waves on the Minkowski spacetime. We will focus here on the study of global solutions for small and smooth initial data. 

For equations of the form
$$ \Box \psi = O (|\psi|^2 , |\partial \psi |^2 )  $$
we can establish small data global well-posedness in dimensions $d \meg 4$ (our Minkowski spacetime being in this case $\rr^{1+d}$) as it was shown by Klainerman \cite{SK80} (see also the book of H\"{o}rmander \cite{hormander} and the lecture notes of Selberg \cite{selberg}). In dimension 3 (which is of interest here) the situation is different. John \cite{john} managed to prove that for the equation
$$ \Box \psi = (\partial_t\psi )^2 $$
any nontrivial $C^3$ solution with compactly supported initial data will blow up in finite time. On the positive side, he managed to prove a lower bound for the time of existence in \cite{john76}, \cite{john83}, \cite{johnklainerman}. 

This led to the requirement of imposing conditions on the nonlinearity in order to prove a global well-posedness statement. This was done by Christodoulou \cite{chcon} and Klainerman \cite{sergiunull} independently and by the use of different methods. They managed to prove global well-posedness for solutions that come from small data for equations 
$$ \Box \psi = F (\psi , \partial \psi)  ,$$
where $F$ satisfies the so-called null condition. An easy example is the following:
$$ \Box \psi = G(\psi ) (|\partial_t \psi |^2 - |\nabla \psi |^2 ) , $$
for some function $G$ that is bounded along with all its derivatives (we should note here that if $G = 1$ then this equation can be turned to a linear one as it was observed by Nirenberg -- see the lecture notes of Selberg \cite{selberg} for the actual transformation).

Christodoulou in \cite{chcon} used the conformal method by embedding the Minkowski space into the Einstein cylinder. Klainerman in \cite{sergiunull} on the other hand used the -- by now, celebrated -- vector field method. This relies on commutations of the equation with the vector fields that generate the symmetries of Minkowski spacetime.

The method of Klainerman turned out to be very influential and found applications in many different settings. Its geometric character made it adaptable to variable coefficient settings as well. 

For black hole spacetimes, nonlinear Klein-Gordon equations have been studied in terms of global well-posedness and asymptotic completeness in \cite{bachelotnicolas}, \cite{nicolas1}, \cite{nicolas2}. On the other hand Blue and Soffer \cite{blu0} and Blue and Sterbenz \cite{blu3} studied semilinear problems on the Schwarzschild black hole, while Dafermos and Rodnianski \cite{MDIR05i} studied spherically symmetric nonlinear waves on subextremal Reissner-Nordstr\"{o}m spacetimes. Moreover the vector field method has also been used by Luk \cite{luknullcondition} in the case of the Kerr black hole with small angular momentum $|a| \ll M$. In all these works the problem of global well-posedness for small data is considered through the study of energy decay estimates.

More recently, the Strauss conjecture on Kerr black holes with small angular momentum was resolved in \cite{lindmet} (after being combined with the ill-posedness results of \cite{cataniageorgiev}), i.e. small data global well-posedness was proven for the equation:
\begin{equation}\label{strauskerr}
\Box_{g_K} \psi = |\psi |^p
\end{equation}
for $p > 1 + \sqrt{2}$ and for $g_K$ the Kerr metric with small angular momentum.

We  also mention here the work of Shiwu Yang \cite{shiwu} that combines Klainerman's method with the method for proving decay for linear waves of Dafermos and Rodnianski \cite{newmethod}. Yang's proof can be appropriately adapted to give an alternative proof for Luk's result. This is something that we heavily exploit in the present work.

It should also be noted that Yang's method is rather robust and has the potential for wide applications. The main advantage of using the hierarchy of $r$-weighted estimates of Dafermos and Rodnianski is that on the nonlinear level one doesn't have to commute with the vector field $t \partial_r + r \partial_r$ (or any other vector fields growing in $t$), so no stationarity assumption on the metric is needed. Yang exploits this fact to give the first small data global well-posedness proof for a nonlinear wave equation satisfying the null condition in time dependent inhomogeneous media. Later in \cite{shiwuq} he extended this result to the quasilinear setting.

Finally we highlight the fact that a special case of a nonlinear wave with a nonlinearity that satisfies the null condition is the Wave Maps equation. For a nice review on the topic we refer to the survey articles of Tataru \cite{tatarubulletin} and Rodnianski \cite{rodnianskiwm}. This equation is related to the Einstein equation under the assumption of axisymmetry (see \cite{CB04}), and this is one of the main motivations for studying such nonlinear waves on black hole spacetimes. The bulk of this work on this equation though has been done for the case of the domain manifold being the Minkowski space. In the variable coefficient setting we mention the work of Lawrie \cite{lawrie}, and the more recent work of Gudapati \cite{gudapati} which studies the Einstein-Wave Maps coupled system.

\subsection{The Case Of Extremal Black Hole Spacetimes And Aretakis Instabilities}
\label{ernnn}
A black hole spacetime is called extremal if the surface gravity vanishes on the horizon. The surface gravity of a Killing null hypersurface is defined as follows: if $V$ is the normal Killing vector field for this hypersurface, then we have that
$$ \nabla_V V = \kappa V ,$$
and $\kappa$ is the surface gravity of the null hypersurface.

As an example, in the case of the extremal Reissner-Nordstr\"{o}m black we have that the Killing vector field $T = \partial_v$ in the Eddington--Filkenstein coordinates (see the subsequent section \ref{gernfm} for the definitions) is normal on $\mathcal{H}^{+}$ and satisfies:
$$ \nabla_T T = 0 .$$ 
The extremal case for this two parameter family of black hole spacetimes occurs when the mass is equal in value to the one of the absolute value of the electromagnetic charge $M = |e|$. We have the subextremal case when $|e| < M$ and the naked singularity case when $|e| > M$. Another example of an extremal black hole spacetime is the extremal Kerr black hole which occurs when $M = |a|$.

The wave equation equation on the extremal Reissner-Nordstr\"{o}m black hole spacetime with general data in the whole domain of outer communications (including the horizon $\mathcal{H}^{+}$) was studied in detail by Aretakis in \cite{A1} and \cite{A2} (the extremal Kerr was studied by the same author in \cite{aretakis3}). 

In the works \cite{A1}, \cite{A2} several boundedness and decay estimates that were known for the subextremal case were established as well. In particular he obtained decay for the degenerate energy, boundedness for the non-degenerate energy, Morawetz estimates with degeneracy on the horizon and on the photon sphere (note that the degeneracy on the photon sphere holds only in the non-spherically symmetric case), integrated local energy decay estimates for the degenerate energy and 2nd derivative $L^2$ estimates away from the horizon (these are essentially elliptic estimates). One of the major obstacles for obtaining estimates for the non-degenerate energy is the degeneration of the red-shift effect that is caused by the vanishing of surface gravity on the horizon. For showing the aforementioned estimates Aretakis introduced various novel currents that capture the dispersive properties of the wave equation on such backgrounds.

So on the opposite direction, he also established a number of instabilities with no analogues in the subextremal case. He observed first that the quantity
$$ Y^{l+1} \psi + \sum_{i=0}^l \beta_i Y^i \psi ,$$
for constants $\beta_i$ and for $Y=\partial_r$ where $(v,r)$ are the Eddington---Finkelstein coordinates (see Section \ref{gernfm} for the relevant definitions), is conserved on the horizon for a linear solution that is supported on the fixed angular frequency $l$. In the case of the $0$-th angular frequency (the spherically symmetric part of the wave) that we are interested in, the conserved quantity has the form
$$Y \psi + \frac{1}{M} \psi .$$

Then he proved that we can have an integrated local energy decay estimate for the non-degenerate energy only if we assume that our solution has angular frequency support away from 0. The same result was obtained for 2nd derivative $L^2$ estimates where the area of integration includes the horizon. Moreover he showed that 2nd and higher derivatives in $r$ blow up asymptotically on the horizon. Later, in \cite{janpaper}, Sbierski showed that the integrated local energy decay estimate has to fail for the non-degenerate energy of a spherically symmetric linear wave (it is still open if such an estimate can hold after losing derivatives, although this is expected to be false as well).

These instabilities were later extended by Aretakis again in more general geometric settings, see \cite{aretakis4} and \cite{aretakisglue}, and by Lucietti and Reall \cite{hj2012} for other equations. Further work on extremal black holes was done by Lucietti et al \cite{hm2012}, by Reall et al \cite{harvey2013}, by Bizon and Friedrich \cite{bizon2012}, by Dain and Dotti \cite{dd2012}, and recently by Hollands and Ishibashi \cite{hollands}. 

In the nonlinear setting in \cite{aretakis2013}, Aretakis showed that the conservation laws on the horizon can produce finite time blow up for smooth data that can be arbitrarily small. The nonlinearities that he's working with can have the form (this is just an example, more general forms for the nonlinearity can be allowed)
$$ \psi^{2n} + (Y \psi )^{2n}  \mbox{ for $n \meg 1$} ,$$
in the case of the extremal Reissner-Nordstr\"{o}m spacetime for spherically symmetric nonlinear waves. Note that this result is in contrast to the Minkowskian setting where we know that for such nonlinearities with $n \meg 2$ we can establish easily global well-posedness for small data.

In the present work we show that by imposing the null condition on the nonlinearity the finite time blow up scenario doesn't occur. On the other hand, we establish all the observed instabilities of the linear case in a slightly different form.

\section{Geometry Of The Extremal Reissner--Nordstr\"{o}m Black-Hole Spacetime}\label{gern}
\subsection{Forms Of The Metric And Coordinates Systems}\label{gernfm}
The Reissner--Nordstr\"{o}m family of metrics is the unique spherically symmetric family of solutions to the Einstein--Maxwell equations (see the Appendix for their exact form). It was first introduced in \cite{R} and \cite{N}. It has the following form:
\begin{equation}\label{rn}
 g = -Ddt^2 + D^{-1} dr^2 + r^2 \gamma_{\mathbb{S}^2},
 \end{equation}
where $\gamma_{\mathbb{S}^2} $ is the standard metric on the 2-sphere and
$$ D = 1-\frac{2M}{r} + \frac{e^2}{r^2}, $$
for $M$ the mass, $e$ being the electromagnetic charge. In the case where $e=M$ we have the extreme Reissner--Nordstr\"{o}m solution with
$$ D = \left( 1 - \frac{M}{r} \right)^2 . $$
The coordinate singularity at $r=M$ can be removed by considering the tortoise coordinate $r^{*}$ given by:
$$ \dfrac{dr^{*}}{dr} = \frac{1}{D}, $$
which transforms the metric \eqref{rn} into the form:
\begin{equation}\label{trn}
g = -Ddt^2 + D ( dr^{*} )^2 + r^2 \gamma_{\mathbb{S}^2}.
\end{equation}
The spacetime can be extended beyond $r=M$ by considering the ingoing Eddington--Finkelstein $(v,r)$ coordinates, where $v = t + r^{*}$. The metric \eqref{rn} then takes the following form:
\begin{equation}\label{rnef}
g = -D dv^2 + 2 dvdr + r^2 \gamma_{\mathbb{S}^2}.
\end{equation}
These are the coordinates that we will use mostly throughout this paper. In these coordinates, the hypersurface at $r=M$ will be called the future event horizon $\mathcal{H}^{+}$, the area $r<M$ is the black hole region and the area $r \meg M$ is the domain of outer communications. On the hypersurface $r=2M$ there exists null geodesics that neither cross the future event horizon $\mathcal{H}^{+}$ nor end up at future null infinity $\mathcal{I}^{+}$. This is a trapping effect (that won't bother us though in this work) and the hypersurface at $r=M$ is called the photon sphere. We will work exclusively in the domain of outer of communications (which as we said, includes the future event horizon $\mathcal{H}^{+}$.

We use the following notation for the derivative vector fields $\partial_v$ and $\partial_r$:
\begin{equation}\label{trvf}
T = \partial_v , \quad Y = \partial_r .
\end{equation}

The Penrose diagram of this spacetime in the $(v,r)$ coordinate is as follows:
\begin{figure}[H]
\centering
\includegraphics[width=5cm]{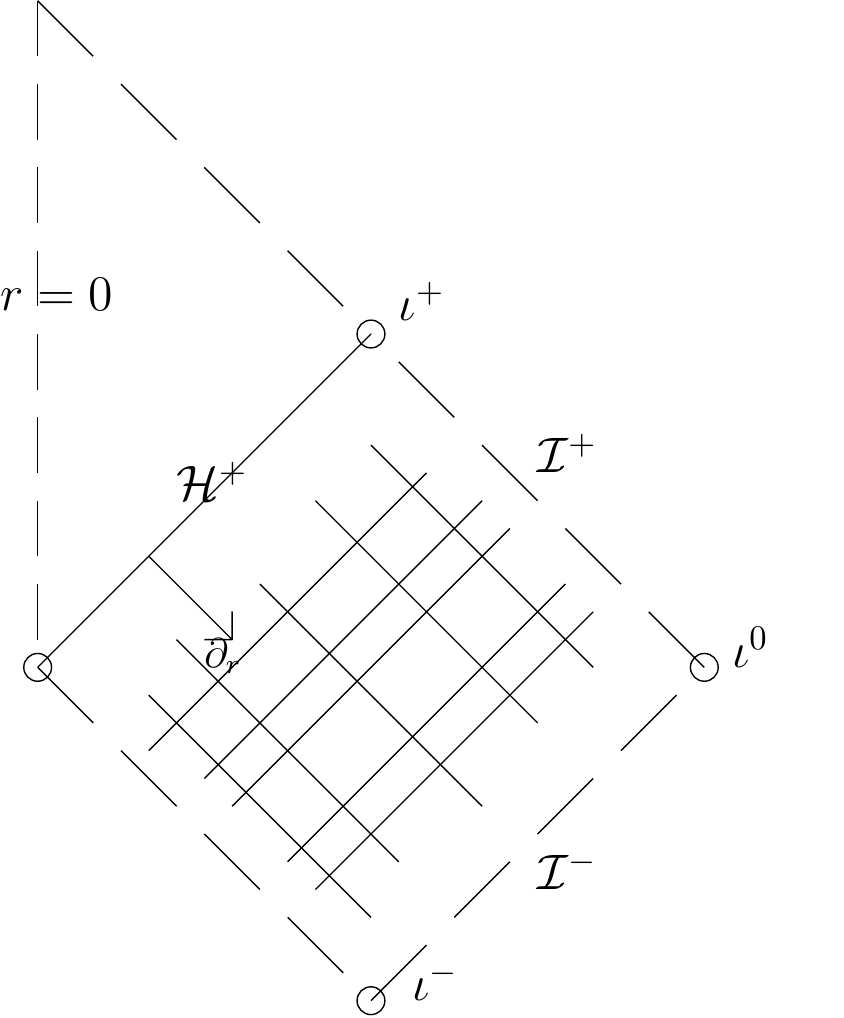}
\caption{The shaded region is the domain of outer communications, while its complement is the black hole region.}
\end{figure}

At $r=0$ we have the timelike curvature singularity inside the black hole region. At $\mathcal{I}^{+}$ we have future null infinity (and past null infinity at $\mathcal{I}^{-}$).

In these $(v,r)$ coordinates, and assuming that the coordinates on $\mathbb{S}^2$ are $(\theta, \phi)$, equation \eqref{nw} takes the following precise form (dropping the terms that are of $l$-th or higher order and expanding the quadratic nonlinearity):
\begin{equation}\label{nwef}
\Box_g \psi = A(\psi ) \left( D (\partial_r \psi)^2 + 2 \partial_v \psi \partial_r \psi + \frac{1}{r^2} (\partial_{\theta} \psi )^2 + \dfrac{1}{r^2 \sin^2 \phi} (\partial_{\phi} \psi)^2 \right) .
\end{equation}

Another useful set of coordinates are the null coordinates 
$$u = t-r^{*} , \quad v = t+r^{*} ,$$
which are singular on the horizon $r=M$ but turn out to be convenient for calculations away from it (and especially close to future null infinity as we will see later). In these coordinates the metric has the following form:
\begin{equation}\label{nrn}
g = -D dudv + r^2 \gamma_{\mathbb{S}^2}.
\end{equation}

\subsection{Foliations}\label{gernf}
 We define the set $N_{\tau} = \{ u = u_{\tau}, v \meg v_{\tau} \}$ (where $(u = t-r^{*},v = t+r^{*})$ the null coordinates of the previous section) for $u_{\tau} = u(p_0)$ for a point $p_0$ where $r(p_0) = R_0 > 2M$ (to be chosen later) and the set $\Sigma_{\tau}$ is the union of $N_{\tau}$ and a space-like hypersurface $S_{\tau} = \{ t^{*} = \tau \} \cap \{ r \mik R_0 \}$ for $t^{*} = v-r$ (for $(v,r)$ the ingoing Eddington--Finkelstein coordinates of Section \ref{gernfm}).  We also define the space-like hypersurfaces $\widetilde{\Sigma}_{\tau} =\{ t^{*} = \tau \}$. Moreover we define the corresponding spacetime regions in each one of the cases:
 $$ \widetilde{\mathcal{R}} (\tau_1 , \tau_2 ) = \cup_{\tau \in [\tau_1 , \tau_2]} \widetilde{\Sigma}_{\tau} , \quad \mathcal{R} (\tau_1 , \tau_2 ) = \cup_{\tau \in [\tau_1 , \tau_2]} \Sigma_{\tau} ,$$ $$ \so (\tau_1 , \tau_2 ) = \cup_{\tau \in [\tau_1 , \tau_2]} S_{\tau} , \quad \nnn (\tau_1 ,\tau_2 ) = \cup_{\tau \in [\tau_1 , \tau_2]} N_{\tau} , $$
for any $\tau_1$, $\tau_2$ with $\tau_1 < \tau_2$. 
 
Both of these foliations of the domain of outer communications can be considered through the flow of the vector field $T = \partial_v$ (for $(v,r)$ the Eddingont---Finkelstein coordinates of Section \ref{gernfm}). In the case of the spacelike foliation $\{ \widetilde{\Sigma}_{\tau} \}_{\tau \meg 0}$ we start with a spacelike hypersurface $\widetilde{\Sigma}_0$ that terminates at the spacelike infinity $\iota^0$ and we define each hypersurface $\widetilde{\Sigma}_{\tau}$ by the flow $\Phi_{\tau}^T$ as $\widetilde{\Sigma}_{\tau} = \Phi_{\tau}^T (\widetilde{\Sigma}_0)$. Similarly for the foliation $\{ \Sigma_{\tau}\}_{\tau \meg 0}$ we start with a spacelike-null hypersurface that coincides with $\widetilde{\Sigma}_0$ up to $R_0$ and each hypersurface $\Sigma_{\tau}$ is defined as $\Sigma_{\tau} = \Phi_{\tau}^T (\Sigma_0 )$.

In pictures (i.e. in the corresponding Penrose diagram) the spacelike foliation $\{ \widetilde{\Sigma}_{\tau} \}_{\tau \meg 0}$ looks as follows:
\begin{figure}[H]
\centering
\includegraphics[width=7cm]{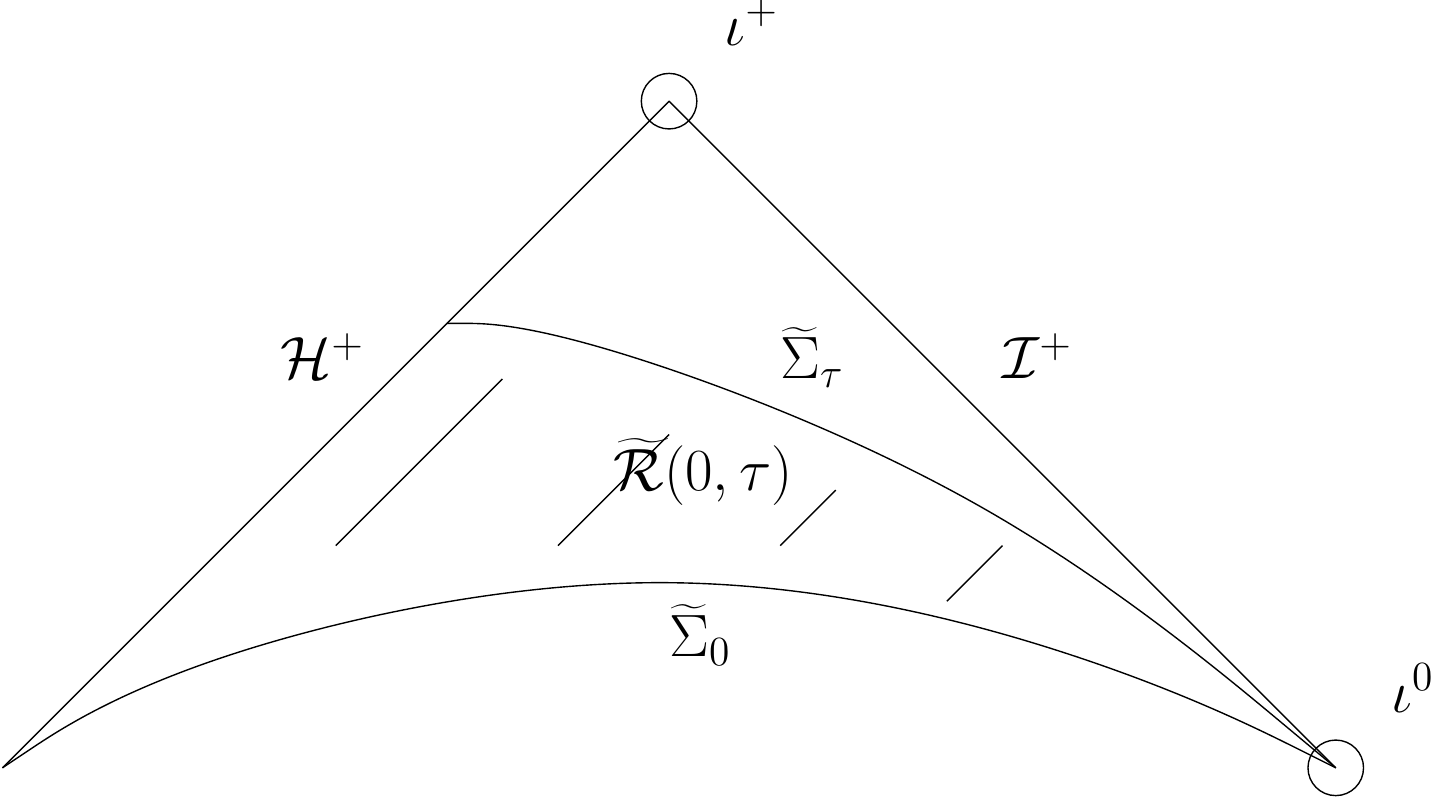}
\end{figure}
and the spacelike-null foliation $\{ \Sigma_{\tau} \}_{\tau \meg 0}$ as it is shown below:
\begin{figure}[H]
\centering
\includegraphics[width=7cm]{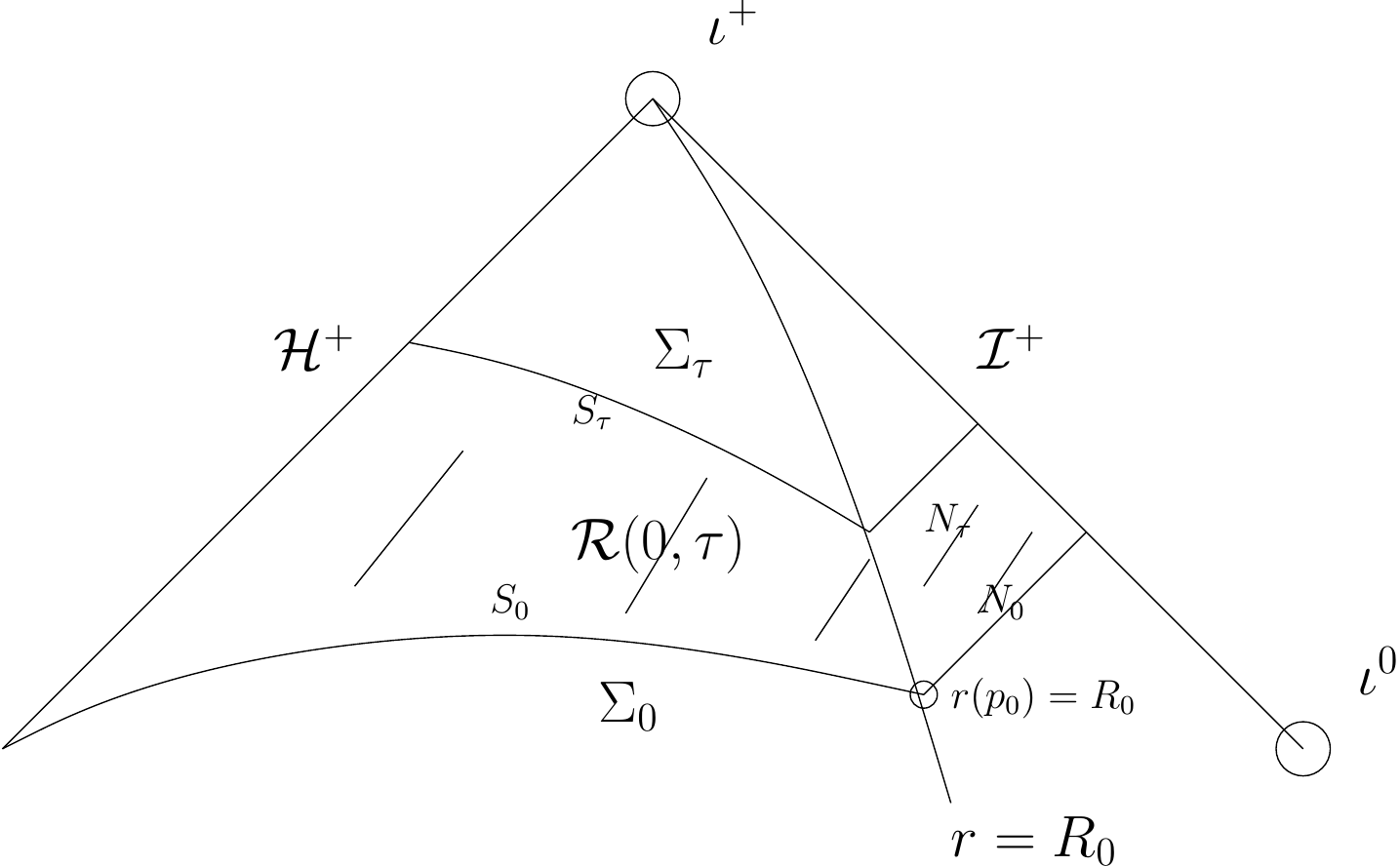}
\end{figure}

Finally we should mention that we can actually put the following coordinates on theses two foliation. For the spacelike one we can define the coordinate system $( \widetilde{r} , \omega)$ for $\widetilde{r} \in [M, \infty)$ and $\omega \in \mathbb{S}^2$, where $\widetilde{r}$ is defined through the following vector field definition:
\begin{equation}\label{rbar}
 \partial_{\widetilde{r}} = \bar{r} T + Y , 
 \end{equation}
for $\bar{r}$ a bounded function (which exists by the assumption that the initial spacelike hypersurface $\widetilde{\Sigma}_0$ has a future directed normal $n$ that satisfies $1 \lesssim - g(n,n) \lesssim 1$ and $1 \lesssim -g(n,T) \lesssim 1$). 

We have the following picture:
\begin{figure}[H]
\centering
\includegraphics[width=7cm]{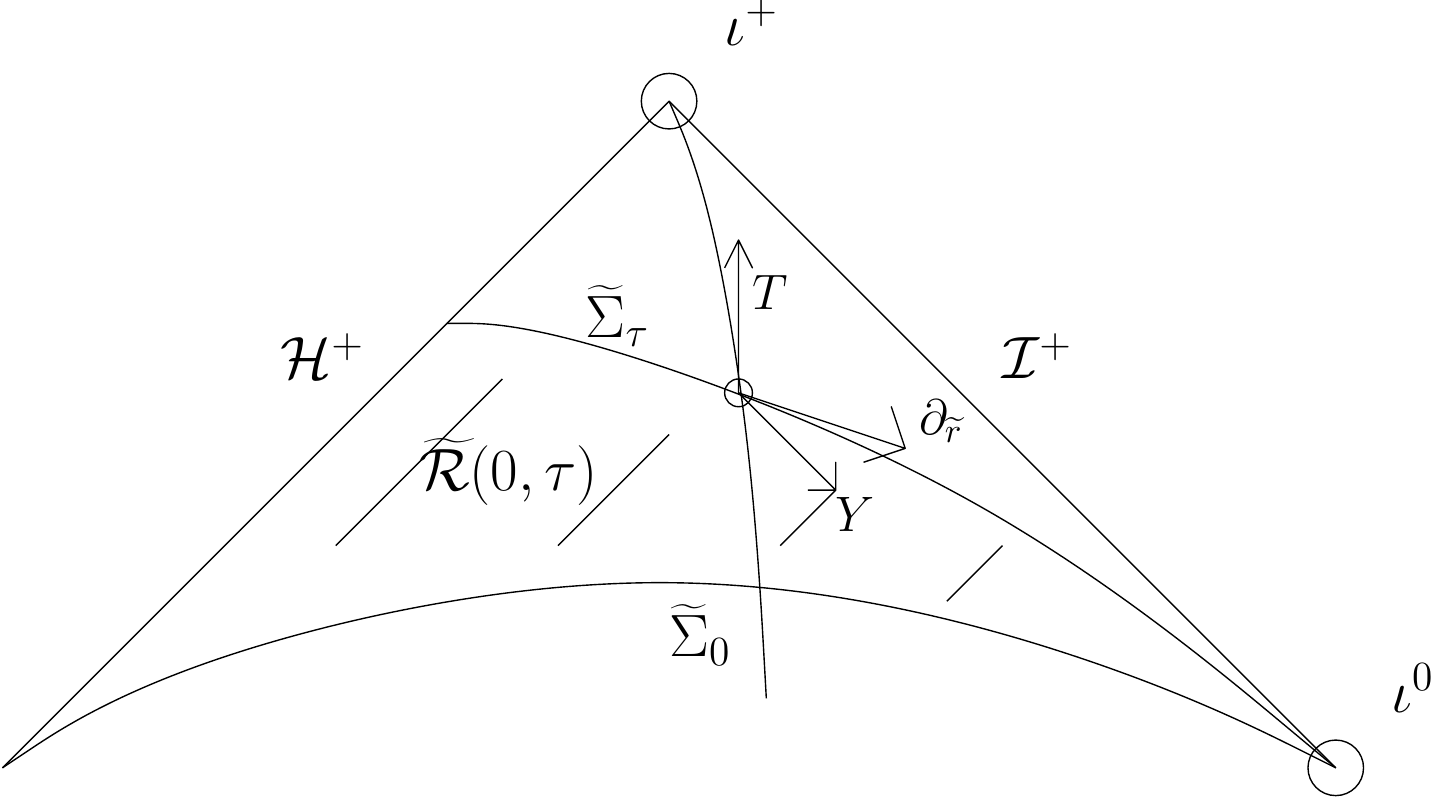}
\end{figure}

For the spacelike-null foliation $\{ \Sigma_{\tau} \}_{\tau \meg 0}$ we use the coordinate system $(\widetilde{r} , \omega)$ that was just specified on $\{ S_{\tau} \}_{\tau \meg 0}$ and the coordinate system $(v , \omega)$ (for $v = t+r^{*}$) on $\{ N_{\tau} \}_{\tau \meg 0}$.
\section{The Energy-Current Formalism}\label{ecf}
We will use throughout this paper the energy-current formalism which is another vector field based method. The Energy--Momentum tensor for the wave equation has the form:
$$ T_{\alpha \beta} = \partial_{\alpha} \psi \partial_{\beta} \psi - \frac{1}{2} g_{\alpha \beta} \partial^a \psi \partial_a \psi .$$
For some vector field $V$ we define the energy-current:
$$ J^V_{\mu} [\psi] = T_{\mu \nu} [\psi] V^{\nu} .$$
We compute the divergence of $J$ and we have:
$$ Div (J) = Div (T) V + T (\nabla V) \Rightarrow \nabla^{\mu} J^V_{\mu} [\psi] = Div (T[\psi]) V + T_{\mu\nu} [\psi] (\nabla V)^{\mu\nu} ,$$
for $(\nabla V)^{\mu \nu} = (\nabla^{\mu} V )^{\nu}$. We denote these last terms as:
$$ K^V [\psi ] = T_{\mu \nu} [\psi] (\nabla^{\mu} V )^{\nu}$$
and (since $Div(T[\psi]) = \Box_g \psi d\psi$):
$$\mathcal{E}^V [\psi ] = Div(T) V = \Box_g \psi \cdot V\psi .$$
We note that $(\nabla^{\mu} V )^{\nu} := \pi_V^{\mu \nu} = (\mathcal{L}_V g )^{\mu \nu}$ is the deformation tensor which is 0 if $V$ is a Killing vector field.

Finally we record the following computations for the currents $J^T$, $J^n$ on our different foliations with a spherically symmetric $\psi$:
\begin{equation}\label{current1}
 J^T_{\mu} [\psi ] n^{\mu}_{\sis} \approx (T\psi )^2 + D (Y\psi )^2 ,
 \end{equation}
 \begin{equation}\label{current2}
J^{n_{\sis}} [\psi ] n^{\mu}_{n_{\sis}} \approx (T\psi )^2 + (Y\psi )^2 ,
\end{equation}
 \begin{equation}\label{current3}
J^T_{\mu} [\psi ] n^{\mu}_{N} \approx J^{n_N}_{\mu} [\psi ] n^{\mu}_{N} \approx (\partial_v \psi )^2 .
\end{equation}
Estimate \eqref{current1} tells us that the current $J^T$ gives us an energy that degenerates on the horizon, while estimate \eqref{current2} tells us that the currect $J^n$ controls the $\dot{H}^1$ norm on the spacelike foliation $\{ \sis_{\tau} \}_{\tau \meg 0}$.

On the other hand, estimate \eqref{current3} shows that we can't control the $\dot{H}^1$ on the null part of the spacelike-null foliation$\{ \si_{\tau} \}_{\tau \meg 0}$ by any of the aforementioned currents.
\section{Notation}\label{Not}
We call the right hand side of \eqref{nw} by $F$. The metric $g$ is always the extremal Reissner-Nordstr\"{o}m metric, and $R_0$ is a constant related to the spacelike-null foliation $\{ \si_{\tau} \}_{\tau \meg 0 }$ which is chosen according to the support of the initial data. We define $T = \partial_v$, $Y = \partial_r$ for the Eddington---Finkelstein coordinates $(v,r)$ of the Section \ref{gern}, and each time that is used, $d\mi_g$ is the induced volume form, where every restriction of it will be indicated by a subscript (e.g. for an integral over a hypersurface $\si_{\tau}$ we will use the notation $d\mi_{g_{\si}}$).

We define as well the following initial energy:
\begin{equation}\label{inen1}
E_0 = \int_{S_0 } (Y \psi_0 )^2  + (T \psi_0 )^2  + |\slashed{\nabla} \psi_0 |^2  d\mi_{g_S} + \int_{N_0} (\partial_v \psi_0 )^2 + |\slashed{\nabla} \psi_0 |^2  d\mi_{g_N} + 
\end{equation}
$$ + \int_{\Sigma_0 } J^{n_{\Sigma_0}}_{\mu} [ \psi_1 ] n^{\mu} d\mi_{g_{\si}} ,
$$
where on the null part $N_0$ the null derivative $\partial_v$ is used (everywhere else we just use the Eddington--Finkelstein coordinates $(v,r)$).

We also note that we will use the symbols $\lesssim$ and $\gtrsim$ to denote the relations $\mik C \cdot$ and $\meg C \cdot$ for some constant $C > 0$ that is not related to any of the quantities involved in the inequality that is examined. Throughout the text a constant $\alpha$ will be fixed for our bootstrap assumptions. The constants in these assumptions will depend on $\alpha$ (and they actually become degenerate if $\alpha \rightarrow 0$) but this won't be always written down explicitly. Moreover all constants depend on the mass $M$ of the spacetime and the size of the support of the initial data. Finally we point out that the letter $C$ will be used to denote some constant several times, although this won't always be the same. 

\section{The Main Results}\label{Main}
We state here the main results of this paper.

First, we state the global well-posedness result.

\begin{thm}[\textbf{Global Well-Posedness}]\label{mainthm}
There exists an $\ee_0 >0$ such that if $0 \mik \ee \mik \ee_0$ then for all spherically symmetric compactly supported initial data 
$$ \psi [0] = (\ee \psi_0 , \ee \psi_1 ) $$
on a spacelike hypersurface $\widetilde{\Sigma}_0$ that terminates at spacelike infinity $\iota^0$ (as defined in Section \ref{Not}), satisfying
$$ \| \ee \psi_0 \|_{H^s (\widetilde{\Sigma}_0 )} \mik \sqrt{E_0} \ee , \quad \| \ee \psi_1 \|_{ H^{s-1} (\widetilde{\Sigma}_0 )} \mik \sqrt{{E}_0} \ee, $$
for some $s > 5/2$, $s\in \mathbb{N}$, there exists a unique globally defined spherically symmetric solution of \eqref{nw} with $l$ big enough (see Section \ref{othernon}) in $H^s$ in the domain of outer communications $\mathcal{M}$ up to and including the future event horizon of the extremal Reissner--Nordstr\"{o}m spacetime with mass and electromangetic charge equal to $M$.

Moreover the solution satisfies the following pointwise estimates:

1) \textbf{(Quantitative Decay for $\psi$)}. Let $\{ \Sigma_{\tau} \}_{\tau \meg 0}$ be the spacelike-null foliation as defined in Section \ref{gernf} and  let $\aaa>0$. Then there exists $\ee_{0}$ depending on $\aaa$ such that 
$$ \| \psi \|_{L^{\infty} (\Sigma_{\tau} )} \lesssim \dfrac{C_{\alpha}\sqrt{E_0} \ee}{(1+ \tau )^{3/5 - \alpha} }.$$
2) \textbf{(Uniform Boundedness for First-Order Derivatives).} Let $T,Y$ be the vector fields defined in Section \ref{gernfm}. Then 
$$ \| T \psi \|_{L^{\infty} (\mathcal{M} )} \lesssim \sqrt{E_0}  \ee , \quad \| Y\psi \|_{L^{\infty} (\mathcal{M} )} \lesssim \sqrt{E_0}  \ee .$$

\end{thm}
The next Theorem establishes non-decay and blow-up results along the event horizon for higher order derivatives of the solutions. 

\begin{thm}[\textbf{Asymptotic Instabilities}]\label{asb}
For a globally defined spherically symmetric solution of \eqref{nw} $\psi$ that arises from small initial data as in Theorem \ref{mainthm} the following instabilities on the future event horizon $\mathcal{H}^{+}$ hold:

a) \textbf{(Almost Conservation Law).}  The quantity 
$$ H(v) = Y \psi (v,r=M)  + \frac{1}{M} \psi (v,r=M)  $$ is conserved 
in the sense that:
$$ \left| H(\tau) - H(0) \right| \mik O (\ee^2) \mbox{ for all $\tau\geq 0$}. $$
Here  $(v,r)$ are the Eddington---Finkelstein coordinates.

b) \textbf{(Non-Decay for $Y\psi$).} The translation-invariant transversal to $\mathcal{H}^{+}$ derivative $Y\psi$ does not decay along  on the event horizon $\mathcal{H}^{+}$. 

c) \textbf{(Asymptotic Blow-up For Higher Derivatives).}  If the initial data satisfy the additional positivity assumption
$$ \psi_0 (M) > 0 , \quad Y \psi (0 , M ) > 0 $$ then 
$$ | Y^k \psi (\tau , M )| \xrightarrow{\tau \rightarrow \infty} \infty \mbox{ for all $k \meg 2$}, $$
along the event horizon $\mathcal{H}^{+}$. 
\end{thm}
We prove our global well-posedness statement using the vector field method through the machinery of Dafermos--Rodnianski (see \cite{lecturesMD} and \cite{newmethod}) as this was demonstrated in the work of S. Yang \cite{shiwu}, and the method of characteristics as this was used in the proof of Price's law for a scalar field coupled to the Einstein--Maxwell equations by Dafermos and Rodnianski \cite{MDIR05}. We particularly exploit the novel method of Yang in order to close the bootstrap assumption on the null part of the spacelike-null foliation. One of the major obstacles in our situation is the fact that we have to work with the degenerate energy 
\begin{equation}\label{degenerateenergy}
\int_{\Sigma_{\tau}} J^T_{\mu} [\psi ] n^{\mu} d\mi_{g_{\si}}, 
\end{equation}
due to the horizon instabilities for extremal black holes that were observed in \cite{A1} and \cite{A2}. The assumption of spherical symmetry (that allows us to ignore the trapping effect of the photon sphere) is a rather technical one but the instabilities for the spherically symmetric modes (that are not actually present in higher angular frequencies) pose several difficulties in proving the necessary a priori estimates. In this sense, this work is a significant first step for dealing with the general problem, that will be addressed in future work -- see the upcoming \cite{rnnonlin2}.

\begin{rem}
In the special case of spherically symmetric wave maps, due to the subcriticality of the equation, the fact that we work in a region where $r \meg M > 0$, and the fact that the degenerate energy is conserved, the global well-posedness result is immediate. Our contribution in this situation is the precise description of the asymptotic behaviour of such waves.
\end{rem}

We proceed as follows: 

(i) In Section \ref{eeh} we record and give proofs of all the energy inequalities that we will use.

(ii) In Sections \ref{tlwpt} and \ref{tlwpcc} we prove a standard local well-posedness theorem for solutions of \eqref{nw} (even non-spherically symmetric ones, and without the smallness restriction on the initial data) with respect to the spacelike foliations $\{\widetilde{\Sigma}_{\tau} \}_{\tau \meg 0}$ and then we formulate a continuation criterion. Such results are considered to be classical (analogues in Minkowski space can be found in many standard references for nonlinear wave equations, we just mention here the lecture notes of Selberg \cite{selberg} and the book of H\"{o}rmander \cite{hormander} that contains analogues for variable coefficient settings as well) and the proofs that we give here are adaptations to the specific black hole spacetime that we work with.

(iii) Having established a continuation criterion, we set out to verify it in the Sections \ref{apede} and \ref{apdbe}. Initially in Section \ref{apede} we prove decay for the degenerate energy \eqref{degenerateenergy} of $\psi$ assuming decay for the nonlinearity $F$. Then in Section \ref{apdbe}, assuming again the same decay estimates for the nonlinear term, we establish pointwise decay for $\psi$ and pointwise boundedness for $T \psi$ and $Y\psi$. The decay estimates for $\psi$ (everywhere in the domain of outer communications) are generalizations of the estimates derived in the linear case (that were established again in the work of Aretakis \cite{A1}, \cite{A2}). On the other hand the proofs for the boundedness of $Y \psi$ and for the boundedness of $T\psi$ are based on a novel technique which makes use of a subtle combination of the bootstrap arguments, $L^1$ and $L^2$ estimates and the method of characteristics. The same method gives us improved decay for $\psi$ with respect to $r$. Specifically we have managed to almost recover the bounds for $r\psi$ and $r^2 \partial_v \psi$ that were known in the subextremal case (but not of course a bound for $r \partial_u \psi$). It should be noted that we have managed to close all the estimates by using a bootstrap assumption only on $F$ and without proving decay for $T\psi$, which is a difference with the linear case and the subextremal situation. This is a consequence of the fact that even a commutation with $T$ for our equation is problematic (as opposed to the subextremal case) since for the time being we can't close a bootstrap argument for $TF$ on the horizon.

(iv) In Section \ref{bre} we improve the bootstrap assumptions (that were present in most of our estimates up to this point) closing thus all estimates. It should be noted though that this does not automatically imply global well-posedness. Indeed the estimates are in terms of the spacelike-null foliation $\{ \Sigma_{\tau} \}_{\tau \meg 0}$ for which we don't have a local well-posedness theorem. This issue is addressed in Section \ref{tgwpt} through the use of the standard local theory of Section \ref{tlwpt}, of all the estimates obtained up to this point, and the assumption of compact support for our initial data.

In the two Sections that follow we examine the asymptotic behaviour of our solution. First in Section \ref{acwothh} we prove an almost conservation law on the horizon $\mathcal{H}^{+}$ for the quantity $Y\psi + \frac{1}{M} \psi$ (a quantity which is actually conserved in the linear case as it was already mentioned in the Introduction). The proof is based on a bootstrap argument and integration by parts. Second, and by working with similar techniques, the almost conservation law allows us to prove asymptotic blow up for derivatives of order 2 and higher with respect to $r$  (i.e. for the quantities $Y^k \psi$ for $k\meg 2$) along the horizon. The behaviour of the solution obtained in Sections \ref{acwothh} and \ref{asbufhd} is reminiscent of an asymptotic version of a shock.

We note that all the above computations take place for the nonlinearity
$$ A(\psi ) g^{\alpha \beta} \partial_{\alpha} \psi \partial_{\beta} \psi ,$$
since the other cases that involve all the terms present in \eqref{nw} are either similar or easier. We explain the necessary modifications needed for our arguments to apply to more general nonlinearities in the last Section \ref{othernon}.

\begin{rem}
We should note here that the assumption of compact support on the initial data is probably removable in light of the recent works \cite{pin1}, \cite{pin2}, \cite{pin3} and \cite{shiwul} (the last one treating more general nonlinearities than the others and also being closer in spirit to the techniques that we employ here).  
\end{rem}

\section{Energy Estimates}\label{eeh}
We state here inhomogeneous versions of estimates found in \cite{A1} and \cite{A2}. They apply to the more general equation
\begin{equation}\label{nweff}
\Box_g \psi = F ,
\end{equation}
for $\psi$ being spherically symmetric.

These estimates will be used throughout the rest of the paper. 

\subsection{An Auxiliary Inequality}
First we record Hardy's inequality in the context of our black-hole spacetime. This inequality holds for general functions and not just for solutions of wave equations, it turns out to be rather useful.
\begin{prop}[\textbf{Hardy's Inequality}]\label{hardyineq}
For any sufficiently smooth function $f$ and any $\tau$ we have:
\begin{equation}\label{hardy}
\int_{\mathcal{S}_{\tau}} \frac{1}{r^2} f^2 d\mi_{g_{\mathcal{S}}} \lesssim \int_{\mathcal{S}_{\tau}} J^T_{\mu} [f] n^{\mu} d \mi_{g_{\mathcal{S}}} ,
\end{equation} 
where $\mathcal{S}_{\tau}$ is either $\Sigma_{\tau}$ or $\widetilde{\Sigma}_{\tau}$.
\end{prop}
For the proof of this inequality in the case of the spacelike foliation \{$\widetilde{\Sigma}_{\tau} \}_{\tau \meg 0}$ see section 6, Proposition 6.0.2 of \cite{A1}.

The case of the spacelike-null foliation is similar, the only difference is that we integrate with respect to the null direction $v = t+r^{*}$ in the region $r\meg R_0$ now.

To be precise, the actual inequalities that we get are the following:
$$ \int_{\widetilde{\Sigma}_{\tau}} \frac{1}{r^2} f^2 d\mi_{g_{\sis}} \lesssim \int_{\widetilde{\Sigma}_{\tau}} D \left[ (T f )^2 + (Y f )^2 \right] d \mi_{g_{\widetilde{\Sigma}}} ,$$ 
and 
$$ \int_{\Sigma_{\tau}} \frac{1}{r^2} f^2 d\mi_{g_{\si}} \lesssim \int_{\Sigma_{\tau} \cap \{ r \mik R_0 \}} D \left[ (T f )^2 + (Y f )^2 \right] d \mi_{g_S} + \int_{\Sigma_{\tau} \cap \{ r \meg R_0 \}} (\partial_v f )^2 d \mi_{g_N} .$$

\subsection{Morawetz Estimates With Degeneracy at $\mathcal{H}^{+}$}
The inequalities stated below give us bounds for "space-time" integrals for degenerate energies (with the degeneracy taking place on the horizon $\mathcal{H}^{+}$) under the assumption of spherical symmetry. These estimates will be useful tools in the proofs of the integrated local energy decay results that will follow.
\begin{prop}[\textbf{Morawetz Estimate On Spacelike Hypersurfaces}]\label{degx}
Let $\psi$ be a spherically symmetric solution of \eqref{nweff}. We have the inequality:
\begin{equation}\label{degix}
\int_{\tau_1}^{\tau_2} \int_{\widetilde{\Sigma}_{\tau'}} \left( \dfrac{( T \psi )^2 }{r^{1+\eta}} +  D^2 \dfrac{(Y \psi )^2}{r^{1+\eta}} \right) d\mu_{g_{\rrrr}} \lesssim_{R_0,\eta} \int_{\widetilde{\Sigma}_{\tau_1}} J^T_{\mu} [\psi] n^{\mu} d\mu_{g_{\sis}} + 
\end{equation}
$$ + \int_{\tau_1}^{\tau_2} \int_{S_{\tau'}} |F|^2 d\mu_{g_{\so}} + \int_{\tau_1}^{\tau_2} \int_{\sis_{\tau'} \cap \{ r \meg R_0\}} r^{1+\eta} |F|^2 d\mu_{g_{\rrrr}} ,$$ 
for any $\tau_1$, $\tau_2$ with $\tau_1 < \tau_2$, and for any $\eta > 0$.
\end{prop} 
\begin{proof}
We apply Stokes' theorem to $J^{X^0}$ for $X^0 = f^0 \partial_{r^{*}}$ with $f^0 = - \dfrac{1}{r^3}$ and we have:
$$  \int_{\widetilde{\Sigma}_{\tau_2}} J^{X^0}_{\mu} [\psi ] n^{\mu} d\mi_{g_{\sis}} + \int_{\mathcal{H}^{+}} J^{X^0}_{\mu} [\psi ] n^{\mu} d\mi_{g_{\ho}} + $$ $$ + \int_{\widetilde{\mathcal{R}} (\tau_1 , \tau_2 )} K^{X^0} [\psi ] d \mi_{g_{\rrrr}} + \int_{\widetilde{\mathcal{R}} (\tau_1 , \tau_2 )} \mathcal{E}^{X^0} [\psi ] d\mi_{g_{\rrrr}} = \int_{\widetilde{\Sigma}_{\tau_1}} J^{X^0}_{\mu} [\psi ] n^{\mu} d\mi_{g_{\sis}} .$$
  
By Proposition 9.3.1 of \cite{A1} we have that:
$$ K^{X^0} [\psi] = \frac{1}{r^4} (\partial_t \psi )^2 + \frac{5}{r^4} (\partial_{r^{*}} \psi )^2 ,$$
for $\psi$ a spherically symmetric solution of \eqref{nwef} as assumed. 

This gives us the following:
\begin{equation}\label{degix11}
\int_{\tau_1}^{\tau_2} \int_{\widetilde{\Sigma}_{\tau'}} \left( \dfrac{( \partial_t \psi )^2 }{r^4} +  \dfrac{5(\partial_{r^{*}} \psi )^2}{r^4} \right) d\mi_{g_{\rrrr}} \lesssim \left| \int_{\widetilde{\Sigma}_{\tau_2}} J^{X^0}_{\mu} [\psi] n^{\mu} d\mi_{g_{\sis}} \right| + 
\end{equation}
$$ +   \left| \int_{\ho} J^{X^0}_{\mu} [\psi] n^{\mu} d\mi_{g_{\ho}} \right| + \left| \int_{\widetilde{\Sigma}_{\tau_1}} J^{X^0}_{\mu} [\psi] n^{\mu} d\mi_{g_{\sis}} \right| + $$ $$ +\int_{\tau_1}^{\tau_2} \int_{\widetilde{\Sigma}_{\tau'}} \frac{1}{r^3} \left| F \cdot (T \psi + D\cdot Y \psi ) \right | d\mi_{g_{\rrrr}} .$$

Proposition 9.4.1 of \cite{A1} tells us that if $\mathcal{S}$ is an $SO(3)$-invariant spacelike or null hypersurface, then we have the bound:
\begin{equation}\label{941a}
\left| \int_{\mathcal{S}} J^{X^0}_{\mu} [\psi ] n^{\mu} d\mi_{g_{\mathcal{S}}} \right| \lesssim \int_{\mathcal{S}} J^T_{\mu} [\psi ] n^{\mu} d\mi_{g_{\mathcal{S}}} .
\end{equation}
Then by Stokes' Theorem for $J^T$ and \eqref{941a}, estimate \eqref{degix11} becomes:
$$ \int_{\tau_1}^{\tau_2} \int_{\widetilde{\Sigma}_{\tau'}} \left( \dfrac{( \partial_t \psi )^2 }{r^4} +  \dfrac{5(\partial_{r^{*}} \psi )^2}{r^4} \right) d\mi_{g_{\rrrr}} \lesssim $$ $$ \lesssim \int_{\si_{\tau_1}} J^T_{\mu} [\psi ] n^{\mu} d\mi_{g_{\si}} + \int_{\tau_1}^{\tau_2}\int_{\widetilde{\Sigma}_{\tau'}} \frac{1}{r^3} \left| F \cdot (T \psi + D\cdot Y \psi ) \right | d\mi_{g_{\rrrr}} + \int_{\tau_1}^{\tau_2}\int_{\sis_{\tau'}} |F \cdot T\psi  | d\mi_{g_{\sis}} .$$
Before applying Cauchy--Schwarz with the correct weights in $r$ and moving the appropriate terms on the left-hand side, we optimize the weights.

This can be done as it is noted in Remark 9.1 of \cite{A1}. In the region $\{r \meg R\}$ for a sufficiently large $R$ we consider the currents
$$ J^{\widetilde{X}} = J^{X^{f^1 , 1}} + J^{X^{f^2}} ,$$
where $f^1 = 1 - \dfrac{1}{r^{\eta}}$, $f^2 = \dfrac{\eta}{2+\eta} \dfrac{1}{r^{\eta}}$ for some $\eta > 0$, and 
$$J^{X^{f^1},1}_{\mu} = J^{X^{f^1}}_{\mu} + 2G \psi (\nabla_{\mu} \psi ) - (\nabla_{\mu} G ) \psi^2 ,$$
for $G$ defined as in (9.2) of \cite{A1}:
$$ G = \dfrac{ (f^1 )'}{4} + \dfrac{f^1 \cdot D}{2r} , $$
where $' = \dfrac{d}{dr^{*}}$.

Proposition 9.4.2 of \cite{A1} tells us that the analogue of \eqref{941a} holds for the current $J^{\widetilde{X}}$. We also have the observation that for $r\meg R$:
$$ \nabla^{\mu} J^{\widetilde{X}}_{\mu} \gtrsim_{\eta} \left( r^{-1-\eta} (\partial_t \psi )^2 + r^{-1-\eta} (\partial_{r^{*}} \psi )^2 + r^{-3-\eta} \psi^2 + \dfrac{1}{r} \psi \cdot F + \mathcal{E}^{X^{f^1}} [\psi ] + \mathcal{E}^{X^{f^2}} [\psi ]\right) .$$
We now apply Stokes' Theorem to the current:
$$ \chi  J^{X^0} + (1-\chi ) J^{\widetilde{X}} ,$$
for 
$$ \chi \in C^{\infty}_0 ( [M , \infty )) , \quad \chi (r) = 1 \mbox{ for $r\mik R$}, \quad \chi = 0 \mbox{ for $r\meg R+1$} .$$
Our previous observations and \eqref{degix11} give us that:
\begin{equation}\label{degix12}
 \int_{\tau_1}^{\tau_2} \int_{\widetilde{\Sigma}_{\tau'} \cap \{ r \mik R \}} \left( \dfrac{( T \psi )^2 }{r^4} +  D^2 \dfrac{(Y \psi )^2}{r^4} \right) d\mi_{g_{\rrrr}} + 
 \end{equation}
  $$ + \int_{\tau_1}^{\tau_2} \int_{\sis_{\tau'} \cap \{ r\meg R+1 \}} \left( \dfrac{( T \psi )^2 }{r^{1+\eta}} +  D^2 \dfrac{( Y \psi )^2}{r^{1+\eta}} \right) d\mi_{g_{\rrrr}} + \int_{\tau_1}^{\tau_2} \int_{\sis_{\tau'} \cap \{ r \meg R+1 \}} \dfrac{\psi^2}{r^{3+\eta}} d\mi_{g_{\rrrr}} \lesssim $$ $$ \lesssim \int_{\si_{\tau_1}} J^T_{\mu} [\psi ] n^{\mu} d\mi_{g_{\si}} + \int_{\tau_1}^{\tau_2}\int_{\widetilde{\Sigma}_{\tau'}} \frac{1}{r^3} \left| F \cdot (T \psi + D \cdot Y \psi ) \right | d\mi_{g_{\rrrr}} + \int_{\tau_1}^{\tau_2}\int_{\sis_{\tau'}} |F \cdot T\psi  | d\mi_{g_{\sis}} + $$ $$ + \left| \int_{\tau_1}^{\tau_2} \int_{\sis_{\tau'} \cap \{ r \meg R+1 \}} \dfrac{1}{r} \psi \cdot F d\mi_{g_{\rrrr}} \right| + \int_{\tau_1}^{\tau_2} \int_{\sis_{\tau'} \cap \{ r \meg R+1 \}} |F \cdot (X^{f^1} \psi + X^{f^2} \psi )| d\mi_{g_{\rrrr}} + $$ $$ +\int_{\tau_1}^{\tau_2} \int_{\sis_{\tau'} } |\nabla \chi | \cdot \left( J^{X^0}_{\mu} [\psi] n^{\mu} + J^{\widetilde{X}}_{\mu} [\psi] n^{\mu} \right) d\mi_{g_{\rrrr}} .$$
For the first term of the last line we apply Cauchy-Schwarz with weights $r^{1+\eta}$ which gives us that:
$$ \left| \int_{\tau_1}^{\tau_2} \int_{\sis_{\tau'} \cap \{ r \meg R+1 \}} \dfrac{1}{r} \psi \cdot F d\mi_{g_{\rrrr}} \right| \mik $$ $$ \mik \gamma \int_{\tau_1}^{\tau_2} \int_{\sis_{\tau'} \cap \{ r \meg R+1 \}} \dfrac{\psi^2}{r^{3+\eta}} \psi  d\mi_{g_{\rrrr}} + \dfrac{1}{\gamma} \int_{\tau_1}^{\tau_2} \int_{\sis_{\tau'} \cap \{ r \meg R+1 \}} r^{1+\eta} |F|^2 d\mi_{g_{\rrrr}} . $$
Choosing $\gamma$ small enough we can absorb the first term of the last line in the left hand side of \eqref{degix12}.

For the term 
$$ \int_{\tau_1}^{\tau_2} \int_{\sis_{\tau'} } |\nabla \chi | \cdot \left( J^{X^0}_{\mu} [\psi] n^{\mu} + J^{\widetilde{X}}_{\mu} [\psi] n^{\mu} \right) d\mi_{g_{\rrrr}} $$
we note that $supp ( |\nabla \chi | ) \subset \rrrr \cap \{ R \mik R+1 \}$ and then we use \eqref{degix11} to bound it. Finally we note that for any $\beta > 0$ we have the following:
$$ \int_{\tau_1}^{\tau_2}\int_{\widetilde{\Sigma}_{\tau'}} \frac{1}{r^3} \left| F \cdot (T \psi + D Y \psi ) \right | d\mi_{g_{\rrrr}} + \int_{\tau_1}^{\tau_2}\int_{\sis_{\tau'}} |F \cdot T\psi  | d\mi_{g_{\rrrr}} \mik $$ $$ \mik \dfrac{2}{\beta} \int_{\tau_1}^{\tau_2}\int_{\widetilde{\Sigma}_{\tau'}} r^{1+\eta} |F|^2 d\mi_{g_{\rrrr}} + 2\beta  \int_{\tau_1}^{\tau_2}\int_{\widetilde{\Sigma}_{\tau'}} \dfrac{ (T\psi)^2 + D^2 (Y\psi)^2 }{r^{1+\eta}} d\mi_{g_{\rrrr}}  .$$
The proof finishes by choosing $\beta$ small enough in order to move the last term above to the right hand side of \eqref{degix12} after noticing that in the region $\{ r\mik R \}$ all $r$-weights are equivalent.

\end{proof}

We point out that Proposition \ref{degx} has the same form on the spacelike-null hypersurfaces $\{ \Sigma_{\tau} \}_{\tau \meg 0}$.
\begin{prop}[\textbf{Morawetz Estimate On Spacelike-Null Hypersurfaces}]\label{degxsn}
Let $\psi$ be a spherically symmetric solution of \eqref{nweff}. We have the inequality:
\begin{equation}\label{degixsn}
\int_{\tau_1}^{\tau_2} \int_{\Sigma_{\tau'}} \left( \dfrac{( T \psi )^2 }{r^{1+\eta}} +  D^2 \dfrac{(Y \psi )^2}{r^{1+\eta}} \right) d\mi_{g_{\rrr}} \lesssim_{R_0,\eta} \int_{\Sigma_{\tau_1}} J^T_{\mu} [\psi] n^{\mu} d\mi_{g_{\si}}+ 
\end{equation}
$$ + \int_{\tau_1}^{\tau_2} \int_{S_{\tau'}} |F|^2 d\mi_{g_S} +  \int_{\tau_1}^{\tau_2} \int_{N_{\tau'}} r^{1+\eta} |F|^2 d\mi_{g_N} ,$$
for any $\tau_1$, $\tau_2$ with $\tau_1 < \tau_2$, and for any $\eta > 0$.
\end{prop} 
\begin{proof}
The proof is almost identical to that of Proposition \ref{degx} and won't be repeated in detail. The only difference is that for all the energy currents involved we must bound also the boundary terms on future null infinity. For this we use Propositions 9.4.1 and 9.4.2 of \cite{A1} and we have that if $\mathcal{Y}$ is any of the vector fields used in the proof of Proposition \ref{degx} then the following holds:
$$ \left| \int_{\mathcal{I}^{+}} J^{\mathcal{Y}}_{\mu} [\psi] n^{\mu} d\mi_{g_{\mathcal{I}^{+}}} \right| \lesssim \int_{\mathcal{I}^{+}} J^T_{\mu} [\psi] n^{\mu} d\mi_{g_{\mathcal{I}^{+}}} $$
The $T$-flux on $\mathcal{I}^{+}$ is positive since $T$ is timelike. Hence from Stokes' Theorem for $J^T$ we have 
$$ \int_{\mathcal{I}^{+}} J^T_{\mu} [\psi] n^{\mu} d\mi_{g_{\mathcal{I}^{+}}} \lesssim \int_{\Sigma_{\tau_1}} J^T_{\mu} [\psi] n^{\mu} d\mi_{g_{\si}} + \int_{\mathcal{R} (\tau_1 , \tau_2 )} |F \cdot T\psi | d\mi_{g_{\rrr}} .$$
Now the result follows by applying Cauchy-Scwharz to the last term of the right-hand side with weight $r^{-1-\eta}$, and by moving the term that involves $T\psi$ to the left-hand side.
\end{proof}

The version of estimate \eqref{degixsn} that we will actually use most frequently won't need a weight in $r$, since we will just localize in a compact region. It can be stated as follows for some spacetime region $\mathcal{G} \subset \widetilde{\mathcal{R}} (\tau_1 , \tau_2 ) \cap \{ r \mik R_0 \}$ and for $\psi$ being again a spherically symmetric solution of \eqref{nwef}:
\begin{equation}\label{degix1}
\int_{\mathcal{G}} \left(  T \psi )^2 +  D^2 (Y \psi )^2 \right) d\mi_{g_{\mathcal{G}}} \lesssim_{R_0,\eta} \int_{\Sigma_{\tau_1} \cap \{ r \mik R_0 \}} J^T_{\mu} [\psi] n^{\mu} d\mi_{g_{\si}} + 
\end{equation}
$$ + \int_{\tau_1}^{\tau_2} \int_{S_{\tau'}} |F|^2 d\mi_{g_{\so}} +  \int_{\tau_1}^{\tau_2} \int_{N_{\tau'}} r^{1+\eta} |F|^2 d\mi_{g_{\nnn}} ,  $$
for any $\eta > 0$.

Finally notice that we stated inequalities \eqref{degix}, \eqref{degixsn} and \eqref{degix1} with a constant depending on $R_0$. This convention that will be employed elsewhere later on comes from the trivial inequality:
$$ \int_{\tau_1}^{\tau_2} \int_{S_{\tau'}} r^{1+\eta} |F|^2 d\mi_{g_S} \lesssim_{R_0} \int_{\tau_1}^{\tau_2} \int_{S_{\tau'}} |F|^2 d\mi_{g_S} , $$
and is convenient for some estimates that will be used later in the article.
\subsection{Uniform Boundedness of the Degenerate Energy}
A combination of the previous results will give us a uniform bound for the degenerate energy (with the degeneracy taking place on the horizon as noted before) given by the current $J^T$. 

First we state an inequality for the degenerate energy coming from the current $J^T$ over the spacelike hypersurfaces $\{\widetilde{\Sigma}_{\tau} \}_{\tau \meg 0}$.
\begin{prop}[\textbf{Uniform Boundedness For The $T$-flux Over Spacelike Hypersurfaces}]\label{deg5}
Let $\psi$ be a spherically symmetric solution of \eqref{nweff}. Then we have that:
\begin{equation}\label{degi5}
\int_{\widetilde{\Sigma}_{\tau_2}} J^T_{\mu} [\psi ] n^{\mu} d\mi_{g_{\sis}} \lesssim_{R_0}  \int_{\widetilde{\Sigma}_{\tau_1} } J^T_{\mu} [\psi ] n^{\mu} d\mi_{g_{\sis}} + 
\end{equation}
$$ + \int_{\tau_1}^{\tau_2} \int_{S_{\tau'}} |F|^2 d\mi_{g_{\so}} + \int_{\tau_1}^{\tau_2} \int_{\sis_{\tau'} \cap \{ r \meg R_0\}} r^{1+\eta} |F|^2 d\mi_{g_{\sis}} ,$$ 
for any $\tau_1$, $\tau_2$ with $\tau_1 < \tau_2$, and any $\eta > 0$.
\end{prop}
\begin{proof}
Stokes' theorem for $J^T$ gives us the following:
$$ \int_{\sis_{\tau_2}} J^T_{\mu} [\psi ] n^{\mu} d\mi_{g_{\sis}} \lesssim \int_{\sis_{\tau_1}} J^T_{\mu} [\psi ] n^{\mu} d\mi_{g_{\sis}} +\int_{\rrrr (\tau_1 , \tau_2 )} |F \cdot T\psi | d\mi_{g_{\rrrr}} , $$
by the fact that $K^T [\psi ] = 0$ since $T$ is Killing, and because the integral $\int_{\ho} J^T_{\mu} [\psi ] n^{\mu} d \mi_{g_{\ho}}$ is positive since $T$ is causal on $\ho$.

We apply Cauchy-Schwarz with the weight $r^{1+\eta}$, for some $\eta > 0$, to the second term of the right-hand side, and by Proposition \ref{degx} we get that:
$$ \int_{\rrrr (\tau_1 , \tau_2 )} |F \cdot T\psi | d\mi_{g_{\rrrr}} \lesssim \int_{\sis_{\tau_1}} J^T_{\mu} [\psi ] n^{\mu} d\mi_{g_{\sis}} + $$ $$ + \int_{\rrrr (\tau_1 , \tau_2 )} r^{1+\eta} |F|^2 d\mi_{g_{\rrrr}} .$$
This finishes the proof.
\end{proof}
The same result holds on the spacelike-null foliation $\{ \si_{\tau}\}_{\tau \meg 0}$. 
\begin{prop}[\textbf{Uniform Boundedness For The $T$-flux Over Spacelike-Null Hypersurfaces}]\label{deg2}
Let $\psi$ be a spherically symmetric solution of \eqref{nweff}. Then we have that:
\begin{equation}\label{degi2}
\int_{\Sigma_{\tau_2}} J^T_{\mu} [\psi ] n^{\mu} d\mi_{g_{\si}} \lesssim_{R_0,\eta} \int_{\Sigma_{\tau_1}} J^T_{\mu} [\psi ] n^{\mu} d\mi_{g_{\si}}  + 
\end{equation}
$$ + \int_{\tau_1}^{\tau_2} \int_{S_{\tau'}} |F|^2 d\mi_{g_{\so}} +  \int_{\tau_1}^{\tau_2} \int_{N_{\tau'}} r^{1+\eta} |F|^2 d\mi_{g_{\nnn}} ,$$
for any $\tau_1$, $\tau_2$ with $\tau_1 < \tau_2$, and any $\eta > 0$.
\end{prop}
\begin{proof}
The proof is the same as that of Proposition \ref{deg5} by noticing additionally that the $T$-flux over $\mathcal{I}^{+}$ is positive, and by using Proposition \ref{degxsn} instead of \ref{degx}.
\end{proof}
\subsection{Non-Uniform Boundedness of the Non-Degenerate Energy}
We now state a ``time"-dependent estimate for the $\dot{H}^1$ norm of $\psi$ on the spacelike foliations $\{ \sis_{\tau} \}_{\tau \meg 0}$ which will be useful for the local theory for equation \eqref{nw}.

We should note here that such estimates hold also for more general vector fields, the proof boils down in the end to an application of Gr\"{o}nwall's inequality (an example can be given by the degenerate energy on the horizon given by $J^T$ -- note that this energy defines a norm as well and that local well-posedness can be proven with this too). Specifically we have the following general result:
 
\begin{prop}[\textbf{``Time"-Dependent Boundedness Estimate For The Non-Degenerate Energy}]\label{deg1}
For $\psi$ a solution of \eqref{nweff} we have the following bound for any $\tau_1$, $\tau_2$ with $\tau_1 < \tau_2$:
\begin{equation}\label{degi1}
\int_{\sis_{\tau_2}} J^{n_{\sis_{\tau_2}}}_{\mu} [\psi ] n^{\mu} d\mi_{g_{\sis_{\tau_2}}} \lesssim_{\tau_1 , \tau_2 } \int_{\sis_{\tau_1}} J^{n_{\sis_{\tau_1}}}_{\mu} [\psi ] n^{\mu} d\mi_{g_{\sis_{\tau_1}}} +\int_{\tau_1}^{\tau_2} \int_{\sis_{\tau'}} |F|^2 d\mi_{g_{\rrrr}} .
\end{equation}
\end{prop}
The proof is straight forward and applies also to any timelike vector field $V$.

\subsection{Uniform Boundedness of the Non-Degenerate Energy}
We will prove now a uniform estimate for the non-degenerate energy on the spacelike-null foliations $\{ \si_{\tau} \}_{\tau \meg 0}$. 
\begin{prop}[\textbf{Uniform Boundedness For The $\dot{H}^1$ Norm}]\label{deg4}
Let $\psi$ be a spherically symmetric solution of \eqref{nweff}. Then we have the following for any $\tau_1$, $\tau_2$ with $\tau_1 < \tau_2$ and any $\eta > 0$:
\begin{equation}\label{degi4}
\int_{\si_{\tau_2}} J^{n_{\si}}_{\mu} [\psi ] n^{\mu} d\mi_{g_{\si}} + \int_{\mathcal{A}} \left( (T\psi )^2 + \sqrt{D} (Y\psi )^2 \right) d\mi_{g_{\mathcal{A}}} \lesssim_{R_0,\eta} \int_{\si_{\tau_1}} J^{n_{\si}}_{\mu} [\psi ] n^{\mu} d\mi_{g_{\si}}  +
\end{equation}
$$ +  \int_{\tau_1}^{\tau_2} \int_{S_{\tau'}} |F|^2 d\mi_{g_{\so}} +  \int_{\tau_1}^{\tau_2} \int_{N_{\tau'}} r^{1+\eta} |F|^2 d\mi_{g_N} + \left( \int_{\tau_1}^{\tau_2} \left( \int_{S_{\tau'}} |F|^2 d\mi_{g_S} \right)^{1/2} d\tau' \right)^2 .$$
\end{prop}
where $\mathcal{A} = \rrr (\tau_1 , \tau_2 ) \cap \{ M \mik r \mik A \}$ for some $A < R_0$. 

The last term on the right hand side is a problematic one and will cause several difficulties that we will have to overcome later on. The existence of this term is forced by the degeneracy of the redshift effect on extremal Reissner--Nordstr\"{o}m and is a novel aspect of our analysis. 

\begin{proof}
We will prove the required estimate for the current $J^N$ for the vector field $N$ that we construct below.

We use the vector field $N$ from \cite{A1} which is future directed, timelike and $\Phi_T$-invariant close to the horizon $\ho$ (it should be noted that it acts as a substitute for the red-shift effect which degenerates in the extremal case as we noted before). This is defined as follows:
$$ N = N^v (r) T + N^r (r) Y \mbox{ with } $$
$$ N^v (r) > 0 \mbox{ for all $r\meg M$ and } N^v (r) =1 \mbox{ for all $r \meg \dfrac{8M}{7}$} ,$$
$$ N^v (r) \mik 0 \mbox{ for all $r\meg M$ and } N^v (r) =0 \mbox{ for all $r \meg \dfrac{8M}{7}$} .$$
With this definition as someone can compute we get that the current $J^N$ contracted with the normal of any $\sis_{\tau}$ hypersurface controls the $\dot{H}^1$ norm of $\psi$:
$$ J^N_{\mu} [\psi ] n^{\mu}_{\sis_{\tau}} \approx (T\psi )^2 + (Y\psi )^2 .$$
We will show \eqref{degi1} by proving the analogous statement with $J^N$ in the place of $J^{n_{\sis}}$.

We define the modified current
$$ J^{N , \delta , h}_{\mu} = J^N_{\mu} + h (r) \delta (r) \psi \nabla_{\mu} \psi .$$
We let $h=-\frac{1}{2}$ and $\delta : [M, \infty) \rightarrow \rr$ to be a smooth cut-off function such that $\delta (r) = 0$ for $r\meg \dfrac{8M}{7}$ and $\delta (r) =1$ for $M \mik r \mik \dfrac{9M}{8}$.

For the divergence of $J^{N, \delta , -\frac{1}{2}}$ we have
\begin{equation}\label{bulkk}
 K^{N, \delta , -\frac{1}{2}} [\psi ] + \mathcal{E}^N [\psi ] = \nabla^{\mu} J^{N, \delta , -\frac{1}{2}}_{\mu} [\psi] ,
 \end{equation}
we have that
$$ K^{N, \delta , -\frac{1}{2}} [\psi ] = 0 \mbox{ for $r \meg \dfrac{8M}{7}$} ,$$
since in this region $\delta = 0$ and $N=T$. On the other hand by Proposition 10.2.1 of \cite{A1} we have that:
$$ K^{N,\delta , -\frac{1}{2}} [\psi ] \gtrsim \left( (T\psi )^2 + \sqrt{D} (Y\psi )^2 - \dfrac{1}{2} \psi \cdot F \right) \mbox{ for $r \in \left[ M , \frac{9M}{8} \right]$} ,$$
for a spherically symmetric solution $\psi$ of \eqref{nweff}.

We apply Stokes' Theorem to the current $J^{N, \delta , -\frac{1}{2}}$ and we get that:
\begin{equation}\label{stokesn}
\int_{\si_{\tau_2}} J^{N, \delta , -\frac{1}{2}}_{\mu} [\psi ] n^{\mu} d\mi_{g_{\si}} + \int_{\ho} J^{N, \delta , -\frac{1}{2}}_{\mu} [\psi ] n^{\mu} d\mi_{g_{\ho}} + \int_{\mathcal{I}^{+}} J^{N, \delta , -\frac{1}{2}}_{\mu} [\psi ] n^{\mu} d\mi_{g_{\mathcal{I}^{+}}}  + 
\end{equation}
$$ + \int_{\tau_1}^{\tau_2} \int_{\si_{\tau'}} K^{N, \delta , -\frac{1}{2}}_{\mu} [\psi ] d\mi_{g_{\rrr}} + \int_{\tau_1}^{\tau_2} \int_{\si_{\tau'}} \mathcal{E}^N_{\mu} [\psi ] d\mi_{g_{\rrr}} = \int_{\si_{\tau_1}} J^{N, \delta , -\frac{1}{2}}_{\mu} [\psi ] n^{\mu} d\mi_{g_{\si}} .$$
By Propositions 10.3.1 and 10.4.1, and Corollary 10.1 of \cite{A1}, \eqref{stokesn} gives us the following (after noticing as well that on $\mathcal{I}^{+}$ the $N$-flux is positive since it is equal to $T$ there):
\begin{equation}\label{stokesn1}
\int_{\si_{\tau_2}} J^N_{\mu} [\psi ] n^{\mu} d\mi_{g_{\si}} + \int_{\tau_1}^{\tau_2} \int_{\si_{\tau'} \cap \{ M \mik r \mik 9M/8 \}} \left( (T\psi )^2 + \sqrt{D} (Y\psi )^2 \right)  d\mi_{g_{\rrr}} + 
\end{equation}
$$ + \int_{\tau_1}^{\tau_2} \int_{\si_{\tau'} \cap \{ r \meg 9M/8 \}} K^{N, \delta , -\frac{1}{2} } d\mi_{g_{\rrr}} \lesssim \int_{\si_{\tau_2}} J^T_{\mu} [\psi ] n^{\mu} d\mi_{g_{\si}} + \int_{\si_{\tau_1}} J^N_{\mu} [\psi ] n^{\mu} d\mi_{g_{\si}} + $$ $$ + \int_{\tau_1}^{\tau_2} \int_{\si_{\tau'}} |F \cdot N\psi| d\mi_{g_{\rrr}} + \left| \int_{\tau_1}^{\tau_2} \int_{\si_{\tau'} \cap \{ M \mik r \mik 9M/8 \}} \psi \cdot F d\mi_{g_{\rrr}} \right|  .$$
For the last term  we apply Cauchy-Schwarz, Hardy's inequality \eqref{hardy} and the Morawetz estimate \eqref{degix1} and we arrive at the following estimate for some $\eta > 0$:
$$ \left| \int_{\tau_1}^{\tau_2} \int_{\si_{\tau'} \cap \{ M \mik r \mik 9M/8 \}} \psi \cdot F d\mi_{g_{\rrr}} \right| \lesssim \int_{\si_{\tau_1}} J^T_{\mu} [\psi ] n^{\mu} d\mi_{g_{\si}} + \int_{\tau_1}^{\tau_2} \int_{\si_{\tau'}} r^{1+\eta} |F|^2 d\mi_{g_{\si}} + $$ $$ + \int_{\tau_1}^{\tau_2} \int_{\si_{\tau'} \cap \{ M \mik r \mik 9M/8 \}} |F|^2 d\mi_{g_{\rrr}} .$$

By Stokes' Theorem for $J^T$ we turn \eqref{stokesn1} to the following:
\begin{equation}\label{stokesn2}
\int_{\si_{\tau_2}} J^N_{\mu} [\psi ] n^{\mu} d\mi_{g_{\si}} + \int_{\tau_1}^{\tau_2} \int_{\si_{\tau'}} K^{N, \delta , -\frac{1}{2}}_{\mu} [\psi ] d\mi_{g_{\rrr}} \lesssim
\end{equation}
$$ \lesssim \int_{\si_{\tau_1}} J^T_{\mu} [\psi ] n^{\mu} d\mi_{g_{\si}} + \int_{\si_{\tau_1}} J^N_{\mu} [\psi ] n^{\mu} d\mi_{g_{\si}} + $$ $$ +\int_{\tau_1}^{\tau_2} \int_{\si_{\tau'}} |F \cdot N\psi| d\mi_{g_{\rrr}} + \int_{\tau_1}^{\tau_2} \int_{\si_{\tau'}} |F \cdot T\psi| d\mi_{g_{\rrr}} + \int_{\tau_1}^{\tau_2} \int_{\si_{\tau'} \cap \{ M \mik r \mik 9M/8 \}} |F|^2 d\mi_{g_{\rrr}} .$$

As we already noticed, the bulk term $\int_{\si_{\tau'}} K^{N, \delta , -\frac{1}{2}}_{\mu} [\psi ] d\mi_{g_{\rrr}}$ is 0 for $r \meg \dfrac{8M}{7}$ and positive close to the horizon. In the remaining region that we call $\mathcal{B}$, we use estimate \eqref{degix12} from the proof of Proposition \ref{degx} and we have that:
$$ \int_{\mathcal{B}} K^{N, \delta , -\frac{1}{2}}_{\mu} [\psi ] d\mi_{g_{\mathcal{B}}} \lesssim \int_{\si_{\tau_1}} J^T_{\mu} [\psi ] n^{\mu} d\mi_{g_{\si}} + \int_{\tau_1}^{\tau_2} \int_{\si_{\tau'}} r^{1+\eta} |F|^2 d\mi_{g_{\rrr}} .$$

Now we define a smooth cut-off $\chi : [M , \infty ) \rightarrow [0,1]$  such that $\chi (r) = 1$ for $M \mik r \mik A$ and $\chi (r) = 0$ for $r \meg A+1$. 

We apply Stokes' Theorem to the current $\chi J^{N , \delta , -\frac{1}{2}}$ and we get by estimate \eqref{stokesn2} (adapted to the spacelike-null foliation $\{ \si_{\tau} \}_{\tau \meg 0}$ where we note that the additional integral $\int_{\mathcal{I}^{+}} J^{N, \delta , -\frac{1}{2}} [\psi ] n^{\mu} d\mi_{g_{\mathcal{I}^{+}}}$ is positive, hence it has the right sign) the following:
\begin{equation}\label{stokesn11}
\int_{\si_{\tau_2}} \chi J^N_{\mu} [\psi ] n^{\mu} d\mi_{g_{\si}} + \int_{\tau_1}^{\tau_2} \int_{\si_{\tau'}} \chi K^{N, \delta , -\frac{1}{2}}_{\mu} [\psi ] d\mi_{g_{\rrr}} \lesssim
\end{equation}
$$ \lesssim \int_{\si_{\tau_1}} \chi J^T_{\mu} [\psi ] n^{\mu} d\mi_{g_{\si}} + \int_{\si_{\tau_1}} \chi J^N_{\mu} [\psi ] n^{\mu} d\mi_{g_{\si}} + $$ $$ +\int_{\tau_1}^{\tau_2} \int_{\si_{\tau'}} \chi |F \cdot N\psi| d\mi_{g_{\rrr}} + \int_{\tau_1}^{\tau_2} \int_{\si_{\tau'}} \chi |F \cdot T\psi| d\mi_{g_{\rrr}} + $$ $$  + \left| \int_{\tau_1}^{\tau_2} \int_{\si_{\tau'}} |\nabla \chi | \cdot J^{N, \delta , -\frac{1}{2}} [\psi ] n^{\mu} d\mi_{g_{\rrr}} \right|  .$$ 
First we note that we can apply Cauchy-Schwarz and estimate \eqref{degix1} for the term $\int_{\tau_1}^{\tau_2} \int_{\si_{\tau'}} \chi |F \cdot T\psi| d\mi_{g_{\rrr}}$ for some $\eta > 0$ as follows:
\begin{equation}\label{stokesn12}
 \int_{\tau_1}^{\tau_2} \int_{\si_{\tau'}} \chi |F \cdot T\psi| d\mi_{g_{\rrr}} \mik \int_{\tau_1}^{\tau_2} \int_{\si_{\tau'}}  |F \cdot T\psi| d\mi_{g_{\rrr}} \mik 
 \end{equation} 
 $$ \mik \int_{\tau_1}^{\tau_2} \int_{\si_{\tau'}}  r^{1+\eta} |F |^2 d\mi_{g_{\rrr}} + \int_{\tau_1}^{\tau_2} \int_{\si_{\tau'}} \dfrac{( T\psi )^2}{r^{1+\eta}} d\mi_{g_{\rrr}} \lesssim_{R_0} $$ $$ \lesssim_{R_0} \int_{\si_{\tau_1}}  J^T_{\mu} [\psi ] n^{\mu} d\mi_{g_{\si}} + \int_{\tau_1}^{\tau_2} \int_{S_{\tau'}} |F|^2 d\mi_{g_S} +  \int_{\tau_1}^{\tau_2} \int_{N_{\tau'}} r^{1+\eta} |F|^2 d\mi_{g_N} . $$ 
We note also that the last term of \eqref{stokesn11} can be bounded as well by the three last terms of \eqref{stokesn12}, just by noticing that we can apply estimate \eqref{degix1} to it as well because of Corollary 10.1 of \cite{A1} (which is a pointwise estimate) and because $supp (|\nabla \chi | ) \subset \{ A \mik r \mik A+1 \}$.

Hence we can rewrite \eqref{stokesn11} as follows:
\begin{equation}\label{stokesn13}
\int_{\si_{\tau_2} \cap \{ r \mik A \}}  J^N_{\mu} [\psi ] n^{\mu} d\mi_{g_{\si}} + \int_{\tau_1}^{\tau_2} \int_{\si_{\tau'} \cap \{ r\mik A\}}  \left( (T\psi )^2 + \sqrt{D} (Y\psi )^2 \right) d\mi_{g_{\rrr}} \lesssim_{R_0}
\end{equation}
$$ \lesssim_{R_0} \int_{\si_{\tau_1}}  J^N_{\mu} [\psi ] n^{\mu} d\mi_{g_{\si}} + \int_{\tau_1}^{\tau_2} \int_{S_{\tau'}} |F|^2 d\mi_{g_S} +  \int_{\tau_1}^{\tau_2} \int_{N_{\tau'}} r^{1+\eta} |F|^2 d\mi_{g_N} + $$ $$ + \int_{\tau_1}^{\tau_2} \int_{\si_{\tau'} \cap \{ r\mik R_0 \}}  |F \cdot N\psi| d\mi_{g_{\rrr}} .$$
Since $\int_{\si_{\tau_2}} J^T_{\mu} [\psi] n^{\mu} d\mi_{g_{\si}}$ is positive we add it on both sides of \eqref{stokesn13} and as $J^T$ behaves approximately like $J^N$ away from the horizon we have that:
$$ \int_{\si_{\tau_2}} J^T_{\mu} [\psi] n^{\mu} d\mi_{g_{\si}} + \int_{\si_{\tau_2} \cap \{ r\mik A\}} J^N_{\mu} [\psi] n^{\mu} d\mi_{g_{\si}} \approx \int_{\si_{\tau_2}} J^N_{\mu} [\psi] n^{\mu} d\mi_{g_{\si}} .$$
By applying the result of Proposition \ref{deg5} we have that \eqref{stokesn13} becomes:
\begin{equation}\label{stokesn14}
 \int_{\si_{\tau_2}} J^N_{\mu} [\psi] n^{\mu} d\mi_{g_{\si}} +  \int_{\tau_1}^{\tau_2} \int_{\si_{\tau'} \cap \{ r\mik A\}}  \left( (T\psi )^2 + \sqrt{D} (Y\psi )^2 \right) d\mi_{g_{\rrr}} \lesssim_{R_0}
\end{equation}
$$ \lesssim_{R_0} \int_{\si_{\tau_1}}  J^N_{\mu} [\psi ] n^{\mu} d\mi_{g_{\si}} + \int_{\tau_1}^{\tau_2} \int_{S_{\tau'}} |F|^2 d\mi_{g_S} +  \int_{\tau_1}^{\tau_2} \int_{N_{\tau'}} r^{1+\eta} |F|^2 d\mi_{g_N} + $$ $$ + \int_{\tau_1}^{\tau_2} \int_{\si_{\tau'} \cap \{ r\mik R_0 \}}  |F \cdot N\psi| d\mi_{g_{\rrr}} .$$
We finally deal with the last term by applying Cauchy-Schwarz to it twice:
\begin{equation}\label{stokesn15} \int_{\tau_1}^{\tau_2} \int_{\si_{\tau'} \cap \{ r\mik R_0 \}}  |F \cdot N\psi| d\mi_{g_{\rrr}} \mik 
\end{equation}
 $$\mik \int_{\tau_1}^{\tau_2} \left( \int_{\si_{\tau'} \cap \{ r\mik R_0 \}}  (N\psi )^2 d \mi_{g_{\si}} \right)^{1/2} \left( \int_{\si_{\tau'} \cap \{ r\mik R_0 \}}  |F|^2 d \mi_{g_{\si}} \right)^{1/2}  d\tau' \mik $$ $$ \mik \sup_{\tau'' \in [\tau_1 , \tau_2 ]}  \left( \int_{\si_{\tau''} \cap \{ r\mik R_0 \}}  (N\psi )^2 d \mi_{g_{\si}} \right)^{1/2} \int_{\tau_1}^{\tau_2} \left( \int_{\si_{\tau'} \cap \{ r\mik R_0 \}}  |F|^2 d \mi_{g_{\si}} \right)^{1/2} d\tau' \mik $$ $$ \mik \beta \sup_{\tau'' \in [\tau_1 , \tau_2 ]}  \int_{\si_{\tau''} \cap \{ r\mik R_0 \}}  (N\psi )^2 d \mi_{g_{\si}} + \dfrac{1}{\beta} \left( \int_{\tau_1}^{\tau_2} \left( \int_{\si_{\tau'} \cap \{ r\mik R_0 \}}  |F|^2 d \mi_{g_{\si}} \right)^{1/2} d\tau' \right)^2 , $$
for any $\beta > 0$. Since in the region $\{ r\mik R_0 \}$ we have that:
$$ \int_{\si_{\tau}\cap \{r \mik R_0\}} (N\psi )^2 d\mi_{g_{\si}} \lesssim \int_{\si_{\tau}\cap \{r \mik R_0\}} J^N_{\mu} [\psi] n^{\mu} d\mi_{g_{\si}},$$
we can restate \eqref{stokesn15} as
\begin{equation}\label{stokesn16}
\int_{\tau_1}^{\tau_2} \int_{\si_{\tau'} \cap \{ r\mik R_0 \}}  |F \cdot N\psi| d\mi_{g_{\rrr}} \mik 
\end{equation} 
$$ C\beta \sup_{\tau'' \in [\tau_1 , \tau_2 ]}  \int_{\si_{\tau''} \cap \{ r\mik R_0 \}}  J^N_{\mu} [\psi] n^{\mu} d \mi_{g_{\si}} + \dfrac{C}{\beta} \left( \int_{\tau_1}^{\tau_2} \left( \int_{\si_{\tau'} \cap \{ r\mik R_0 \}}  |F|^2 d \mi_{g_{\si}} \right)^{1/2} d\tau' \right)^2 .$$

Now the proof finishes by taking the supremum over all $\tau \in [\tau_1 , \tau_2 ]$ on the left hand side of \eqref{stokesn14} and by choosing $\beta$ small enough in order to be able to absorb the term $\sup_{\tau'' \in [\tau_1 , \tau_2 ]}  \int_{\si_{\tau'} \cap \{ r\mik R_0 \}}  J^N_{\mu} [\psi ] n^{\mu} d \mi_{g_{\si}}$ in the left hand side of \eqref{stokesn14}.

\end{proof}

\subsection{$P$-Energy Estimates}
In order to get useful integrated local energy decay estimates we use the novel vector field $P$ that was introduced in \cite{A2}. 

This vector field has the form $P = P^v (r) T + P^r (r) Y $ where 

(a) $P^r (r) = -\sqrt{D}$ for $M \mik r \mik r_0 < 2M$ ($r_0$ will be determined later), $P^r (r) = 0$ for $r \meg r_1 > r_0$ where $r_1 < 2M$ and where the extension in the remaining region is a smooth one.

(b) $P^v (r)$ for $M \mik r \mik r_0 < 2M$ is such that $\frac{1}{\beta} \left[ \sqrt{D} \left(Y P^v \right) + \frac{2}{r} \right]^2 < Y P^v$ for some $\beta > 0$ which is small enough, and $P^v (r) = 1$ for $r \meg r_1$ where again the extension in the remaining region is a smooth one.

Choosing $\beta$ to be small enough, then we can choose an $r_0$ with the property that in the region $\mathcal{F} = \{ M \mik r \mik r_0 < 2M \}$ we have:
\begin{equation}\label{pbulk}
K^P [\psi ] \approx (T\psi )^2 + D (Y \psi )^2 + |\slashed{\nabla} \psi |^2 \approx J^T_{\mu} [\psi] n^{\mu}_{\Sigma} .
\end{equation}
Also since in $\mathcal{B}$ we have that $g (P, P ) \approx -\sqrt{D}$, it is also true that:
\begin{equation}\label{pest}
J^P_{\mu} [\psi] n^{\mu}_{\Sigma} \approx (T\psi )^2 + \sqrt{D} (Y \psi )^2 + |\slashed{\nabla} \psi |^2 .
\end{equation}
Note that we stated the above estimates for general $\psi$ and not just spherically symmetric ones.

Let us state first a very simple estimate.
\begin{prop}\label{pener}
For all spherically symmetric solutions $\psi$ of \eqref{nweff} we have for any $\tau_1$, $\tau_2$ with $\tau_1 < \tau_2$ and any $\eta > 0$:
\begin{equation}\label{pen}
\int_{\Sigma_{\tau_2}} J^P_{\mu} [\psi] n^{\mu} d\mi_{g_{\si}} \lesssim_{R_0} \int_{\Sigma_{\tau_1}} J^P_{\mu} [\psi] n^{\mu} d\mi_{g_{\si}} + 
\end{equation}
$$ + \int_{\tau_1}^{\tau_2} \int_{S_{\tau'}} |F|^2 d\mi_{g_{\so}} +  \int_{\tau_1}^{\tau_2} \int_{N_{\tau'}} r^{1+\eta} |F|^2 d\mi_{g_{\nnn}}+ \int_{\tau_1}^{\tau_2} \int_{\Sigma_{\tau'}} | F \cdot P\psi | d\mi_{g_{\rrr}} . $$
\end{prop}
\begin{proof}
We apply Stokes' theorem to the current $J^P$ and we get:
$$ \int_{\Sigma_{\tau_2}} J^P_{\mu} [\psi]n^{\mu} d\mi_{g_{\si}} + \int_{\mathcal{H}^{+}} J^P_{\mu} [\psi]n^{\mu} d\mi_{g_{\ho}} + \int_{\mathcal{I}^{+}} J^P_{\mu} [\psi]n^{\mu} d\mi_{g_{\mathcal{I}^{+}}} + $$ $$ + \int_{\mathcal{R} (\tau_1 , \tau_2 )}  K^P [\psi] dvol + \int_{\mathcal{R} (\tau_1 , \tau_2 )} \mathcal{E}^P [\psi]  d\mi_{g_{\rrr}} = \int_{\Sigma_{\tau_1}} J^P_{\mu} [\psi]n^{\mu} d\mi_{g_{\si}} .$$ 
Since $P$ is a causal future-directed vector field, the boundary integrals over $\mathcal{H}^{+}$ and $\mathcal{I}^{+}$ are positive. On the other hand we have that $K^P$ is positive close to the horizon by \eqref{pbulk}, and away from it we can just use the Morawetz estimate \eqref{degix1} (since the different power of $D$ in front of $Y \psi$ is not important away from the horizon).

Finally the second term on the right hand side of \eqref{pen} comes from the term $\int_{\mathcal{R} (\tau_1 , \tau_2 )} \mathcal{E}^P [\psi] d\mi_{g_{\rrr}}$.
\end{proof}
Proposition \ref{pener} has as a consequence the following bound for the integrated $T$-flux close to the horizon.

\begin{prop}\label{ilep}
For all spherically symmetric solutions $\psi$ of \eqref{nweff} we have for any $\tau_1$, $\tau_2$ with $\tau_1 < \tau_2$, any $\eta > 0$ and for $r_0$ given in the definition of the $P$-flux:
\begin{equation}\label{ilep1}
\int_{\tau_1}^{\tau_2} \int_{ \Sigma_{\tau'} \cap \{ r \mik r_0 \}} J^T_{\mu} [\psi] n^{\mu} d\mi_{g_{\rrr}} \lesssim_{R_0,\eta} \int_{\Sigma_{\tau_1}} J^P_{\mu} [\psi] n^{\mu} d\mi_{g_{\si}} + 
\end{equation}
$$ + \int_{\tau_1}^{\tau_2} \int_{S_{\tau'}} |F|^2 d\mi_{g_{\so}} +  \int_{\tau_1}^{\tau_2} \int_{N_{\tau'}} r^{1+\eta} |F|^2 d\mi_{g_{\nnn}}   .$$
\end{prop}
\begin{proof}
This is just a combination of Proposition \ref{pener} and estimate \eqref{pbulk}.

The only term that needs some attention is the one that comes from the nonlinearity, namely the term:
$$ \int_{\tau_1}^{\tau_2} \int_{\Sigma_{\tau'}} | F \cdot P\psi | d\mi_{g_{\rrr}} . $$
We apply Cauchy-Schwarz with weight $r^{1+\eta}$ (for any $\eta > 0$) and for any $\beta > 0$ we have:
$$ \int_{\tau_1}^{\tau_2} \int_{\Sigma_{\tau'}} | F \cdot P\psi | d\mi_{g_{\rrr}} \mik \beta \int_{\tau_1}^{\tau_2} \int_{\Sigma_{\tau'}}  \dfrac{( P\psi )^2}{r^{1+\eta}} d\mi_{g_{\rrr}} + \dfrac{1}{\beta} \int_{\tau_1}^{\tau_2} \int_{\Sigma_{\tau'}} r^{1+\eta} | F |^2 d\mi_{g_{\rrr}} .$$
We note that $(P\psi )^2 \lesssim (T\psi )^2 + D (Y\psi )^2$ close to the horizon. Hence for the first term of the right hand side of this last inequality we can absorb in $\int_{\tau_1}^{\tau_2} \int_{\mathcal{F} \cap \Sigma_{\tau'}} J^T_{\mu} [\psi] n^{\mu} d\mi_{g_{\rrr}}$ the part of it that is close to the horizon, while the rest can be bounded by the Morawetz estimate \eqref{degix1}.
\end{proof}

The integrated boundedness of the $T$-flux close to the horizon of Proposition \ref{ilep} that we just stated has as a consequence the uniform boundedness of the $P$-flux.
\begin{prop}[\textbf{Uniform Boundedness For The $P$-flux}]\label{puniform}
Let $\psi$ be a spherically symmetric solution of \eqref{nweff}. Then for any $\tau_1$, $\tau_2$ with $\tau_1 < \tau_2$ and for any $\eta > 0$ we have:
\begin{equation}\label{puniform1}
\int_{\Sigma_{\tau_2}} J^P_{\mu} [\psi] n^{\mu} d\mi_{g_{\si}} \lesssim_{R_0,\eta} \int_{\Sigma_{\tau_1}} J^P_{\mu} [\psi] n^{\mu} d\mi_{g_{\si}} + 
\end{equation}
$$+ \int_{\tau_1}^{\tau_2} \int_{S_{\tau'}} |F|^2 d\mi_{g_{\so}} +  \int_{\tau_1}^{\tau_2} \int_{N_{\tau'}} r^{1+\eta} |F|^2 d\mi_{g_{\nnn}} .$$
\end{prop}
\begin{proof}
We use Proposition \ref{pener} and for the nonlinear term $\int_{\tau_1}^{\tau_2} \int_{\si_{\tau'}} |F \cdot P\psi | d\mi_{g_{\si}}$ we just perform Cauchy-Schwarz with weight $r^{1+\eta}$ for some $\eta > 0$. This gives us that:
$$\int_{\tau_1}^{\tau_2} \int_{\si_{\tau'}} |F \cdot P\psi | d\mi_{g_{\si}} \lesssim_{R_0,\eta}  \int_{\tau_1}^{\tau_2} \int_{S_{\tau'}} |F|^2 d\mi_{g_{\so}} +  $$ $$ +\int_{\tau_1}^{\tau_2} \int_{N_{\tau'}} r^{1+\eta} |F|^2 d\mi_{g_{\nnn}} + \int_{\tau_1}^{\tau_2} \int_{\si_{\tau'}} \dfrac{(T\psi )^2 + D (Y\psi )^2 }{r^{1+\eta}} d\mi_{g_{\rrr}} ,$$
and for the last term we use Proposition \ref{ilep} close to the horizon and the Morawetz estimate \eqref{degix1} everywhere else. 

We should note here that close to the horizon we can't use the Morawetz estimate since in front of $(Y\psi )^2$ we have just $D$ and not $D^2$.

\end{proof}
Finally we note that we can also obtain a bound for the integrated $P$-flux close to the horizon.

\begin{prop}\label{ilepp}
For all solutions $\psi$ of \eqref{nweff} we have for any $\tau_1$, $\tau_2$ with $\tau_1 < \tau_2$, any $\eta > 0$ and for $r_0$ given in the definition of the $P$-flux:
\begin{equation}\label{ilep2}
\int_{\tau_1}^{\tau_2} \int_{\Sigma_{\tau'} \cap \{ r \mik r_0 \}} J^P_{\mu} [\psi] n^{\mu} d\mi_{g_{\rrr}} \lesssim_{R_{0},\eta} \int_{\Sigma_{\tau_1}} J^N_{\mu} [\psi] n^{\mu} d\mi_{g_{\si}} + 
\end{equation}
$$ + \int_{\tau_1}^{\tau_2} \int_{S_{\tau'}} |F|^2 d\mi_{g_S} +  \int_{\tau_1}^{\tau_2} \int_{N_{\tau'}} r^{1+\eta} |F|^2 d\mi_{g_{\nnn}} +  $$ $$ + \left( \int_{\tau_1}^{\tau_2} \left( \int_{S_{\tau'}} | F |^2 d\mi_{g_S} \right)^{1/2} d\tau' \right)^2 .$$
\end{prop}
\begin{proof}
This is just a combination of estimate \eqref{pest} and estimate \eqref{stokesn14} (after estimating the term $\int_{\tau_1}^{\tau_2} \int_{\si_{\tau'} \cap \{ r\mik R_0 \}}  |F \cdot N\psi| d\mi_{g_{\rrr}}$ as in estimates \eqref{stokesn15} and \eqref{stokesn16}) from the proof of Proposition \ref{deg4}.
\end{proof}

\subsection{$r^p -$Weighted Energy Inequality}
We give now weighted inequalities (the weight being with respect to $r$) for integrated and non-integrated energies on the null part $\{ N_{\tau} \}_{\tau \meg 0}$ of the spacelike-null foliation $\{ \Sigma_{\tau} \}_{\tau \meg 0}$. In the following Proposition we use the null coordinates $(u,v) = (t-r^{*} , t+r^{*} )$.  
\begin{prop}[\textbf{$r-$Weighted Energy Inequalities In A Neighbourhood Of Null Infinity}]\label{rw}
Let $p<3$ and define $\ph = r\psi$ for $\psi$ a spherically symmetric solution of \eqref{nwef}. Then we have:
\begin{equation}\label{rwi}
\int_{N_{\tau_2}} r^p \dfrac{( \partial_v \ph )^2}{r^2} d\mi_{g_N} + \int_{\tau_1}^{\tau_2} \int_{N_{\tau}} p r^{p-1} \dfrac{(\partial_v \ph)^2}{r^2}  d\mi_{g_{\nnn}} \lesssim_{p , R_0,\eta}
\end{equation}
$$
\lesssim_{p , R_0} \int_{N_{\tau_1}} r^p \dfrac{( \partial_v \ph )^2}{r^2} d\mi_{g_N} + \int_{\Sigma_{\tau_1}} J^T_{\mu} [\psi] n^{\mu} d\mi_{g_{\si}}  + $$ $$ + \int_{\tau_1}^{\tau_2} \int_{S_{\tau'}} |F|^2 d\mi_{g_{\so}} +  \int_{\tau_1}^{\tau_2} \int_{N_{\tau'}} r^{1+\eta} |F|^2 d\mi_{g_{\nnn}} + \int_{\tau_1}^{\tau_2} \int_{N_{\tau'}} r^{p+1} |F|^2 d \mi_{g_{\nnn}} ,
$$
for any $\tau_1$, $\tau_2$ with $\tau_1 < \tau_2$ and any $\eta  > 0$.
\end{prop}
\begin{proof}

First we record the wave equation that is satisfied by $\ph$:
\begin{equation}
\Box_g \ph = -\frac{2}{r} (\partial_u \ph - \partial_v \ph ) + \frac{D'}{r} \ph + rF .
\end{equation}

Consider the vector field $V = r^{p-2} \partial_v$. For $\chi \in C^{\infty}_0 ([R_0 , \infty ))$ a cut-off function which is $=1$ in $[R_0 + 1 ,\infty )$ and $=0$ in $[R_0 , R_0 + 1/2]$, we apply Stokes' Theorem for $J^V [\chi \ph]$ and we get for any $\tau_1$, $\tau_2$ with $\tau_1 < \tau_2$:
$$ \int_{N_{\tau_2}} J^V_{\mu} [\chi \ph] n^{\mu} d\mi_{g_N} + \int_{\tau_1}^{\tau_2} \int_{N_{\tau'}} ( K^V [\chi \ph] + \mathcal{E}^V [\chi \ph] ) d\mi_{g_{\nnn}} = \int_{N_{\tau_1}} J^V_{\mu} [\chi \ph] n^{\mu} d\mi_{g_N} ,$$
where 
\begin{equation}\label{kpluse}
K^V [\ph] + \mathcal{E}^V [\ph] = p r^{p-3} (\partial_v \ph )^2 + D' r^{p-3} \ph (\partial_v \ph ) + r^{p-1} \partial_v  \ph \cdot F  .
\end{equation}
First we note that from all the terms above we will get error terms coming from the cut-off $\chi$. These are of the form:
$$ \int_{N_{\tau} \cap \{ R_0 \mik r \mik R_0 + 1\}} r^p \dfrac{ \ph^2}{r^2} d\mi_{g_N} , \quad \int_{\tau_1}^{\tau_2} \int_{N_{\tau'} \cap \{ R_0 \mik r \mik R_0 +1 \}} r^p \dfrac{ \ph^2}{r^2} d\mi_{g_{\nnn}} .$$
They can be bounded in the end by:
$$ \lesssim_{R_0} \int_{\Sigma_{\tau_1}} J^T_{\mu} [\psi] n^{\mu} d\mi_{g_{\si}} + \int_{\tau_1}^{\tau_2} \int_{S_{\tau'}} |F|^2 d\mi_{g_{\so}} +  \int_{\tau_1}^{\tau_2} \int_{N_{\tau'}} r^{1+\eta} |F|^2 d\mi_{g_{\nnn}} ,$$ 
by applying first Hardy's inequality \eqref{hardyineq} and then estimates \eqref{degi2} and \eqref{degix1} respectively.

We can conclude now our proof if we manage to find proper estimates for the two last terms of \eqref{kpluse}. For the term $D' r^{p-3} \ph (\partial_v \ph )$ we note that:
$$ \int_{\tau_1}^{\tau_2} \int_{N_{\tau'}} D' r^{p-3} \ph (\partial_v \ph ) d\mi_{g_{\nnn}} = \int_{\tau_1}^{\tau_2} \int_{N_{\tau'}} r^{p-6} \frac{M}{2} D \left[ \sqrt{D} (3-p) - \frac{3M}{r} \right] (\chi \ph )^2 d\mi_{g_{\nnn}} - $$ $$ - \int_{L_{R_0 , \tau_1 , \tau_2 }} \frac{r^{p-3}}{4} D' \sqrt{D} (\chi \ph )^2 d\mi_{g_L} + \int_{\mathcal{I}^{+}} \frac{r^{p-3}}{4} D' D (\chi \ph )^2 d\mi_{g_{\mathcal{I}^{+}}} ,$$
where $L_{R_0 , \tau_1 , \tau_2} = \cup_{\tau \in [\tau_1 , \tau_2 ]} \Sigma_{\tau} \cap \{ r = R_0 \}$. The integral over $L_{R_0 , \tau_1 , \tau_2}$ vanishes because of the cut-off $\chi$, while the other two integrals are positive (for $R_0$ being chosen to be sufficiently large -- this depends on the choice of $p$ as well) and can be ignored.

So till now we have:
\begin{equation}\label{rwnl}
\int_{N_{\tau_2}} r^p \dfrac{( \partial_v \ph )^2}{r^2} d\mi_{g_N} + \int_{\tau_1}^{\tau_2} \int_{N_{\tau}} p r^{p-1} \dfrac{(\partial_v \ph)^2}{r^2}  d\mi_{g_{\nnn}} \lesssim_{p , R_0} 
\end{equation}
$$
\lesssim_{p , R_0} \int_{N_{\tau_1}} r^p \dfrac{( \partial_v \ph )^2}{r^2} d\mi_{g_N} + \int_{\Sigma_{\tau_1}} J^T_{\mu} [\psi] n^{\mu} d\mi_{g_{\si}}  + $$ $$ + \int_{\tau_1}^{\tau_2} \int_{S_{\tau'}} |F|^2 d\mi_{g_{\so}} +  \int_{\tau_1}^{\tau_2} \int_{N_{\tau'}} r^{1+\eta} |F|^2 d\mi_{g_{\nnn}} + \int_{\tau_1}^{\tau_2} \int_{N_{\tau'}} | r^{p-1} \partial_v  \ph \cdot F | d\mi_{g_{\nnn}} .
$$
We apply Cauchy-Schwarz to the last term and we have:
$$\int_{\tau_1}^{\tau_2} \int_{N_{\tau'}} | r^{p-1} \partial_v  \ph \cdot F | d\mi_{g_{\nnn}} \mik $$ $$ \mik \beta (p)  \int_{\tau_1}^{\tau_2} \int_{N_{\tau'}} r^{p-1} \dfrac{(\partial_v \ph )^2}{r^2} d\mi_{g_{\nnn}} + \frac{1}{\beta (p)} \int_{\tau_1}^{\tau_2} \int_{N_{\tau'}} r^{p+1} |F|^2 d\mi_{g_{\nnn}} ,$$
for $\beta$ being appropriately small so that the first term of the last inequality can be absorbed by the right-hand side of \eqref{rwnl}. This finishes the proof. 
\end{proof}

We now state two almost direct consequences of Proposition \ref{rw} which will be useful for us later.

\begin{prop}\label{away1}
For $\psi$ a spherically symmetric solution of \eqref{nwef} and $\ph = r\psi$ we have that:
\begin{equation}\label{awayi1}
\int_{\tau_1}^{\tau_2} \int_{N_{\tau'}} J^T_{\mu} [\psi ] n^{\mu} d\mi_{g_{\nnn}} \lesssim_{R_0} \int_{\Sigma_{\tau_1}} J^T_{\mu} [\psi ] n^{\mu}d\mi_{g_{\si}}  + \int_{N_{\tau_1}} \dfrac{(\partial_v \ph )^2}{r} d\mi_{g_{\si}} + 
\end{equation}
$$ + \int_{\tau_1}^{\tau_2} \int_{S_{\tau'}} |F|^2 d\mi_{g_{\so}} + \int_{\tau_1}^{\tau_2} \int_{N_{\tau'}} r^{2} |F|^2 d\mi_{g_{\nnn}} , $$
for any $\tau_1$, $\tau_2$ with $\tau_1 < \tau_2$.
\end{prop}
\begin{proof}
We consider again the cut-off $\chi$ that was used in the proof of Proposition \ref{rw}. We notice that we have the following computation for $\widetilde{\ph} = r \chi \psi$:
\begin{equation}\label{awayi11}
\int_{\tau_1}^{\tau_2} \int_{N_{\tau'}} \dfrac{(\partial_v \widetilde{\ph} )^2}{r^2} d\mi_{g_{\nnn}} \meg \int_{\tau_1}^{\tau_2} \int_{N_{\tau'}} \dfrac{(\partial_v \widetilde{\ph} )^2}{2 D^2 r^2} d\mi_{g_{\nnn}} =
\end{equation}
$$ = \int_{\tau_1}^{\tau_2} \int_{N_{\tau'}} \dfrac{\left(\partial_v (\chi \psi) \right)^2}{2D^2} d\mi_{g_{\nnn}} + \int_{\tau_1}^{\tau_2} \int_{N_{\tau'}} \dfrac{\partial_v \left(r (\chi \psi )^2 \right)}{4D r^2} d\mi_{g_{\nnn}} .$$
But the last term of \eqref{awayi11} is positive since:
$$ \int_{\tau_1}^{\tau_2} \int_{N_{\tau'}} \dfrac{\partial_v \left(r (\chi \psi )^2 \right)}{4D r^2} d\mi_{g_{\nnn}} = \int_{\mathcal{I}^{+}} \dfrac{(\chi \psi )^2}{8r} d\mi_{g_{\mathcal{I}^{+}}} \meg 0 . $$

So we have that:
\begin{equation}\label{awayi12}
\int_{\tau_1}^{\tau_2} \int_{N_{\tau'}} \left(\partial_v (\chi \psi) \right)^2 d\mi_{g_{\nnn}} \lesssim \int_{\tau_1}^{\tau_2} \int_{N_{\tau'}} \dfrac{(\partial_v \widetilde{\ph} )^2}{r^2} d\mi_{g_{\nnn}} = \int_{\tau_1}^{\tau_2} \int_{N_{\tau'}} \dfrac{\left(\partial_v (\chi \ph) \right)^2}{r^2} d\mi_{g_{\nnn}} \lesssim
\end{equation} 
$$\lesssim \int_{\tau_1}^{\tau_2} \int_{N_{\tau'}} \chi^2 \dfrac{(\partial_v  \ph )^2}{r^2} d\mi_{g_{\nnn}} + \int_{\tau_1}^{\tau_2} \int_{N_{\tau'}} D^2 (\chi' )^2 \psi^2 d\mi_{g_{\nnn}} . $$

The second term of the second line of \eqref{awayi12} can be bounded by:
$$ \int_{\Sigma_{\tau_1}} J^T_{\mu} [\psi ] n^{\mu}d\mi_{g_{\si}} ++ \int_{\tau_1}^{\tau_2} \int_{S_{\tau'}} |F|^2 d\mi_{g_{\so}} + \int_{\tau_1}^{\tau_2} \int_{N_{\tau'}} r^{1+\eta} |F|^2 d\mi_{g_{\nnn}} , $$ 
for any $\eta > 0$, by using Hardy's inequality \eqref{hardyineq} and the Morawetz estimate \eqref{degix1}.

The first term of second line of \eqref{awayi12} can be bounded by:
$$ \int_{\Sigma_{\tau_1}} J^T_{\mu} [\psi ] n^{\mu}d\mi_{g_{\si}}  + \int_{\Sigma_{\tau_1}} \dfrac{(\partial_v \ph )^2}{r} d\mi_{g_{\si}} + 
 \int_{\tau_1}^{\tau_2} \int_{S_{\tau'}} |F|^2 d\mi_{g_{\so}} + \int_{\tau_1}^{\tau_2} \int_{N_{\tau'}} r^{2} |F|^2 d\mi_{g_{\nnn}} , $$
 by an application of Proposition \ref{rw} with $p=1$.
 
 The proof finishes by recalling that:
 $$ J^T_{\mu} [\psi] n^{\mu}_{N} \approx (\partial_v \psi )^2 .$$
\end{proof}
\begin{prop}\label{awayvr}
For $\psi$ a spherically symmetric solution of \eqref{nwef} and $\ph = r\psi$ we have that for any given $\aaa > 0$:
\begin{equation}\label{awayvri}
\int_{N_{\tau_2}} \dfrac{(\partial_v \ph )^2}{r^{1-\aaa}} d\mi_{g_{\nnn}} \lesssim_{R_0,\aaa} \int_{\Sigma_{\tau_1}} J^T_{\mu} [\psi ] n^{\mu}d\mi_{g_{\si}}  + \int_{N_{\tau_1}} \dfrac{(\partial_v \ph )^2}{r^{1-\aaa}} d\mi_{g_{\si}} + 
\end{equation}
$$ + \int_{\tau_1}^{\tau_2} \int_{S_{\tau'}} |F|^2 d\mi_{g_{\so}} + \int_{\tau_1}^{\tau_2} \int_{N_{\tau'}} r^{2+\aaa} |F|^2 d\mi_{g_{\nnn}} , $$
and
\begin{equation}\label{awayvri1}
\int_{\tau_1}^{\tau_2} \int_{N_{\tau'}} \dfrac{(\partial_v \ph )^2}{r^{1+\aaa} }d\mi_{g_{\nnn}} \lesssim_{R_0,\aaa} \int_{\Sigma_{\tau_1}} J^T_{\mu} [\psi ] n^{\mu}d\mi_{g_{\si}}  + \int_{N_{\tau_1}} \dfrac{(\partial_v \ph )^2}{r^{\aaa}} d\mi_{g_{\si}} + 
\end{equation}
$$ + \int_{\tau_1}^{\tau_2} \int_{S_{\tau'}} |F|^2 d\mi_{g_{\so}} + \int_{\tau_1}^{\tau_2} \int_{N_{\tau'}} r^{3-\aaa} |F|^2 d\mi_{g_{\nnn}} , $$
for any $\tau_1$, $\tau_2$ with $\tau_1 < \tau_2$.
\end{prop}
\begin{proof}
The proof of estimate \eqref{awayvri} follows directly by an application of Proposition \ref{rw} with $p=1+\aaa$ while the proof of estimate \eqref{awayvri1} follows again directly by another application of Proposition \ref{rw} with $p= 2-\aaa$.
\end{proof}

\subsection{Integrated Energy Decay}
We obtain an integrated local energy decay estimate in the entirety of the domain of outer communication for the $T$-flux by combining our previous results. 

\begin{prop}[\textbf{Integrated Energy Decay For The $T$-flux}]\label{iled}
Let $\psi$ be a spherically symmetric solution of \eqref{nwef} and let $\ph = r\psi$. Then for any $\tau_1$, $\tau_2$ with $\tau_1 < \tau_2$ we have:
\begin{equation}\label{iled1}
\int_{\tau_1}^{\tau_2} \int_{\Sigma_{\tau'}} J^T_{\mu} [\psi] n^{\mu} d\mi_{g_{\si}} \lesssim_{R_0} \int_{\Sigma_{\tau_1}} J^P_{\mu} [\psi] n^{\mu} d\mi_{g_{\si}} + \int_{N_{\tau_1}} \dfrac{(\partial_v \ph )^2}{r} d\mi_{g_N} + 
\end{equation}
$$ + \int_{\tau_1}^{\tau_2} \int_{S_{\tau'}} |F|^2 d\mi_{g_{\so}} + \int_{\tau_1}^{\tau_2} \int_{N_{\tau'}} r^2 |F|^2 d\mi_{g_{\nnn}} . $$
\end{prop}
\begin{proof}
We break our integral as follows:
$$ \int_{\tau_1}^{\tau_2} \int_{\Sigma_{\tau'}} J^T_{\mu} [\psi] n^{\mu} d\mi_{g_{\si}} =  $$ $$ = \int_{\tau_1}^{\tau_2} \int_{\Sigma_{\tau'} \cap \{M \mik r \mik r_1 < 2M\}} J^T_{\mu} [\psi] n^{\mu} d\mi_{g_{\rrr}} + \int_{\tau_1}^{\tau_2} \int_{\Sigma_{\tau'} \cap \{ r_1 \mik r \mik R_0\}} J^T_{\mu} [\psi] n^{\mu} d\mi_{g_{\rrr}} + $$ $$ + \int_{\tau_1}^{\tau_2} \int_{\Sigma_{\tau'} \cap \{ r \meg R_0 \}} J^T_{\mu} [\psi] n^{\mu} d\mi_{g_{\rrr}} := A + B + C .$$

We use Proposition \ref{ilep} for the spacetime region close to the event horizon $\mathcal{H}^{+}$ taking care of $A$, we use estimate \eqref{degix1} with some $\eta < 1$ for the "middle" spacetime region, which is both away from the event horizon and future null infinity taking care of $B$, and finally we use Proposition \ref{away1} for the spacetime region close to future null infinity taking care of $C$.
\end{proof}

We also state here another integrated local energy decay estimate for a quantity that we will use later when we will derive a priori energy decay estimates.
\begin{prop}[\textbf{Integrated Energy Decay For The $P$-flux}]\label{iledi}
Let $\psi$ be a spherically symmetric solution of \eqref{nwef} and let $\ph = r\psi$. Then for any $\tau_1$, $\tau_2$ with $\tau_1 < \tau_2$ we have:
\begin{equation}\label{iled2}
\int_{\tau_1}^{\tau_2} \int_{\Sigma_{\tau'}} J^P_{\mu} [\psi] n^{\mu} d\mi_{g_{\rrr}} \lesssim_{R_0} \int_{\Sigma_{\tau_1}} J^N_{\mu} [\psi] n^{\mu} d\mi_{g_{\si}} + \int_{N_{\tau_1}} \dfrac{(\partial_v \ph )^2}{r} d\mi_{g_N} +
\end{equation}
$$  + \int_{\tau_1}^{\tau_2} \int_{S_{\tau'}} |F|^2 d\mi_{g_{\so}} + \int_{\tau_1}^{\tau_2} \int_{N_{\tau'}} r^2 |F|^2 d\mi_{g_{\nnn}} + \left( \int_{\tau_1}^{\tau_2} \left( \int_{S_{\tau'} } | F |^2 d\mi_{g_S} \right)^{1/2} d\tau' \right)^2 . $$
\end{prop}
\begin{proof}
We break the integral $\int_{\tau_1}^{\tau_2} \int_{\Sigma_{\tau'}} J^P_{\mu} [\psi] n^{\mu} d\mi_{g_{\si}}$ into three terms as in the proof of Proposition \ref{iled}. We call the corresponding terms $A_1$, $B_1$, $C_1$. We treat the terms $B_1$, $C_1$ as in the proof of Proposition \ref{iled}, while for $A_1$ we use Proposition \ref{ilepp}. Again we choose $\eta < 1$ for all the Morawetz-type estimates that we use.

\end{proof}

\section{Local Well-Posedness}\label{tlwpt}
In this section we investigate the local theory for \eqref{nwef}. For this section we will set $A(\psi ) \equiv 1$ since the extra terms that come from the $T$ and $Y$ derivatives of $A$ don't change anything in the end.

We will use the following norm:
$$ \| \psi \|_{X^3 (\widetilde{\Sigma}_{\tau} )}^2 = \left\| \dfrac{\psi}{r} \right\|_{L^2 (\widetilde{\Sigma}_{\tau} )}^2  + \sum_{k+l\mik 3, k,l \meg 0 , k+l \neq 0 } \| T^k  Y^l \psi \|_{L^2 (\widetilde{\Sigma}_{\tau} )}^2 ,  $$
for any given $\tau \meg 0$.

Since as someone can check:
$$ \| \psi \|_{X^3 (\widetilde{\Sigma}_{\tau} )}^2 \approx \| ( \psi n_{\sis_{\tau}} \psi ) \|_{X^3 \times X^2 (\sis_{\tau} )} , $$

we note that this norm is a variation of the $H^3 \times H^2 (\widetilde{\Sigma}_{\tau} )$ norm since approximately we have:
$$ \| ( \psi , n_{\sis_{\tau}} \psi ) \|_{H^3 \times H^2 (\widetilde{\Sigma}_{\tau} )}^2 \approx \| \psi\|_{L^2 (\widetilde{\Sigma}_{\tau} )}^2 +\| n_{\widetilde{\Sigma}_{\tau}} \psi \|_{H^2 (\widetilde{\Sigma}_{\tau} )}^2 \approx $$ $$ \approx \| \psi\|_{L^2 (\widetilde{\Sigma}_{\tau} )}^2  + \sum_{k+l\mik 3, k,l \meg 0 , k+l \neq 0 } \| T^k  Y^l \psi \|_{L^2 (\widetilde{\Sigma}_{\tau} )}^2 .$$
We should note of course that this computation is not really precise since the function $\bar{r}$ of \eqref{rbar} tends to 0 as we approach the spacelike infinity $\iota^0$. It is though an appropriate norm for us in order to be able to formulate a local theory for \eqref{nwef}.

\begin{thm}\label{lwp}
Equation \eqref{nwef} is locally well-posed in $X = X^3 \times X^2$ in the sense the if we start with data $( \psi_0 , \psi_1 ) \in X^3 \times X^2 (\widetilde{\Sigma}_0 )$ then there exists some $\mathcal{T} > 0$ such that there exists a unique solution of \eqref{nwef} in $\mathcal{R}(0,\mathcal{T} ) = \cup_{\tau \in [0,\mathcal{T}]} \widetilde{\Sigma}_{\tau}$ for which we have that for each $\tau \in [0, \mathcal{T}]$ it holds that $\left( \psi (\tau) , n_{\sis_{\tau}} \psi (\tau) \right) \in X (\widetilde{\Sigma}_{\tau} )$.
\end{thm}
\begin{rem}
The same result holds for equation \eqref{nw}. The proof is similar.
\end{rem}
\begin{rem}
Recall that if we start with data that is spherically symmetric then our solution of \eqref{nwef} will be spherically symmetric as well. The same holds for solutions of equation \eqref{nw}.
\end{rem}
\begin{proof}[Sketch of Proof]
We will just give the outline of this proof by proving the main estimate. The rest is considered to be standard. 

For the zero-th order term we use Hardy's inequality \eqref{hardy}, for the higher derivatives we use the "time"-dependent uniform boundedness for the non-degenerate energy of Proposition \ref{deg1} for all the terms $J^N [\psi]$, $J^N [Y \psi]$, $J^N [Y^2 \psi]$, $J^N [T\psi]$, $J^N [T^2 \psi]$ (hence, all constants involved in the following inequalities will be "time"-dependent, this is not important for the local theory). We will also use the commutation relations:
\begin{equation}\label{commr}
[ \Box_g , Y ] = -D' \cdot Y^2 + \frac{2}{r^2}  T - R' \cdot Y ,
\end{equation}
where $R = D' + \frac{2D}{r}$, and
\begin{equation}\label{commv}
[\Box_g , T ] = 0 .
\end{equation}
A consequence of \eqref{commr} and \eqref{commv} are the following equations:
\begin{equation}\label{partialr}
\Box_g (Y \psi ) = -D' \cdot ( Y^2 \psi ) + \frac{2}{r^2} ( T \psi ) - R' \cdot  ( Y \psi ) + D' \cdot (Y \psi )^2 + 
\end{equation} 
 $$ + 2D \cdot ( Y^2 \psi ) \cdot Y \psi + 2 ( TY \psi ) \cdot ( Y \psi ) + 2 ( T \psi ) \cdot ( Y^2 \psi ) ,$$

\begin{equation}\label{partialv}
\Box_g T^l \psi = T^l F \mbox{ for any $l\meg 0$} .
\end{equation}
In more detail, for the first term $\left\| \dfrac{\psi}{r} \right\|_{L^2 (\widetilde{\Sigma}_{\tau} )}$ we use Hardy's inequality \eqref{hardy} and we have:
$$ \left\| \dfrac{\psi}{r} \right\|_{L^2 (\widetilde{\Sigma}_{\tau} )}^2 \lesssim \int_{\widetilde{\Sigma}_{\tau} } J^T_{\mu} [\psi ] n^{\mu} d\mi_{g_{\sis}} \lesssim $$ $$ \lesssim \int_{\widetilde{\Sigma}_{0} } J^T_{\mu} [\psi ] n^{\mu} d\mi_{g_{\sis}} + \int_0^{\tau} \int_{\widetilde{\Sigma}_{\tau'}} |F|^2 d\mi_{g_{\sis}} \lesssim $$ $$ \lesssim \int_{\widetilde{\Sigma}_{0} } J^T_{\mu} [\psi ] n^{\mu} d\mi_{g_{\sis}} + \sup_{\tau'' \in [0, \tau ]} \|Y \psi  \|_{L^{\infty} (\widetilde{\Sigma}_{\tau'} )}^2 \int_0^{\tau} \| \psi \|_{X (\widetilde{\Sigma}_{\tau' } )}^2 d\mi_{g_{\sis}} .$$
For the first derivative terms $T \psi$, $Y \psi$ we just use the uniform boundedness of the non-degenerate energy \eqref{degi1} and we get since:
$$ \| Y \psi \|_{L^2 (\widetilde{\Sigma}_{\tau} )}^2 + \| T \psi \|_{L^2 (\widetilde{\Sigma}_{\tau} )}^2 \approx \int_{\widetilde{\Sigma}_{\tau}} J^N_{\mu} [\psi ] n^{\mu} d\mi_{g_{\sis}} ,$$ 
the following: 
$$ \int_{\widetilde{\Sigma}_{\tau} } J^N_{\mu} [\psi ] n^{\mu} d\mi_{g_{\sis}}\lesssim \int_{\widetilde{\Sigma}_0 } J^N_{\mu} [\psi ] n^{\mu} d\mi_{g_{\sis}} + \int_0^{\tau} \int_{\widetilde{\Sigma}_{\tau'}} |F|^2 d\mi_{g_{\sis}} \lesssim $$ $$ \lesssim \int_{\widetilde{\Sigma}_0 } J^N_{\mu} [\psi ] n^{\mu} d\mi_{g_{\sis}} + \sup_{\tau'' \in [0, \tau ]} \|Y \psi  \|_{L^{\infty} (\widetilde{\Sigma}_{\tau'} )}^2 \int_0^{\tau} \int_{\widetilde{\Sigma}_{\tau'}} \left( D^2 |Y \psi |^2 + |T \psi |^2 \right) d\mi_{g_{\sis}} \lesssim $$ $$ \lesssim  \int_{\widetilde{\Sigma}_0 } J^N_{\mu} [\psi ] n^{\mu} d\mi_{g_{\sis}} + \sup_{\tau'' \in [0, \tau ]} \|Y \psi  \|_{L^{\infty} (\widetilde{\Sigma}_{\tau'} )}^2 \int_0^{\tau}  \int_{\widetilde{\Sigma}_{\tau' }}  \left( (T \psi )^2 + (Y \psi )^2 \right) d\mi_{g_{\sis}} \Rightarrow $$
$$ \Rightarrow \| Y \psi \|_{L^2 (\widetilde{\Sigma}_{\tau} )}^2 + \| T \psi \|_{L^2 (\widetilde{\Sigma}_{\tau} )}^2  \lesssim \| \psi \|_{X (\widetilde{\Sigma}_{0} )}^2 + \sup_{\tau'' \in [0, \tau ]} \|Y \psi  \|_{L^{\infty} (\widetilde{\Sigma}_{\tau'} )}^2 \int_0^{\tau}   \| \psi \|_{X (\widetilde{\Sigma}_{\tau' } )}^2 d\mi_{g_{\sis}} .$$

For the second derivatives $TY \psi$, $Y^2 \psi$ we use again the energy current $J^N$  and the commutation relation \eqref{commr}.

Calling the right hand side of \eqref{partialr} by $G$ and since 
$$ \int_{\widetilde{\Sigma}_{\tau}} J^N_{\mu} [Y \psi ] n^{\mu} d\mi_{g_{\sis}} \approx \| Y^2 \psi \|_{L^2 (\widetilde{\Sigma}_{\tau} )}^2+ \| TY \psi \|_{L^2 (\widetilde{\Sigma}_{\tau} )}^2 , $$ 
we get by applying the non-uniform boundedness of the non-degenerate energy \eqref{degi1} to $Y \psi$ that:
$$  \int_{\widetilde{\Sigma}_{\tau}} J^N_{\mu} [Y \psi ] n^{\mu} d\mi_{g_{\sis}} \lesssim \int_{\widetilde{\Sigma}_{0}} J^N_{\mu} [Y \psi ] n^{\mu} d\mi_{g_{\sis}} + \int_0^{\tau} \int_{\widetilde{\Sigma}_{\tau'}} |G|^2 d\mi_{g_{\sis}} $$
$$ \lesssim \int_{\widetilde{\Sigma}_{0}} J^N_{\mu} [Y \psi ] n^{\mu} d\mi_{g_{\sis}} + \int_0^{\tau} \int_{\widetilde{\Sigma}_{\tau'}}  ( \| \psi \|_{X (\widetilde{\Sigma}_{\tau'})}^2 + |Y F |^2 ) d\mi_{g_{\sis}} \lesssim $$
$$ \lesssim \| \psi \|_{X (\widetilde{\Sigma}_0 )}^2 + \sup_{\tau'' \in [0,\tau]} \left( \|Y \psi \|_{L^{\infty} (\widetilde{\Sigma}_{\tau''} )}^2 + \|T \psi \|_{L^{\infty} (\widetilde{\Sigma}_{\tau''} )}^2 \right) \int_0^{\tau} \int_{\widetilde{\Sigma}_{\tau'}} \| \psi \|_{X (\widetilde{\Sigma}_{\tau' } )}^2 d\mi_{g_{\sis}} .$$

For the third derivatives $Y^3 \psi$, $T Y^2 \psi$ on the other hand we use the relation \eqref{commr} to get:
$$ \Box_g (Y^2 \psi ) = - 2D' \cdot ( Y^3 \psi ) + \frac{4}{r^2} ( TY \psi ) - ( 2R' + D'' ) \cdot ( Y^2 \psi )  + $$ $$ + \frac{4}{r^3} ( T \psi ) + R'' \cdot ( Y \psi ) + Y^2 F .$$
We call the right hand side of the above equality by $Q$.

Now we use the non-uniform boundedness for $J^N [Y^2 \psi ]$ to arrive at the following:
$$  \| Y^3 \psi \|_{L^2 (\widetilde{\Sigma}_{\tau})}^2 + \| T Y^2 \psi \|_{L^2 (\widetilde{\Sigma}_{\tau})}^2 \lesssim \int_{\widetilde{\Sigma}_{\tau}} J^N_{\mu} [Y^2 \psi ] n^{\mu} d\mi_{g_{\sis}} \lesssim $$ $$ \lesssim \int_{\widetilde{\Sigma}_{\tau}} J^N_{\mu} [Y^2 \psi ] n^{\mu} d\mi_{g_{\sis}} + \int_0^{\tau} \int_{\widetilde{\Sigma}_{\tau'}} |Q|^2 d\mi_{g_{\sis}} \lesssim $$ $$ \lesssim \| \psi \|_{X (\widetilde{\Sigma}_0 )}^2 + \int_0^{\tau} ( \| \psi \|_{X (\widetilde{\Sigma}_{\tau'} )}^2 + \| (Y \psi + T \psi )^2 \|_{H^2 (\widetilde{\Sigma}_{\tau'} )}^2 ) d\mi_{g_{\sis}} $$ $$ \lesssim \| \psi \|_{X (\widetilde{\Sigma}_0 )}^2 + \sup_{\tau'' \in [0,\tau]} \left( \|Y \psi \|_{L^{\infty} (\widetilde{\Sigma}_{\tau''} )}^2 + \|T \psi \|_{L^{\infty} (\widetilde{\Sigma}_{\tau''} )}^2 \right)\int_0^{\tau} \| \psi  \|_{X (\widetilde{\Sigma}_{\tau'} )}^2 d\mi_{g_{\sis}} ,$$
where in the last line we used the Calculus inequality for products in Sobolev spaces.

For the second derivative $T^2 \psi$ we recall that because of \eqref{commv} we have the equation \eqref{partialv} for $Y \psi$ and we can compute:
$$ \| T^2 \psi \|_{L^2 (\widetilde{\Sigma}_{\tau} )}^2 \lesssim \int_{\widetilde{\Sigma}_{\tau}} J^N_{\mu} [\psi ] n^{\mu} d\mi_{g_{\sis}} \lesssim \int_{\widetilde{\Sigma}_{0}} J^N_{\mu} [\psi ] n^{\mu} d\mi_{g_{\sis}} + \int_0^{\tau} \int_{\widetilde{\Sigma}_{\tau'}} |TF|^2 d\mi_{g_{\sis}} \lesssim $$ $$ \lesssim \int_{\widetilde{\Sigma}_{0}} J^N_{\mu} [\psi ] n^{\mu} d\mi_{g_{\sis}} + $$ $$ + \sup_{\tau'' \in [0,\tau]} \left( \|Y \psi \|_{L^{\infty} (\widetilde{\Sigma}_{\tau''} )}^2 + \|T \psi \|_{L^{\infty} (\widetilde{\Sigma}_{\tau''} )}^2 \right)\int_0^{\tau} \int_{\widetilde{\Sigma}_{\tau'}} \left( |T^2 \psi|^2 + | TY \psi |^2 \right) d\mi_{g_{\sis}} $$ 
$$ \Rightarrow \| T^2 \psi \|_{L^2 (\widetilde{\Sigma}_{\tau} )}^2 \lesssim \| \psi \|_{X (\widetilde{\Sigma}_{\tau} )}^2 + $$ $$ + \sup_{\tau'' \in [0,\tau]} \left( \|Y \psi \|_{L^{\infty} (\widetilde{\Sigma}_{\tau''} )}^2 + \|T \psi \|_{L^{\infty} (\widetilde{\Sigma}_{\tau''} )}^2 \right) \int_0^{\tau} \int_{\widetilde{\Sigma}_{\tau'}} \| \psi \|_{X (\widetilde{\Sigma}_{\tau' } )}^2 d\mi_{g_{\sis}} .$$ 

For the third derivatives $T^3 \psi$, $T^2 Y \psi$ we use the non-uniform boundedness of $J^N [ T^2 \psi ]$ and we have:
$$ \| T^3 \psi \|_{L^2 (\widetilde{\Sigma}_{\tau})}^2 + \| T^2 Y \psi \|_{L^2 (\widetilde{\Sigma}_{\tau})}^2 \lesssim \int_{\widetilde{\Sigma}_{\tau}} J^N_{\mu} [T^2 \psi ] n^{\mu} d\mi_{g_{\sis}} \lesssim $$ $$ \lesssim \int_{\widetilde{\Sigma}_{0}} J^N_{\mu} [T^2 \psi ] n^{\mu} d\mi_{g_{\sis}} + \int_0^{\tau} \int_{\widetilde{\Sigma}_{\tau'}} |T^2 F|^2 d\mi_{g_{\sis}} \lesssim $$ $$ \lesssim \int_{\widetilde{\Sigma}_{0}} J^N_{\mu} [T^2 \psi ] n^{\mu} d\mi_{g_{\sis}} + \int_0^{\tau} \| (Y \psi + T \psi )^2 \|_{H^2 (\widetilde{\Sigma}_{\tau'} )}^2 d\mi_{g_{\sis}} \lesssim $$ $$ \lesssim  \int_{\widetilde{\Sigma}_{0}} J^N_{\mu} [T^2 \psi ] n^{\mu} d\mi_{g_{\sis}} + $$ $$ + \sup_{\tau'' \in [0,\tau]} \left( \|Y \psi \|_{L^{\infty} (\widetilde{\Sigma}_{\tau''} )}^2 + \|T \psi \|_{L^{\infty} (\widetilde{\Sigma}_{\tau''} )}^2 \right) \int_0^{\tau} \| Y \psi + T \psi  \|_{H^2 (\widetilde{\Sigma}_{\tau'} )}^2 d\mi_{g_{\sis}} \lesssim $$ $$ \lesssim \| \psi \|_{X (\widetilde{\Sigma}_0 )}^2 + \sup_{\tau'' \in [0,\tau]} \left( \|Y \psi \|_{L^{\infty} (\widetilde{\Sigma}_{\tau''} )}^2 + \|T \psi \|_{L^{\infty} (\widetilde{\Sigma}_{\tau''} )}^2 \right) \int_0^{\tau} \| \psi  \|_{X (\widetilde{\Sigma}_{\tau'} )}^2 d\mi_{g_{\sis}} ,$$ 
where in the last step we used again the Calculus inequality for products in Sobolev spaces. 

Finally we gather all the above estimates together and we have:
\begin{equation}\label{contlwp}
 \| \psi \|_{X (\widetilde{\Sigma}_{\tau})}^2 \lesssim \| \psi \|_{X (\widetilde{\Sigma}_0 )}^2 + \sup_{\tau'' \in [0,\tau]} \left( \|Y \psi \|_{L^{\infty} (\widetilde{\Sigma}_{\tau''} )}^2 + \|T \psi \|_{L^{\infty} (\widetilde{\Sigma}_{\tau''} )}^2 \right) \int_0^{\tau} \| \psi  \|_{X (\widetilde{\Sigma}_{\tau'} )}^2 d\mi_{g_{\sis}} .
 \end{equation}

By Sobolev's inequality we have for $\| \psi \|_{\widetilde{X}} := \sup_{\tau \in [0,\mathcal{T}]} \| \psi \|_{X (\widetilde{\Sigma}_{\tau} )}$ (replacing $\tau$ with $\mathcal{T}$):
$$ \Rightarrow \| \psi \|_{\widetilde{X} }^2 \lesssim \| \psi \|_{X (\widetilde{\Sigma}_0 )}^2 +  \mathcal{T} \| \psi \|_{\widetilde{X}}^4 . $$
For an appropriate choice of a small enough $\mathcal{T}$ we can apply the fixed point argument.
\end{proof}
The above proof implies by Sobolev's inequality that if we start with smooth $C^{\infty}$ data, our solution will be $C^{\infty}$ as well.
\section{Continuation Criterion}\label{tlwpcc}
We will state here a condition that allows us to extend a solution beyond the time $\mathcal{T}$ given by the local theory. We have the following continuation criterion.
\begin{prop}[\textbf{Breakdown Criterion}]\label{cont}
Let $\psi$ be a solution of \eqref{nwef} with smooth compactly supported data $\psi [0] = (f , g)$ with finite $X$ norm. Denote by $\mathcal{T} = \mathcal{T} (f,g)$ the maximal time of existence for $\psi$ given by Theorem \ref{lwp}. Then we have that either $\mathcal{T} = \infty$ (in which case we say that $\psi$ is globally well-posed) or we have that:
$$ T\psi \not\in L^{\infty} (\mathcal{\widetilde{R}} (0,\mathcal{T} )), \quad Y\psi \not\in L^{\infty} (\mathcal{\widetilde{R}} (0,\mathcal{T} )).  $$ 
\end{prop}

\begin{proof}
This is a standard consequence of \eqref{contlwp}, Gr\"{o}nwall's lemma and the local theory of the previous section.
\end{proof}

Proposition \ref{cont} tells us essentially that we have to verify that both 
$$ \| T\psi \|_{L^{\infty} (\mathcal{\widetilde{R}} (0,\mathcal{T} ))} \mbox{ and } \| Y\psi \|_{L^{\infty} (\mathcal{\widetilde{R}} (0,\mathcal{T} ))} $$
are finite at any given $\mathcal{T}$ in order to conclude that our solution $\psi$ (as in the assumption of Proposition \ref{cont}) is globally defined.

\begin{rem}
For equation \eqref{nw} the corresponding continuation criterion is:
$$  \| \psi \|_{L^{\infty} (\mathcal{\widetilde{R}} (0,\mathcal{T} ))}, \quad \| T\psi \|_{L^{\infty} (\mathcal{\widetilde{R}} (0,\mathcal{T} ))} \mbox{ and } \| Y\psi \|_{L^{\infty} (\mathcal{\widetilde{R}} (0,\mathcal{T} ))} $$ $$ \mbox{ are finite at any given time $\mathcal{T}$.} $$
If we consider equation \eqref{nw} without the derivative terms then the term 
$$\| \psi \|_{L^{\infty} (\mathcal{\widetilde{R}} (0,\mathcal{T} ))}$$
 alone determines the continuation criterion.
\end{rem}

\section{The Bootstrap Assumptions}
Let $\aaa>0$. The following bootstrap assumptions will be used from now on:
\begin{equation}\label{A1}
\int_{S_{\tau_0 }} |F|^2 d\mi_{g_S} \mik C_{\aaa } E_0 \ee^2 (1+\tau_0 )^{-2+\aaa} \tag{\textbf{A1}},
 \end{equation}
\begin{equation}\label{A2}
\int_{\tau_1}^{\tau_2} \int_{S_{\tau'}} |F|^2 d\mi_{g_{\so}} \mik C_{\aaa } E_0 \ee^2 (1+\tau_1 )^{-2+\aaa} \tag{\textbf{A2}},
\end{equation}
\begin{equation}\label{A3}
\int_{\tau_1}^{\tau_2} \int_{N_{\tau'}} r^{3-\aaa} |F|^2 d\mi_{g_{\nnn}} \mik C_{\aaa} E_0 \ee^2 (1+\tau_1 )^{-2+\aaa} \tag{\textbf{A3}} ,
 \end{equation}
 for any $\tau_0 $, $\tau_1$, $\tau_2$ with $\tau_1 < \tau_2$.

Their validity for $0\leq \ee\leq \ee_{0}(\aaa)$ will be demonstrated in Section \ref{bre}.
\section{A Priori Energy Decay Estimates}\label{apede}
We will now combine the aforementioned results in order to prove the following theorem which shows decay for the degenerate energy $\int_{\Sigma_{\tau}} J^T_{\mu} [\psi ] n^{\mu} d\mi_{g_{\si}}$ under some decay assumptions on the nonlinear term.
\begin{thm}[\textbf{Decay For The Degenerate Energy}]\label{endec}
Assume that the bootstrap assumptions \eqref{A1}, \eqref{A2}, \eqref{A3} hold true for some constant $C_{\aaa} := C(\aaa )$, for some $\ee > 0$ (which is related to the initial data), for any $\tau_0$, $\tau_1$, $\tau_2$ with $\tau_1 < \tau_2$ and for any given $\aaa > 0$.

Then for $\psi$ a spherically symmetric solution of \eqref{nwef} we have:
\begin{equation}\label{end}
\int_{\Sigma_{\tau}} J^T_{\mu} [\psi ] n^{\mu} d\mi_{g_{\si}} \lesssim E_0 \ee^2 (1+\tau)^{-2+\aaa} ,
\end{equation}
for any $\tau$.
\end{thm}
\begin{proof}
We will prove this in two steps. 

First we will show decay for 
$$  \int_{\Sigma_{\tau}} J^P_{\mu} [\psi ] n^{\mu} d\mi_{g_{\si}} , $$
for any $\tau$. This will imply decay for $\int_{\Sigma_{\tau}} J^T_{\mu} [\psi ] n^{\mu} d\mi_{g_{\si}}$ as well, since 
$$ \int_{\Sigma_{\tau}} J^T_{\mu} [\psi ] n^{\mu} d\mi_{g_{\si}} \mik \int_{\Sigma_{\tau}} J^P_{\mu} [\psi ] n^{\mu} d\mi_{g_{\si}} . $$
This decay though won't be the required one. We will upgrade it to the desired one through the use of the integrated local energy decay estimate of Proposition \ref{iled}.

\paragraph{Step 1:}  We examine one by one the terms in the right hand side of Proposition \ref{iledi}. 

First for the $N$-flux we use Proposition \ref{deg4} and we have that for some $0 < \eta < 1$:
$$ \int_{\si_{\tau_1}} J^N_{\mu} [\psi ] n^{\mu} d\mi_{g_{\si}} \lesssim_{R_0} \int_{\si_0} J^N_{\mu} [\psi ] n^{\mu} d\mi_{g_{\si}}  + $$
$$ +  \int_{0}^{\tau_1} \int_{S_{\tau'}} |F|^2 d\mi_{g_S} +  \int_{0}^{\tau_1} \int_{N_{\tau'}} r^{1+\eta} |F|^2 d\mi_{g_N} + \left( \int_{0}^{\tau_1} \left( \int_{S_{\tau'}} |F|^2 d\mi_{g_S} \right)^{1/2} d\tau' \right)^2$$ $$ \lesssim_{R_0} E_0 \ee^2 + E_0 \ee^2 \tau_1 ^{\aaa} ,$$
with the polynomial growth coming from the term $\left( \int_{0}^{\tau_1} \left( \int_{S_{\tau'}} |F|^2 d\mi_{g_S} \right)^{1/2} d\tau' \right)^2$ and our first assumption.

For the term $\int_{N_{\tau_1}} \dfrac{(\partial_v \ph )^2}{r} d\mi_{g_N}$ we have by the $r$-weighted energy inequality of Proposition \ref{rw} with $p=1$ and our assumptions that:
$$ \int_{N_{\tau_1}} \dfrac{(\partial_v \ph )^2}{r} d\mi_{g_N} \lesssim_{R_0} \int_{N_0} \dfrac{( \partial_v \ph )^2}{r} d\mi_{g_N} + \int_{\Sigma_0} J^T_{\mu} [\psi] n^{\mu} d\mi_{g_{\si}}  + $$ $$ + \int_{0}^{\tau_1} \int_{S_{\tau'}} |F|^2 d\mi_{g_{\so}} +  \int_{0}^{\tau_1} \int_{N_{\tau'}} r^2 |F|^2 d\mi_{g_{\nnn}} \lesssim_{R_0} E_0 \ee^2 .$$

For the first two nonlinear terms we use our assumptions and we have that:
$$ \int_{\tau_1}^{\tau_2} \int_{S_{\tau'}} |F|^2 d\mi_{g_{\so}} + \int_{\tau_1}^{\tau_2} \int_{N_{\tau'}} r^2 |F|^2 d\mi_{g_{\nnn}} \lesssim E_0 \ee^2 (1+\tau_1 )^{-2+\aaa} .$$

Finally from the third nonlinear term we get polynomial growth (as it was noticed earlier too):
$$\left( \int_{\tau_1}^{\tau_2} \left( \int_{S_{\tau'}} |F|^2 d\mi_{g_S} \right)^{1/2} d\tau' \right)^2 \lesssim \left( \int_{\tau_1}^{\tau_2} \dfrac{E_0 \ee^2 }{(1+\tau' )^{1-\aaa/2}} d\tau' \right)^2 \lesssim E_0 \ee^2 \tau_2^{\aaa} . $$

Gathering all the above estimates together and going back to estimate \eqref{iled2} we get the following bound for the integrated $P$-flux:
\begin{equation}\label{pinteg}
\int_{\tau_1}^{\tau_2} \int_{\Sigma_{\tau'}} J^P_{\mu} [\psi] n^{\mu} d\mi_{g_{\rrr}} \lesssim_{R_0}  E_0 \ee^2 \tau_2^{\aaa} .
\end{equation}
Estimate \eqref{pinteg} tells us that through a standard contradiction argument we can find a dyadic sequence $\{ \la_n \}_n$, where the following bound holds true:
\begin{equation}\label{pinteg1}
\int_{\Sigma_{\la_n}} J^P_{\mu} [\psi] n^{\mu} d\mi_{g_{\rrr}} \lesssim_{R_0}  E_0 \ee^2 (1+\la_n )^{-1+\aaa} \mbox{ for all $n$} .
\end{equation}
We would like to extend this estimate to every given $\tau$. Any $\tau$ will always satisfy the property that there will be some $k$ such that $\la_k \mik \tau \mik \la_{k+1}$. We use the uniform boundedness estimate for the $P$-flux of Proposition \ref{puniform} for some $0 < \eta < 1$:
\begin{equation}\label{pinteg2}
\int_{\Sigma_{\tau}} J^P_{\mu} [\psi] n^{\mu} d\mi_{g_{\si}} \lesssim_{R_0} \int_{\Sigma_{\la_k}} J^P_{\mu} [\psi] n^{\mu} d\mi_{g_{\si}} + 
\end{equation}
$$  + \int_{\la_k}^{\la_{k+1}} \int_{S_{\tau'}} |F|^2 d\mi_{g_{\so}} +  \int_{\la_k}^{\la_{k+1}} \int_{N_{\tau'}} r^{1+\eta} |F|^2 d\mi_{g_{\nnn}} .$$
For the first term of the right hand side we use estimate \eqref{pinteg1}, while for the last two we use our assumptions, and we arrive at:
\begin{equation}\label{pinteg3}
\int_{\Sigma_{\tau}} J^P_{\mu} [\psi] n^{\mu} d\mi_{g_{\si}} \lesssim_{R_0} E_0 \ee^2 (1+\la_k )^{-1+\aaa} +  
 + E_0 \ee^2 (1+\la_k )^{-2+\aaa} \lesssim_{R_0}
 \end{equation}
 $$ \lesssim_{R_0} E_0 \ee^2 (1+\la_k )^{-1+\aaa} \approx E_0 \ee^2 (1+\la_{k+1} )^{-1+\aaa} \approx E_0 \ee^2 (1+\tau )^{-1+\aaa} \Rightarrow  $$
 \begin{equation}\label{pintegd}
 \int_{\Sigma_{\tau}} J^P_{\mu} [\psi] n^{\mu} d\mi_{g_{\si}} \lesssim_{R_0} E_0 \ee^2 (1+\tau )^{-1+\aaa} .
 \end{equation}

\paragraph{Step 2:} We already have the following decay for the $T$-flux from estimate \eqref{pintegd}:
\begin{equation}\label{tintegb}
\int_{\Sigma_{\tau}} J^T_{\mu} [\psi] n^{\mu} d\mi_{g_{\si}} \lesssim_{R_0} E_0 \ee^2 (1+\tau )^{-1+\aaa} \mbox{ for all $\tau$}.
\end{equation}
We will now improve this to decay of rate $-2+\aaa$. For this we will use the integrated local energy decay estimate of Proposition \ref{iled} for any $\tau_1$, $\tau_2$ with $\tau_1 < \tau_2$. We will examine one by one the terms of the right hand side of estimate \eqref{iled2}.

For the term $\int_{\Sigma_{\tau}} J^P_{\mu} [\psi] n^{\mu} d\mi_{g_{\si}}$ we have estimate \eqref{pintegd}, so we have decay of rate $-1+\aaa$ for this term.

For the last two nonlinear terms of \eqref{iled1} we just use our assumptions and we get for them decay of rate $-2+\aaa$.

For the term $\int_{N_{\tau}} \dfrac{(\partial_v \ph )^2}{r} d\mi_{g_{N}}$ we will now prove decay of rate $-1+\aaa$.

By estimate \eqref{awayvri} we have that for any $\tau > 0$ and for any $\aaa_1 > 0$ such that $2+\aaa_1 \mik 3-\aaa$ the following holds true:
$$
\int_{N_{\tau}} \dfrac{(\partial_v \ph )^2}{r^{1-\aaa_1}} d\mi_{g_{\nnn}} \lesssim_{R_0} \int_{\Sigma_0 } J^T_{\mu} [\psi ] n^{\mu}d\mi_{g_{\si}}  + \int_{\Sigma_0} \dfrac{(\partial_v \ph )^2}{r^{1-\aaa_1}} d\mi_{g_{\si}} + 
$$
$$ + \int_{0}^{\tau} \int_{S_{\tau'}} |F|^2 d\mi_{g_{\so}} + \int_{0}^{\tau} \int_{N_{\tau'}} r^{2+\aaa_1} |F|^2 d\mi_{g_{\nnn}} \lesssim_{R_0} E_0 \ee^2 \Rightarrow $$ 
\begin{equation}\label{pintegint1} 
\int_{N_{\tau}} \dfrac{(\partial_v \ph )^2}{r^{1-\aaa_1}} d\mi_{g_{\nnn}} \lesssim_{R_0} E_0 \ee^2 , 
\end{equation}
where for the nonlinear terms we used our assumptions.

By estimate \eqref{awayvri1} we have that for any $\tau_1$, $\tau_2$ with $\tau_1 < \tau_2$ and any $\aaa_2 > 0$ such that $3-\aaa_2 \mik 3-\aaa$ the following holds true:
$$ \int_{\tau_1}^{\tau_2} \int_{N_{\tau'}} \dfrac{(\partial_v \ph )^2}{r^{1+\aaa_2} }d\mi_{g_{\nnn}} \lesssim_{R_0} \int_{\Sigma_{\tau_1}} J^T_{\mu} [\psi ] n^{\mu}d\mi_{g_{\si}}  + \int_{\Sigma_{\tau_1}} \dfrac{(\partial_v \ph )^2}{r^{\aaa_2}} d\mi_{g_{\si}} + $$
$$ + \int_{\tau_1}^{\tau_2} \int_{S_{\tau'}} |F|^2 d\mi_{g_{\so}} + \int_{\tau_1}^{\tau_2} \int_{N_{\tau'}} r^{3-\aaa_2} |F|^2 d\mi_{g_{\nnn}} \lesssim_{R_0} $$ $$ \lesssim_{R_0}  \int_{\Sigma_{0}} J^T_{\mu} [\psi ] n^{\mu}d\mi_{g_{\si}}  + \int_{\Sigma_{0}} \dfrac{(\partial_v \ph )^2}{r^{\aaa_2}} d\mi_{g_{\si}} + $$
$$ + \int_{0}^{\tau_2} \int_{S_{\tau'}} |F|^2 d\mi_{g_{\so}} + \int_{0}^{\tau_2} \int_{N_{\tau'}} r^{3-\aaa_2} |F|^2 d\mi_{g_{\nnn}} \lesssim_{R_0} E_0 \ee^2 \Rightarrow $$
\begin{equation}\label{pintegint2}
\int_{\tau_1}^{\tau_2} \int_{N_{\tau'}} \dfrac{(\partial_v \ph )^2}{r^{1+\aaa_2} }d\mi_{g_{\nnn}} \lesssim_{R_0} E_0 \ee^2 ,
\end{equation}
where we used Proposition \ref{rw} with $p=2-\aaa_2$, the uniform boundedness of the degenerate energy of Proposition \ref{deg1} and our assumptions.

Estimate \eqref{pintegint2} implies through a standard contradiction argument that there exists a dyadic sequence $\{ \rho_n \}_n$, where the following holds true:
\begin{equation}\label{pintegint3}
\int_{N_{\rho_n}} \dfrac{(\partial_v \ph )^2}{r^{1+\aaa_2} }d\mi_{g_N} \lesssim_{R_0} E_0 \ee^2 (1+\rho_n )^{-1} \mbox{ for all $n$}.
\end{equation}

Now we set $\aaa_1 = 1-\aaa$ and $\aaa_2 = \aaa$ and we have by an interpolation argument on each $N_{\rho_n}$ for all $n$ between estimate \eqref{pintegint3} and estimate \eqref{pintegint1} (with $\tau = \rho_n$) that:
\begin{equation}\label{pintegint}
\int_{N_{\rho_n}} \dfrac{(\partial_v \ph )^2}{r}d\mi_{g_N} \lesssim_{R_0} E_0 \ee^2 (1+\rho_n )^{-1+\aaa} \mbox{ for all $n$}.
\end{equation} 
See also Theorem 1.18.5 of \cite{triebel} for a more general interpolation result that includes the one that we just used as a very special case.

So finally we note that for any given $\tau$ there exists some $k$ such that $\rho_k \mik \tau \mik \rho_{k+1}$ and then we have:
$$\int_{N_{\tau}} \dfrac{(\partial_v \ph )^2}{r} d\mi_{g_{N}} \lesssim_{R_0} \int_{N_{\la_k}} \dfrac{(\partial_v \ph )^2}{r} d\mi_{g_{N}} + \int_{\Sigma_{\rho_k}} J^T_{\mu} [\psi] n^{\mu} d\mi_{g_{\si}} + $$ $$ + \int_{\rho_k}^{\la_{k+1}} \int_{S_{\tau'}} |F|^2 d\mi_{g_{\so}} +  \int_{\rho_k}^{\la_{k+1}} \int_{N_{\tau'}} r^2 |F|^2 d\mi_{g_{\nnn}} \lesssim_{R_0}$$ $$ \lesssim_{R_0} E_0 \ee^2 (1+\rho_k )^{-1+\aaa} + E_0 \ee^2 (1+\rho_k )^{-2+\aaa} \lesssim_{R_0}   E_0 \ee^2 (1+\rho_k )^{-1+\aaa} \approx  $$ $$ \approx E_0 \ee^2 (1+\rho_{k+1} )^{-1+\aaa} \approx  E_0 \ee^2 (1+\tau )^{-1+\aaa} .$$

So in the end we have by gathering all the aforementioned estimates together that:
$$\int_{\tau_1}^{\tau_2} \int_{\Sigma_{\tau'}} J^T_{\mu} [\psi] n^{\mu} d\mi_{g_{\si}} \lesssim_{R_0} E_0 \ee^2 (1+\tau_1 )^{-1+\aaa} + E_0 \ee^2 (1+\tau_1 )^{-2+\aaa} \Rightarrow $$
\begin{equation}\label{tintegbu}
\int_{\tau_1}^{\tau_2} \int_{\Sigma_{\tau'}} J^T_{\mu} [\psi] n^{\mu} d\mi_{g_{\si}} \lesssim_{R_0} E_0 \ee^2 (1+\tau_1 )^{-1+\aaa} .
\end{equation}
Once again through a contradiction argument we can find a dyadic sequence $\{\kappa_n \}_n$, where the following holds:
\begin{equation}\label{tinteg1}
\int_{\Sigma_{\kappa_n}} J^T_{\mu} [\psi] n^{\mu} d\mi_{g_{\si}} \lesssim_{R_0} E_0 \ee^2 (1+\kappa_n )^{-2+\aaa} \mbox{ for all $n$}.
\end{equation}
We note again that for any given $\tau$ there exists some $k$ such that $\kappa_k \mik \tau \mik \kappa_{k+1}$. Applying now estimate \eqref{degi2} for some $0 < \eta < 1$ and using our assumptions we finally get that:
$$ \int_{\Sigma_{\tau}} J^T_{\mu} [\psi] n^{\mu} d\mi_{g_{\si}} \lesssim_{R_0} \int_{\Sigma_{\kappa_k}} J^T_{\mu} [\psi] n^{\mu} d\mi_{g_{\si}} + $$ $$ + \int_{\kappa_k}^{\kappa_{k+1}} \int_{S_{\tau'}} |F|^2 d\mi_{g_{\so}} +  \int_{\kappa_k}^{\kappa_{k+1}} \int_{N_{\tau'}} r^{1+\eta} |F|^2 d\mi_{g_{\nnn}} \lesssim_{R_0} $$ $$ \lesssim_{R_0} E_0 \ee^2 (1+\kappa_k )^{-2+\aaa} \approx E_0 \ee^2 (1+\kappa_{k+1} )^{-2+\aaa} \approx E_0 \ee^2 (1+\tau )^{-2+\aaa} , $$
as required.

\end{proof}

\section{A Priori Pointwise Decay and Boundedness Estimates}\label{apdbe}
\subsection{A Priori Pointwise Decay for $\psi$}\label{appdp1}
Under the same assumptions as in Theorem \ref{endec} we can prove the following pointwise decay result for $\psi$ away from the horizon (something that we will need later).
\begin{thm}[\textbf{Pointwise Decay For $\psi$ Away From The Horizon}]\label{pdec}
Assume that the bootstrap assumptions \eqref{A1}, \eqref{A2}, \eqref{A3} hold true for some constant $C_{\aaa} := C(\aaa )$, for some $\ee > 0$ (which is related to the initial data), for any $\tau_0$, $\tau_1$, $\tau_2$ with $\tau_1 < \tau_2$ and for any given $\aaa > 0$.

Then for $\psi$ a spherically symmetric solution of \eqref{nwef} and any $\tau$ we have that:
\begin{equation}\label{pdec1}
\|\psi  \|_{L^{\infty} ( \si_{\tau} \cap \{ M < r_0 \mik r \} )} \lesssim_{r_0 , \aaa} \dfrac{\sqrt{E_0} \ee }{\sqrt{r} (1+\tau )^{1-\aaa/2}} ,
\end{equation}
for any such $r_0$ and with the constant that is $r_0 , \aaa$-dependent becoming singular as $r_0 \rightarrow M$ or $\aaa \rightarrow 0$.
\end{thm}
\begin{proof}
For any function we have the following inequality (inequality (6.1) in \cite{A2}) just by the Fundamental Theorem of Calculus:
\begin{equation}\label{pdec3}
\int_{\mathbb{S}^2} \psi^2 (r_0 , \omega ) d\omega \mik \dfrac{c}{r_0} \int_{\Sigma_{\tau} \cap \{ r \meg r_0 \}} J^N_{\mu} [\psi ] n^{\mu} d\mi_{g_{\si}} ,
\end{equation}
for $c = c (M, \Sigma_0 )$.

Since away from the horizon we have that 
$$ J^N_{\mu} [\psi ] n^{\mu} \approx J^T_{\mu} [\psi ] n^{\mu} ,$$
we use the result of Theorem \ref{endec} (because of our assumptions) and \eqref{pdec3} becomes for any $r \meg r_0$:
\begin{equation}\label{pdec4}
\int_{\mathbb{S}^2} \psi^2 (r , \omega ) d\omega \lesssim \dfrac{ E_0 \ee^2 }{r (1+ \tau)^{2-\aaa}}  .
\end{equation}
The required result \eqref{pdec1} follows by the assumption of spherical symmetry on $\psi$.
\end{proof}
The aforementioned result can be improved in terms of the $r$ decay, although this will have the effect of worsening the decay with respect to $\tau$.
\begin{thm}[\textbf{Improved Pointwise Decay For $\psi$ Away From The Horizon With Respect To $r$}]\label{pdecr}
Assume that the bootstrap assumptions \eqref{A1}, \eqref{A2}, \eqref{A3} hold true for some constant $C_{\aaa} := C(\aaa )$, for some $\ee > 0$ (which is related to the initial data), for any $\tau_0$, $\tau_1$, $\tau_2$ with $\tau_1 < \tau_2$ and for any given $\aaa > 0$.

Then for $\psi$ a spherically symmetric solution of \eqref{nwef} and any $\tau$ we have that:
\begin{equation}\label{pdecr1}
\| \psi  \|_{L^{\infty} ( \si_{\tau} \cap \{ M < r_0 \mik r \} )} \lesssim_{r_0 , \aaa} \dfrac{\sqrt{E_0} \ee }{r^{1-\aaa} (1+\tau )^{1/2-\aaa/4}} ,
\end{equation}
for any such $r_0$ and with the constant that is $r_0,\aaa$-dependent becoming singular as $r_0 \rightarrow M$ or $\aaa\rightarrow 0$.
\end{thm}
\begin{proof}
We work as before, applying the Fundamental Theorem of Calculus to $r^{1-\aaa} \psi$ this time, for a  fixed $r_0 > M$, $r \meg r_0$, and for any $\tau$:
$$ \left( r^{1-\aaa} \psi \right)^2 (r, \tau) = \left( r_0^{1-\aaa} \psi \right)^2 (r_0 , \tau ) + 2\int_{r_0}^r \rho^{1-\aaa} \psi \partial_{\rho} \left( \rho^{1-\aaa} \psi \right) d\rho \Rightarrow $$ $$ \Rightarrow \left( r^{1-\aaa} \psi \right)^2 \mik \left( r_0^{1-\aaa} \psi \right)^2 (r_0 ,\tau ) + \left( \int_{\si_{\tau}} \dfrac{\psi^2}{\rho^2} d\mi_{g_{\si}} \right)^{1/2} \left( \int_{\si_{\tau}} \dfrac{\left( \partial_{\rho} (\rho \psi ) \right)^2}{\rho^{4\aaa}} d\mi_{g_{\si}} \right)^{1/2} \lesssim $$ $$\lesssim \left( r_0^{1-\aaa} \psi \right)^2 (r_0 , \tau) + \left( \int_{\si_{\tau}} J^T_{\mu} [\psi ] n^{\mu} d\mi_{g_{\si}} \right)^{1/2} \left( \int_{\si_{\tau}} \dfrac{\left( \partial_{\rho} (\rho \psi ) \right)^2}{\rho^{4\aaa}} d\mi_{g_{\si}} \right)^{1/2} , $$
where in the second line we used the assumption of spherical symmetry and that
$$ 2\int_{r_0}^r \rho^{2-\aaa} \partial_{\rho} (\rho^{-\aaa} ) \psi^2  d\rho \mik 0 ,$$ 
and in the third line we used Hardy's inequality \eqref{hardy}. By Proposition \ref{pdec} applied to the term $\left( r_0^{1-\aaa} \psi \right)^2 (r_0 , \tau)$, the fact that the $T$-flux decays as $\tau^{-2+\aaa}$ under our assumptions (as it was shown in Theorem \ref{endec}), and the boundedness of the term $\left( \int_{\si_{\tau}} \dfrac{\left( \partial_{\rho} (\rho \psi ) \right)^2}{\rho^{4\aaa}} d\mi_{g_{\si}} \right)^{1/2}$ by Proposition \ref{rw} with $p= 2-4\aaa$ and our assumptions (since $3-4\aaa < 3-\aaa$) we obtain the following:
$$ \left( r^{1-\aaa} \psi \right)^2 (r, \tau) \lesssim E_0 \ee^2 (1+\tau )^{-2+\aaa} + E_0 \ee^2 (1+\tau )^{-1+\aaa/2} \Rightarrow $$ $$ \Rightarrow \left( r^{1-\aaa} \psi \right)^2 (r, \tau) \lesssim E_0 \ee^2 (1+\tau )^{-1+\aaa/2} , $$
as desired.

\end{proof}

Unfortunately the decay rate of $-1+\aaa$ can't be extended to the whole domain of outer communications. This is a novel aspect of our analysis on extremal black holes. 

\begin{thm}[\textbf{First Global Pointwise Decay Result For $\psi$}]\label{dpsi}
Assume that the bootstrap assumptions \eqref{A1}, \eqref{A2}, \eqref{A3} hold true for some constant $C_{\aaa} := C(\aaa )$, for some $\ee > 0$ (which is related to the initial data), for any $\tau_0$, $\tau_1$, $\tau_2$ with $\tau_1 < \tau_2$ and for any given $\aaa > 0$.

Then for $\psi$ a spherically symmetric solution of \eqref{nwef} we have:
\begin{equation}\label{dw}
\|\psi \|_{L^{\infty} (\Sigma_{\tau} )} \lesssim_{\aaa} \dfrac{\sqrt{E_0} \ee}{(1+\tau )^{1/2 - \aaa/2}} ,
\end{equation}
for any $\tau$.
\end{thm}
\begin{proof}
Using the previously stated inequality \eqref{pdec3} and applying Theorem \ref{endec} (which holds true because of our assumptions) we have that for any $M < r' < R_0$ the following holds:
$$ \int_{\mathbb{S}^2} \psi^2 (r' , \omega) d\omega \lesssim \dfrac{1}{(r' - M)^2} \int_{\Sigma_{\tau}} J^T_{\mu} [\psi ] n^{\mu} d\mi_{g_{\sis}} \Rightarrow $$ 
\begin{equation}\label{p1}
\int_{\mathbb{S}^2} \psi^2 (r' , \omega) d\omega \lesssim \dfrac{E_0 \ee^2}{(r' - M)^2 (1+\tau )^{2-\aaa} }.
\end{equation}

Using now Stokes' Theorem in the region $\rrr (\tau_1 , \tau_2 ) \cap \{ r_c \mik r \mik r_c +1 \}$ for $r_c \meg M$ (we also assume that $r_c + 1 < R_0$ without loss of generality), Cauchy-Schwarz, the uniform boundedness of the non-degenerate energy \eqref{degi4}, Theorem \ref{endec}, our assumptions and Hardy's inequality \eqref{hardy} we arrive at the following estimate for any $\tau$:
$$ \int_{\mathbb{S}^2} \psi^2 (r_c , \omega) d\omega \mik \int_{\mathbb{S}^2} \psi^2 (r_c + 1 , \omega) d\omega + \int_{\Sigma_{\tau} \cap \{ r_c \mik r \mik r_c + 1 \}} \psi \partial_{\rho} \psi d\mi_{g_{\si}} \lesssim $$ $$ \lesssim  \dfrac{E_0 \ee^2}{ (1+\tau )^{2-\aaa}} + \dfrac{\sqrt{E_0} \ee}{ (1+\tau )^{1-\aaa/2} } \left( \int_{\Sigma_{0}} J^N_{\mu} [\psi ] n^{\mu} d\mi_{g_{\si}} +  \int_{0}^{\tau} \int_{S_{\tau'}} |F|^2 d\mi_{g_S} + \right. $$ $$\left. + \int_{0}^{\tau} \int_{N_{\tau'}} r^{1+\eta} |F|^2 d\mi_{g_N} + \left(\int_{0}^{\tau} \left( \int_{S_{\tau'}} |F|^2 d\mi_{S_{\tau'}} \right)^{1/2} d\tau'\right)^{2} \right)^{1/2} \Rightarrow $$
$$
\int_{\mathbb{S}^2} \psi^2 (r' , \omega) d\omega \lesssim  \dfrac{E_0 \ee^2}{ (1+ \tau )^{2-\aaa} } + \dfrac{\sqrt{E_0} \ee}{ (1+ \tau )^{1-\aaa/2} } \left( E_0 \ee^2 + E_0 \ee^2 \tau^{\aaa} \right)^{1/2} 
$$
\begin{equation}\label{p2}
 \Rightarrow \int_{\mathbb{S}^2} \psi^2 (r' , \omega) d\omega \lesssim \dfrac{E_0 \ee^2}{ (1+ \tau )^{1 - \aaa} } .
 \end{equation}
 
Estimate \eqref{p2} gives us in the end the required decay once we assume that $\psi$ is spherically symmetric.

\end{proof}

\subsection{A Priori Pointwise Boundedness For $Y \psi$}\label{drbound}
We will now prove a priori boundedness results for $Y \psi$. 
\begin{thm}[\textbf{Boundedness For $Y\psi$}]\label{dr}
There exists some $\ee_r : = \ee_r (\aaa ) > 0$ such that if $0 < \ee < \ee_r$ and the bootstrap assumptions \eqref{A1}, \eqref{A2}, \eqref{A3} hold true for $\psi$ a spherically symmetric solution of \eqref{nwef}, for some constant $C_{\aaa} := C(\aaa )$, for any $\tau_0$, $\tau_1$, $\tau_2$ with $\tau_1 < \tau_2$ and for any given $\aaa > 0$, then we have that:
\begin{equation}\label{drr}
\| Y \psi  \|_{L^{\infty} (\Sigma_{\tau})} \lesssim \sqrt{E_0} \ee ,
\end{equation}
\begin{equation}\label{rdecpsi1}
| r \psi | \lesssim_{\bar{\aaa}} \sqrt{E_0} \ee r^{\bar{\aaa}} ,
\end{equation}
\begin{equation}\label{rdecdv1}
| r^2 \partial_v \psi | \lesssim_{\bar{\aaa}} \sqrt{E_0} \ee r^{\bar{\aaa}} ,
\end{equation}
for any $\bar{\aaa} > 0$ and for all $\tau$, where the $\bar{\aaa}$ dependent constant becomes singular (like $1/\bar{\aaa}$) as $\bar{\aaa} \rightarrow 0$.
\end{thm}

Before giving the proof of this Theorem we will record the equations for certain quantities of interest for this Section. These will be derived in the abstract null coordinate system and in the region $\mathcal{D}$ (which can be seen as a subset of the domain of outer communications which includes a part -- and not all -- of future null infinity) that are described in Appendix \ref{ngrn}.

In this coordinate system we define the following quantities:
$$ \la = \partial_v r, \quad \nu = \partial_u r, \quad \zeta = r \partial_u \psi, \quad \theta = r \partial_v \psi .$$
A computation shows the following:
\begin{equation}\label{zeta}
 \partial_v \zeta = -\frac{\theta \nu}{r} - \frac{\Omega^2}{4} rF ,
 \end{equation}
\begin{equation}\label{theta}
\partial_u \theta = -\frac{\zeta \lambda}{r} - \frac{\Omega^2}{4} rF.
\end{equation}
Recall also that in this coordinate system we have that:
$$ F = A(\psi ) g^{\alpha \beta} \partial_{\alpha} \psi \partial_{\beta} \psi = -\frac{4}{\Omega^2} A(\psi ) \partial_u \psi \partial_v \psi ,$$
so the aforementioned equations \eqref{zeta} and \eqref{theta} are consequences of the equation:
\begin{equation}\label{nwefnull}
r \partial_{uv} \psi = -\partial_u r \partial_v \psi - \partial_v r \partial_u \psi + r A(\psi ) \partial_u \psi \partial_v \psi .
\end{equation}

We notice first that the quantity $Y\psi$ can be written as follows:
\begin{equation}\label{ypsi}
Y\psi = \dfrac{1}{r} \dfrac{\zeta}{\nu} .
\end{equation}

We examine the quantity $\dfrac{\zeta}{\nu}$ since proving boundedness for this will imply boundedness (and decay with respect to $r$ additionally) for $Y\psi$. We have:
$$ \partial_v \left( \frac{\zeta}{\nu} \right) = \frac{\partial_v \zeta}{\nu} - \frac{\zeta}{\nu} \frac{\partial_v \nu}{\nu} .$$
For this last ODE, by noticing first that $e^{-\int_{\bar{v}}^{v} \frac{\partial_{v'} \nu}{\nu} dv'} = \dfrac{\nu (u,\bar{v})}{\nu(u,v)}$, we obtain the following solution:
\begin{equation}\label{zetan}
\frac{\zeta}{\nu} (u,v) = \frac{\zeta}{\nu} (u, 0 ) \dfrac{\nu (u,0)}{\nu (u,v)} - \int_0^v \partial_{\bar{v}} \psi \dfrac{\nu (u,\bar{v})}{\nu(u,v)} d\bar{v}  + \dfrac{1}{\nu(u,v)} \int_0^v r A(\psi) \partial_u \psi \partial_{\bar{v}} \psi d\bar{v} .
\end{equation}
We look at the last term that comes from the nonlinearity:
\begin{equation}\label{zetan11}
 \dfrac{1}{\nu(u,v)} \int_0^v r A(\psi) \partial_u \psi \partial_{\bar{v}} \psi d\bar{v} = \dfrac{1}{\nu(u,v)} \int_0^v \partial_{\bar{v}} ( r \psi ) A(\psi) \partial_u \psi d\bar{v} - 
 \end{equation}
  $$ - \dfrac{1}{\nu(u,v)} \int_0^v \partial_{\bar{v}} r A(\psi) \partial_u \psi  \cdot \psi d\bar{v} .$$
We look at the two terms of the equality given above separately.
$$ \dfrac{1}{\nu(u,v)} \int_0^v \partial_{\bar{v}} ( r \psi ) A(\psi) \partial_u \psi d\bar{v} = \dfrac{1}{\nu (u,v)} r A(\psi ) \partial_u \psi \cdot \psi |_0^v - $$ $$ - \dfrac{1}{\nu(u,v)} \int_0^v r A' (\psi) \partial_{\bar{v}} \psi \partial_u \psi \cdot \psi d\bar{v}  - \dfrac{1}{\nu(u,v)} \int_0^v r \partial_{u\bar{v}} \psi A(\psi) \psi d\bar{v} .$$
We look now at:
$$ - \dfrac{1}{\nu(u,v)} \int_0^v \partial_{\bar{v}} r A(\psi) \partial_u \psi  \psi d\bar{v} = - \dfrac{1}{\nu(u,v)} \int_0^v \partial_{\bar{v}} ( r \partial_u \psi ) A(\psi) \psi d\bar{v} + $$ $$ + \dfrac{1}{\nu(u,v)} \int_0^v  r \partial_{u \bar{v}} \psi  A(\psi) \psi d\bar{v} .$$
We add the right hand sides of the last two equalities and we turn \eqref{zetan11} into the following:
\begin{equation}\label{null1}
 \dfrac{1}{\nu (u,v)} \int_0^v r A(\psi) \partial_u \psi \partial_{\bar{v}} \psi d\bar{v} = \dfrac{1}{\nu (u,v)} r A(\psi ) \partial_u \psi \cdot \psi |_0^v - 
\end{equation} 
$$ - \dfrac{1}{\nu (u,v)} \int_0^v r A' (\psi) \partial_{\bar{v}} \psi \partial_u \psi \cdot \psi d\bar{v}  - \dfrac{1}{\nu(u,v)} \int_0^v \partial_{\bar{v}} ( r \partial_u \psi ) A(\psi) \psi d\bar{v} .$$
Now for the second term of the equality \eqref{null1} given above we have:
$$ - \dfrac{1}{\nu(u,v)} \int_0^v r A' (\psi) \partial_{\bar{v}} \psi \partial_u \psi \cdot \psi d\bar{v} = - \dfrac{1}{2\nu(u,v)} \int_0^v r A' (\psi ) \partial_u \psi \partial_{\bar{v}} (\psi^2 ) d\bar{v} = $$ $$ = - \dfrac{1}{2\nu(u,v)} r A' (\psi ) \partial_u \psi \cdot \psi^2 |_0^v + \dfrac{1}{2\nu(u,v)} \int_0^v \partial_{\bar{v}} r A' (\psi ) \partial_u \psi \cdot \psi^2 d\bar{v} + $$ $$ + \dfrac{1}{2\nu(u,v)} \int_0^v r A'' (\psi ) \partial_u \psi \partial_{\bar{v}} \psi \cdot \psi^2 d\bar{v} + \dfrac{1}{2\nu(u,v)} \int_0^v r A' (\psi ) \partial_{u \bar{v}} \psi \cdot \psi^2 d\bar{v} .$$
For the last term of the equality given above we use the equation \eqref{nwefnull}:
$$ \dfrac{1}{2\nu(u,v)} \int_0^v r A' (\psi ) \partial_{u \bar{v}} \psi \cdot \psi^2 d\bar{v} = - \dfrac{1}{2\nu(u,v)} \int_0^v \partial_u r A' (\psi ) \partial_{\bar{v}} \psi \cdot \psi^2 d\bar{v} - $$ $$ - \dfrac{1}{2\nu(u,v)} \int_0^v \partial_{\bar{v}} r A' (\psi ) \partial_u \psi \cdot \psi^2 d\bar{v} + \dfrac{1}{2\nu(u,v)} \int_0^v  r A(\psi ) A' (\psi ) \partial_u \psi \partial_{\bar{v}} \psi \cdot \psi^2 d\bar{v} .$$

Having computed all the sub-terms of the second term of \eqref{null1}, we turn now to the third and last term of this equality. We use \eqref{zeta}:
$$ - \dfrac{1}{\nu(u,v)} \int_0^v \partial_{\bar{v}} ( r \partial_u \psi ) A(\psi) \psi d\bar{v} = \dfrac{1}{\nu(u,v)} \int_0^v \partial_u r A(\psi ) \partial_{\bar{v}} \psi \cdot \psi d\bar{v} - $$ $$ - \dfrac{1}{\nu(u,v)} \int_0^v  r A^2 (\psi ) \partial_u \psi \partial_{\bar{v}} \psi \cdot \psi d\bar{v} .$$

For the second term we have:
$$  \dfrac{1}{\nu(u,v)} \int_0^v  r A^2 (\psi ) \partial_u \psi \partial_{\bar{v}} \psi \cdot \psi d\bar{v} =  \dfrac{1}{2 \nu(u,v)} \int_0^v  r A^2 (\psi ) \partial_u \psi \partial_{\bar{v}} ( \psi^2 ) d\bar{v} . $$

Gathering all the terms together we arrive at the following:
$$ \dfrac{1}{\nu (u,v)} \int_0^v r A(\psi) \partial_u \psi \partial_{\bar{v}} \psi d\bar{v} = \dfrac{1}{\nu (u,v)} r A(\psi ) \partial_u \psi \cdot \psi |_0^v - \dfrac{1}{2\nu(u,v)} r A' (\psi ) \partial_u \psi \cdot \psi^2 |_0^v + $$ $$ + \dfrac{1}{2\nu(u,v)} \int_0^v \partial_{\bar{v}} r A' (\psi ) \partial_u \psi \cdot \psi^2 d\bar{v}  + \dfrac{1}{2\nu(u,v)} \int_0^v r A'' (\psi ) \partial_u \psi \partial_{\bar{v}} \psi \cdot \psi^2 d\bar{v} - $$ $$ - \dfrac{1}{2\nu(u,v)} \int_0^v \partial_u r A' (\psi ) \partial_{\bar{v}} \psi \cdot \psi^2 d\bar{v} - \dfrac{1}{2\nu(u,v)} \int_0^v \partial_{\bar{v}} r A' (\psi ) \partial_u \psi \cdot \psi^2 d\bar{v} + $$ $$ + \dfrac{1}{2\nu(u,v)} \int_0^v  r A(\psi ) A' (\psi ) \partial_u \psi \partial_{\bar{v}} \psi \cdot \psi^2 d\bar{v} + \dfrac{1}{\nu(u,v)} \int_0^v \partial_u r A(\psi ) \partial_{\bar{v}} \psi \cdot \psi d\bar{v} - $$ $$ - \dfrac{1}{2 \nu(u,v)} \int_0^v  r A^2 (\psi ) \partial_u \psi \partial_{\bar{v}} ( \psi^2 ) d\bar{v} =$$ 
\begin{equation}\label{zetan2}
= \dfrac{1}{\nu (u,v)} r A(\psi ) \partial_u \psi \cdot \psi |_0^v - \dfrac{1}{2\nu(u,v)} r A' (\psi ) \partial_u \psi \cdot \psi^2 |_0^v + 
\end{equation}
 $$ +\dfrac{1}{2\nu(u,v)} \int_0^v r A'' (\psi ) \partial_u \psi \partial_{\bar{v}} \psi \cdot \psi^2 d\bar{v} -  \dfrac{1}{2\nu(u,v)} \int_0^v \partial_u r A' (\psi ) \partial_{\bar{v}} \psi \cdot \psi^2 d\bar{v} +$$ $$ + \dfrac{1}{2\nu(u,v)} \int_0^v  r A(\psi ) A' (\psi ) \partial_u \psi \partial_{\bar{v}} \psi \cdot \psi^2 d\bar{v} + \dfrac{1}{\nu(u,v)} \int_0^v \partial_u r A(\psi ) \partial_{\bar{v}} \psi \cdot \psi d\bar{v} - $$ $$ - \dfrac{1}{2 \nu(u,v)} \int_0^v  r A^2 (\psi ) \partial_u \psi \partial_{\bar{v}} ( \psi^2 ) d\bar{v} .$$  

We look now at the last two terms above. First we have by integrating by parts once more:
$$ \dfrac{1}{\nu(u,v)} \int_0^v \partial_u r A(\psi ) \partial_{\bar{v}} \psi \cdot \psi d\bar{v} =\dfrac{1}{2\nu(u,v)} \int_0^v \partial_u r A(\psi ) \partial_{\bar{v}} (\psi^2 ) d\bar{v} = $$ $$ =\dfrac{1}{2\nu(u,v)} \partial_u r A(\psi ) \psi^2 |_0^v - \dfrac{1}{2\nu(u,v)} \int_0^v \partial_u r A' (\psi ) \partial_{\bar{v}} \psi \cdot \psi^2 d\bar{v} - $$ $$ - \dfrac{1}{2\nu(u,v)} \int_0^v \partial_{u\bar{v}} r A(\psi ) \psi^2 d\bar{v} .$$ 
And then we have for the last term of \eqref{zetan2}:
$$ - \dfrac{1}{2 \nu(u,v)} \int_0^v  r A^2 (\psi ) \partial_u \psi \partial_{\bar{v}} ( \psi^2 ) d\bar{v} = $$ $$ = - \dfrac{1}{2 \nu(u,v)} r A^2 (\psi ) \partial_u \psi \cdot \psi^2 |_0^v + \dfrac{1}{2 \nu(u,v)} \int_0^v \partial_{\bar{v}} r A^2 (\psi) \partial_u \psi \cdot \psi^2 d\bar{v} + $$ $$ + \dfrac{1}{ \nu(u,v)} \int_0^v r A' (\psi ) A(\psi ) \partial_u \psi \partial_v \psi \cdot \psi^2 d\bar{v} + \dfrac{1}{2 \nu(u,v)} \int_0^v r A^2 (\psi ) \partial_{u\bar{v}} \psi \cdot \psi^2 d\bar{v} .$$
For the last term of the above equality we use the equation \eqref{nwefnull}:
$$ \dfrac{1}{2 \nu(u,v)} \int_0^v r A^2 (\psi ) \partial_{u\bar{v}} \psi \cdot \psi^2 d\bar{v} = - \dfrac{1}{2 \nu(u,v)} \int_0^v  A^2 (\psi ) \partial_u r \partial_{\bar{v}} \psi \cdot \psi^2 d\bar{v} - $$ $$ - \dfrac{1}{2 \nu(u,v)} \int_0^v  A^2 (\psi ) \partial_{\bar{v}} r \partial_u \psi \cdot \psi^2 d\bar{v} + \int_0^v r A^3 (\psi ) A' (\psi ) \partial_u \psi \partial_{\bar{v}} \psi \cdot \psi^2 d\bar{v} .$$
Once more we gather all the terms for $\dfrac{1}{\nu (u,v)} \int_0^v r A(\psi) \partial_u \psi \partial_{\bar{v}} \psi d\bar{v}$ and we have:
\begin{equation}\label{nlnull}
\dfrac{1}{\nu (u,v)} \int_0^v r A(\psi) \partial_u \psi \partial_{\bar{v}} \psi d\bar{v} =
\end{equation}
$$ = \dfrac{1}{\nu (u,v)} r A(\psi ) \partial_u \psi \cdot \psi |_0^v - \dfrac{1}{2\nu(u,v)} r A' (\psi ) \partial_u \psi \cdot \psi^2 |_0^v + $$ $$ + \dfrac{1}{2\nu(u,v)} \partial_u r A(\psi ) \psi^2 |_0^v - \dfrac{1}{2 \nu(u,v)} r A^2 (\psi ) \partial_u \psi \cdot \psi^2 |_0^v + $$ $$ +\dfrac{1}{2\nu(u,v)} \int_0^v \left(A'' (\psi) + 3 A(\psi ) A' (\psi ) + A^3 (\psi ) \right) r \partial_u \psi \partial_{\bar{v}} \psi \cdot \psi^2 d\bar{v} -  $$ $$ - \dfrac{1}{\nu(u,v)} \int_0^v \left( A' (\psi ) +\dfrac{A^2 (\psi)}{2} \right) \partial_u r \partial_{\bar{v}} \psi \cdot \psi^2 d\bar{v}  - \dfrac{1}{2\nu(u,v)} \int_0^v \partial_{u\bar{v}} r A(\psi ) \psi^2 d\bar{v}  := $$ $$ := I + II + III + IV + V + VI +VII .$$ 
At this point we point out that we have the following as well:
$$ - \int_0^v \partial_{\bar{v}} \psi \dfrac{\nu (u,\bar{v})}{\nu(u,v)} d\bar{v} = -\psi (u,v) + \psi (0,v ) \dfrac{\nu (u,0)}{\nu (u,v) } + \dfrac{1}{\nu (u,v)}\int_0^v \psi (u, \bar{v} ) \partial_{\bar{v}} \nu d\bar{v} .$$
Going back to \eqref{zetan} we obtain the following:
\begin{equation}\label{zetan1}
\frac{\zeta}{\nu} (u,v) = \frac{\zeta}{\nu} (u, 0 ) \dfrac{\nu (u,0)}{\nu (u,v)} - \psi (u,v) + \psi (0,v ) \dfrac{\nu (u,0)}{\nu (u,v) } +\dfrac{1}{\nu (u,v)}\int_0^v \psi (u, \bar{v} ) \partial_{\bar{v}} \nu d\bar{v}+
\end{equation}
$$ + I + II + III + IV + V + VI +VII .$$ 

Now we turn to another quantity that will be useful for us in order to prove Theorem \ref{dr}. From equation \eqref{nwefnull} and equation \eqref{ruv1} we notice that we have:
\begin{equation}\label{intupartialv}
\partial_u (\theta + \la \cdot \psi ) = \dfrac{2\nu \psi}{r^2} \left( M - \frac{M^2}{r} \right) + r A(\psi ) \partial_u \psi \partial_v \psi .
\end{equation}
We integrate equation \eqref{intupartialv} with respect to $u$ and we get that for any $v$ the following holds true:
\begin{equation}\label{upartialv1}
 \theta (u,v) + \la (u,v) \cdot \psi (u,v) = \theta (0,v) + \la (0,v) \cdot \psi (0,v) + 
\end{equation}
 $$ + \int_0^u \left[ \dfrac{2\nu \psi}{r^2} \left( M - \frac{M^2}{r}\right) \right] d\bar{u} + \int_0^u r A(\psi ) \partial_{\bar{u}} \psi \partial_v \psi d\bar{u} .$$

\begin{proof}[Proof of Theorem \ref{dr}]

We will prove a bound for $\dfrac{\zeta}{\nu}$
by a bootstrap argument. This will give us the desired result because of \eqref{ypsi}.

We will also prove the boundedness of $\partial_v \psi $ by a bootstrap argument as well, since this will be needed for proving the boundedness of $\dfrac{\zeta}{\nu}$. Additionally we prove a bound for $r\psi$, improving in some sense the estimate that we obtained in Theorem \ref{pdecr}. 

By our condition on the initial data, we can assume that for a constant $B$ we have the following bounds:
\begin{equation}\label{maxas1}
 \max_{0\mik u \mik U}\left| \frac{\zeta}{\nu} (u,0) \right| + \sup_{0\mik v < \infty} \left| \frac{\zeta}{\nu} (0,v) \right| \mik B \sqrt{E_0}  \ee ,   
 \end{equation}
\begin{equation}\label{maxas2}
\max_{0\mik u \mik U}\left| r^2 (u,0) \cdot \partial_v \psi (u,0) \right| + \sup_{0\mik v < \infty} \left| r^2 (0,v) \cdot \partial_v \psi (0,v) \right| \mik B \sqrt{E_0}  \ee r^{\bar{\aaa}} ,
\end{equation} 
\begin{equation}\label{maxas3}
\max_{0\mik u \mik U}\left| r (u,0) \cdot\psi (u,0) \right| + \sup_{0\mik v < \infty} \left| r (0,v) \cdot \psi (0,v) \right| \mik B \sqrt{E_0}  \ee r^{\bar{\aaa}} ,
\end{equation}
\begin{equation}\label{maxas4}
\max_{0\mik u \mik U}\left| \partial_v r (u,0) \cdot \psi (u,0) + r (u,0) \cdot\psi (u,0) \right| + 
\end{equation}
$$ + \sup_{0\mik v < \infty} \left|\partial_v r (0,v) \cdot \psi (0,v) + r (0,v) \cdot \psi (0,v) \right| \mik B \sqrt{E_0}  \ee \dfrac{1}{r^2} , $$
for $\bar{\aaa} = \bar{\aaa} (\aaa)$ small enough to be chosen later.
  
We define $\mathcal{B}$ as the subset of $\mathcal{D}$ having the following form:
$$ \mathcal{B} = \left\{ ( \bar{u} , \bar{v} ) \in \mathcal{D} : \left| \frac{\zeta}{\nu} (\bar{u} , \bar{v} ) \right|\mik 2 B_1 B \sqrt{E_0} \ee,  | r^2 \partial_v \psi |, |r\psi | \mik \frac{2}{\bar{\aaa}} \sqrt{E_0} \ee r^{\bar{\aaa}} \right\} ,$$
for some constants $B_0 , B_1 > 100$.   

By estimates \eqref{maxas1}, \eqref{maxas2}, \eqref{maxas3} we have that $\mathcal{B}$ is non-empty. By the continuity of $\dfrac{\zeta}{\nu}$, $\partial_v \psi$ and $\psi$ it is also closed in $\mathcal{D}$. 

We will now show the following: fix some $v' > 0$ such that for all $(u' , v'' ) \in \mathcal{D}$, $0 \mik v'' \mik v'$ we have that
\begin{equation}\label{bootassummain} 
 \left| \frac{\zeta}{\nu} (u' , v'' ) \right| \mik 2 B_1 B  \sqrt{E_0}  \ee , \quad |r^2 (u' , v'') \cdot \partial_v \psi (u', v'' )| \mik \frac{2}{\bar{\aaa}} B_1 B  \sqrt{E_0}  \ee r^{\bar{\aaa}} , 
 \end{equation}
 $$ \quad |r (u' , v'') \cdot \psi (u', v'' )| \mik \frac{2}{\bar{\aaa}} B_1 B  \sqrt{E_0}  \ee r^{\bar{\aaa}} . $$
 We will now prove that in fact we have the improved estimate 
 $$\left| \frac{\zeta}{\nu} (u' , v'' ) \right|  \mik \frac{3}{2} B_1 B \sqrt{E_0} \ee , \quad  |r^2 (u' , v'' ) \cdot \partial_v \psi (u', v'' )| \mik   \frac{3}{2\bar{\aaa}} B_1 B \sqrt{E_0} \ee r^{\bar{\aaa}} , $$ $$ \quad |r (u' , v'' ) \cdot  \psi (u', v'' )| \mik   \frac{1}{\bar{\aaa}} B_1 B \sqrt{E_0} \ee r^{\bar{\aaa}} ,$$
  for all  $(u' , v'' )$ as before which will give us that $\mathcal{B}$ is in fact open in $\mathcal{D}$, a fact that allows us to conclude that $\mathcal{D} = \mathcal{B}$ by the connectedness property of both sets.
 
As a remark it should be noted that in our coordinate system we can state the decay estimate for $\psi$ in terms of decay in the $v$ coordinate. For all $(u,v) $ for which we have boundedness for $Y \psi$ we can improve estimate \eqref{dpsi} (as we will show in the subsequent subsection \ref{improvedpsi}) to the following decay estimate: 
\begin{equation}\label{deenull}
|\psi(u,v)| \lesssim \dfrac{\sqrt{E_0} \ee}{(1+v )^{3/5 - \aaa_0 }} ,
\end{equation}
for $\aaa_0 = \dfrac{3\aaa}{10}$ where $\aaa$ is chosen so that
$$ \dfrac{3}{5} - \dfrac{3\aaa}{10} > \dfrac{1}{2} . $$

We can choose $B$ such that due to the assumption \eqref{bootassummain} we can turn estimate \eqref{deenull} into the following:
\begin{equation}\label{denull}
|\psi(u,v)| \mik \frac{B \sqrt{E}_0  \ee}{100} .
\end{equation}

First we will close the bootstrap assumptions for $r^2 \partial_v \psi$ and $r\psi$. We use equation \eqref{upartialv1}. Taking absolute values in this equation we obtain the following estimate for any $0 \mik v'' \mik v'$ and $u' \in [0,U]$:
$$ |r(u' , v'' ) \cdot \partial_v \psi (u' , v'' ) + \partial_v r (u' , v'' ) \cdot \psi (u' , v'' )| \mik $$ $$ \mik |r(0 , v'' ) \cdot \partial_v \psi (0 , v'' ) + \partial_v r (0, v'' ) \cdot \psi (0 , v'' )| + \int_0^{u'} \left| \dfrac{2\nu \psi}{r^2} \left( M - \dfrac{M^2}{r} \right) \right| du + $$ $$ + \int_0^{u'} |A(\psi)| \left| \dfrac{\zeta}{\nu} \right| \left| \dfrac{r^2 \partial_v \psi }{r^2} \right| |\nu| du \mik $$ $$ \mik |r(0 , v'' ) \cdot \partial_v \psi (0 , v'' ) + \partial_v r (0, v'' ) \cdot \psi (0 , v'' )| + 2M \sqrt{E_0} \ee \int_0^{u'} - \dfrac{\partial_u r }{r^{3-\bar{\aaa}}} du + $$ $$ + 4 a_0 B_1^2 B^2 E_0 \ee^2 \int_0^{u'} - \dfrac{\partial_u r }{r^{2-\bar{\aaa}}} du \mik $$ $$ B \sqrt{E_0} \ee \dfrac{1}{r^2 (0,v'')} + 4 M \sqrt{E_0} \ee \dfrac{1}{r^{2-\bar{\aaa}} (u' , v'' )} +    8 a_0 B_1^2 B^2 E_0 \ee^2 \dfrac{1}{r^{1-\bar{\aaa}} (u' , v'' )} \mik $$ $$ \mik \dfrac{1}{2} B_1 B \sqrt{E_0} \ee \dfrac{1}{r^{1-\bar{\aaa}} (u' , v'' )} , $$
where we used assumptions \eqref{maxas4} and \eqref{bootassummain}, the fact that $r(0) \meg r(u')$, the sign of $\nu$ and $r$, and finally we chose $\ee$ to be sufficiently small.

We note that by the boundedness of $\psi$ the above estimate finishes the proof in a bounded region with respect to $r$. Hence we focus on proving the desired estimates for $r$ such that $r \meg r_0 \gg M$  (choosing $r_0$ such that $D ( r_0 ) \meg 2/3$ is enough).

Now we can just note that we have also the following ($v^{*}$ depends on $r_0$, see the figure \eqref{auarea} for a pictorial representation of the region of integration with respect to $v$):
$$ \partial_v (r\psi ) = r \cdot \partial_v \psi + \partial_v r \cdot \psi \Rightarrow $$ $$ \Rightarrow | r(u' , v'' ) \cdot \psi (u' , v'' )| \mik | r(u' , v^{*} ) \cdot \psi (u' , v^{*} )| + \int_{v^{*}}^{v''} |r \cdot \partial_v \psi + \partial_v r \cdot \psi | dv \Rightarrow $$ $$ \Rightarrow | r(u' , v'' ) \cdot \psi (u' , v'' )| \mik | r(u' , v^{*} ) \cdot \psi (u' , v^{*} )| + \dfrac{1}{2\bar{\aaa}} \int_{v^{*}}^{v''} B_1 B \sqrt{E_0} \ee \dfrac{1}{r^{1-\bar{\aaa}}} \dfrac{D}{D} dv \mik $$ $$ \mik B \sqrt{E_0} \ee \dfrac{1}{r^2 (u' , v^{*} )} + \dfrac{3}{4\bar{\aaa}} B_1 B \sqrt{E_0} \ee r^{\bar{\aaa}} (u' , v'' ) \mik \dfrac{1}{\bar{\aaa}}B_1 B \sqrt{E_0} \ee r^{\bar{\aaa}} (u' , v'' )  , $$
where we integrated as shown in the picture below by noticing that in this region $v \approx r$ (we introduced $D dv$ to change coordinates and revert to integration in $r$ and we also used that $\dfrac{1}{D(r_0 )} \mik 3/2$) and we used again assumption \eqref{maxas4}.

Finally we turn to the term $r^2 \partial_v \psi$ and by our two last estimates we have as desired:
$$ |r (u' , v'' ) \cdot \partial_v \psi (u' , v'' ) | \mik |\partial_v r (u' , v'' ) | |\psi (u' , v'' )| + \dfrac{1}{2} B_1 B \sqrt{E_0} \ee \dfrac{1}{r^{1-\bar{\aaa}} (u' , v'' )} \mik $$ $$ \mik |\psi (u' , v'' )| + \dfrac{1}{2} B_1 B \sqrt{E_0} \ee \dfrac{1}{r^{1-\bar{\aaa}} (u' , v'' )} \mik \dfrac{3}{2\bar{\aaa}} B_1 B \sqrt{E_0} \ee \dfrac{1}{r^{1-\bar{\aaa}} (u' , v'' )}  . $$

\begin{rem}
The threshold for $\ee$ that we defined above will be called $\ee_1$ in this subsection.
\end{rem}
\begin{figure}[H]\label{auarea}
\centering
\includegraphics[width=6cm]{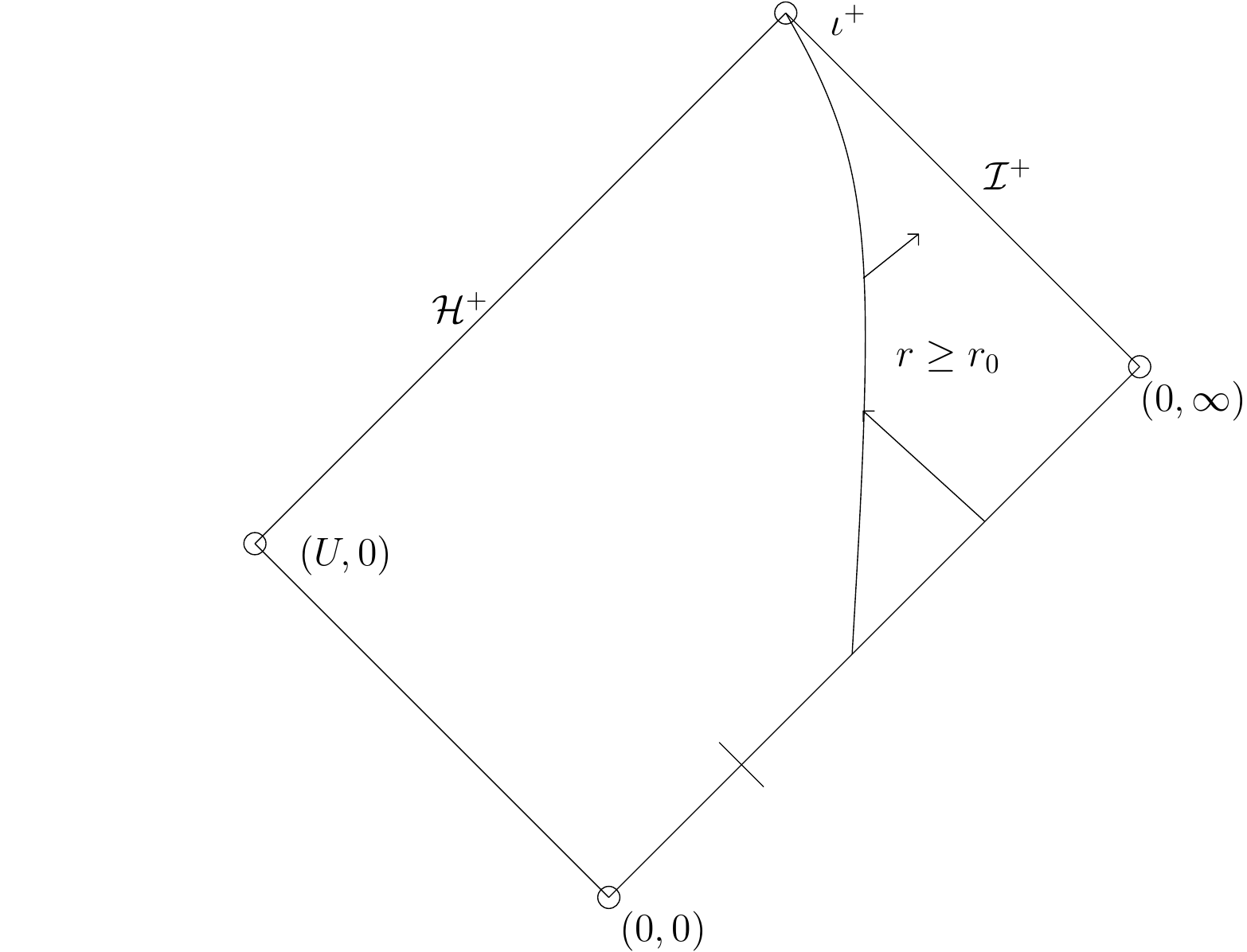}
\caption{The arrows show the domains of integration in $u$ and $v$ that were used in the previous computations.}
\end{figure}

We now turn to the main task of closing the bootstrap for $\dfrac{\zeta}{\nu}$. We fix $u'$ so that $0 \mik u' \mik U$ and we always take $0 < v'' \mik v'$ (note that we don't impose any further restrictions as before). We use equation \eqref{zetan1} and we have:
\begin{equation}\label{est1zetan1}
\left| \frac{\zeta}{\nu} (u' , v'' ) \right| \mik B \sqrt{E_0}  \ee + \frac{B E_0 \ee}{50} + \frac{B E_0 \ee}{50} \int_0^{v''} \dfrac{\partial_{\bar{v}} \nu}{\nu (u' , v'' )}  d\bar{v} + 
\end{equation}
$$ + |I + II + III + IV + V + VI + VII | ,$$
by using estimate \eqref{denull} and the fact that both $\partial_v \nu$ and $\nu$ are negative.

Now we deal with the last seven terms of the above inequality, leaving the third one for the end. 

For the first four terms we use the boundedness of $A$ and $A'$, assumptions \eqref{denull}, \eqref{maxas1}, \eqref{maxas3}, \eqref{bootassummain}, the boundedness of $\left| \dfrac{1}{\nu} \right|$ and the fact that $\left| \dfrac{\nu (u' , 0 )}{\nu (u' , v'' )} \right| \mik 1$, to arrive at the following estimate:
$$ | I + II + III + IV | \mik  a_0 B^2 (B_1 + 1) E_0 \ee^2 + B_1 B^3 \left( a_0^2 + a_1 \right) E_0^{3/2} \ee^3 .$$
For the remaining three terms we have the following computations.

For the term $\frac{1}{2\nu(u,v)} \int_0^v \left(A'' (\psi) + 3 A(\psi ) A' (\psi ) + A^3 (\psi ) \right) r \partial_u \psi \partial_{\bar{v}} \psi \cdot \psi^2 d\bar{v}$ we work as follows:

$$ V = \dfrac{1}{2\nu(u',v'')} \int_0^{v''} \left(A'' (\psi) + 3 A(\psi ) A' (\psi ) + A^3 (\psi ) \right) r \partial_u \psi \partial_{\bar{v}} \psi \cdot \psi^2 d\bar{v} \Rightarrow $$ $$ |V| \mik \sup_{[0,v'' ]} \left( \left| \dfrac{\zeta}{\nu} \right| \cdot |\partial_v \psi | \right)  \int_0^{v''} |\left(A'' (\psi) + 3 A(\psi ) A' (\psi ) + A^3 (\psi ) \right)| \dfrac{B^2 E_0 \ee^2}{(1+v)^{6/5 - 2\aaa_0}} dv  $$ $$ \mik 8 (a_2 + 3a_0 a_1 + a_0^3 ) B_1^2 B^4 E_0^2 \ee^4 , $$ 
where we used assumption \eqref{bootassummain}, estimate \eqref{deenull}, assumption \eqref{maxas3} and the boundedness of $A$, $A'$ and $A''$.
 
For the term $\frac{1}{\nu(u,v)} \int_0^v \left( A' (\psi ) +\frac{A^2 (\psi)}{2} \right) \partial_u r \partial_{\bar{v}} \psi \cdot \psi^2 d\bar{v}$ we work similarly and we have:
 
$$ VI = -  \dfrac{1}{\nu(u',v'')} \int_0^{v''} \left( A' (\psi ) +\dfrac{A^2 (\psi)}{2} \right) \partial_u r \partial_{\bar{v}} \psi \cdot \psi^2 d\bar{v} \Rightarrow $$ $$ |VI| \mik  ( a_1 + a^2 /2 ) B_1 B \sqrt{E_0} \ee  \int_0^{v''} \dfrac{B^2 E_0 \ee^2}{(1+v)^{6/5 - 2\aaa_0}} dv \mik $$ $$ \mik ( a_1 + a^2 /2 ) B_1 B^2  \ee^3 , $$
where we used the monotonicity property of $\nu$, assumption \eqref{bootassummain}, estimate \eqref{deenull}, assumption \eqref{maxas3} and the boundedness of $A$ and $A'$.

Finally we deal with the term $\frac{1}{2\nu(u,v)} \int_0^v \partial_{u\bar{v}} r A(\psi ) \psi^2 d\bar{v}$.

$$ VII = -\dfrac{1}{2\nu(u',v'')} \int_0^{v''} \partial_{u\bar{v}} r A(\psi ) \psi^2 d\bar{v} \Rightarrow $$ $$ VII = -\dfrac{1}{2\nu(u',v'')} \int_0^{v''} \dfrac{2\nu}{r^2} \left( M - \dfrac{M^2}{r} \right) A(\psi ) \psi^2 d\bar{v} \Rightarrow $$  $$ |VII| \mik M a_0 \int_0^{v''} \dfrac{B^2 E_0 \ee^2}{(1+v)^{6/5 - 2\aaa_0}} dv \mik $$ $$ \mik  2M a_0 B^2 E_0 \ee^2 , $$
where we used equation \eqref{ruv1} for $\partial_{uv} r$, estimate \eqref{deenull}, assumption \eqref{maxas2} and the boundedness of $A$.

\begin{figure}[H]
\centering
\includegraphics[width=6cm]{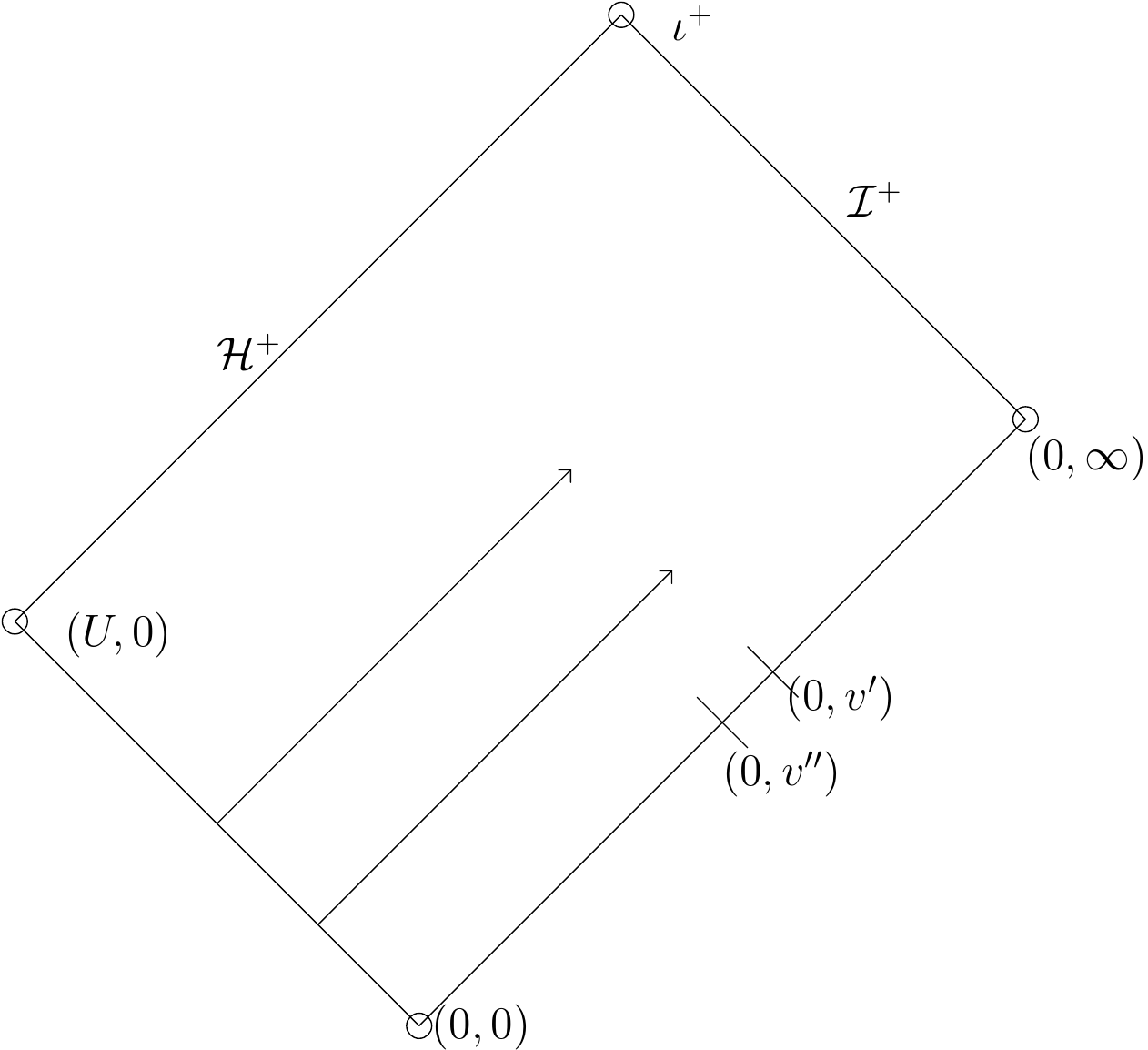}
\caption{The arrows show the domains of integration (in $v$) used for the boundedness of $\dfrac{\zeta}{\nu}$, where $u'$ is always taken to be $0 \mik u' \mik U$. The constant $v''$ is taken to be $0 < v'' \mik v'$. Note that since $v' \mik V$ and $V$ can be infinite, the $L^2$ integrability of $\psi$ in the bootstrap region (where Theorem \ref{dwee} applies) is crucial for our computations.}
\end{figure}

Gathering all the estimates together and going back to \eqref{est1zetan1} we have:
\begin{equation}\label{est2zetan1}
 \left| \frac{\zeta}{\nu} (u' , v'' ) \right| \mik B \sqrt{E_0} \ee + \frac{B \sqrt{E_0} \ee}{50} +  
 \end{equation}
 $$ + \frac{B \sqrt{E_0} \ee}{50}   \int_0^v \dfrac{\partial_{\bar{v}}\nu(u' , \bar{v} )}{\nu (u' , v'' )} d\bar{v} + 2 \left(  a_0 B^2 (B_1 + 1) +   a_0 B^2 \right) E_0 \ee^2 + $$ $$ +  4 \left( ( B_1 B^3 a_0^2 + a_1 ) + ( a_1 + a^2_2 /2 ) B_1 B^2  \right) E_0^{3/2} \ee^3 + $$ $$ + 8(a_2 + 3 a_0 a_1 + a_0^3 ) B_1^2 B^4 E_0^2 \ee^4 . $$
And since it holds that:
$$  \int_0^v \dfrac{\partial_{\bar{v}}\nu(u' , \bar{v} )}{\nu (u' , v'' )} d\bar{v} = \left( 1 - \dfrac{\nu (u' , 0 )}{\nu (u' , v'' )} \right)  \mik 1 , $$
as $\nu$ is negative and non-increasing everywhere in $\mathcal{D}$, we can restate \eqref{est2zetan1} as:
$$ \left| \frac{\zeta}{\nu} (u' , v'' ) \right| \mik B \sqrt{E_0} \ee + \frac{2 B \sqrt{E_0} \ee}{50}  + $$ $$  + 2 \left(  a_0 B^2 (B_1 + 1) +   a_0 B^2 \right) E_0 \ee^2 + $$ $$ +  4 \left( ( B_1 B^3 a_0^2 + a_1 ) + ( a_1 + a^2_2 /2 ) B_1 B^2  \right) E_0^{3/2} \ee^3 + $$ $$ + 8(a_2 + 3 a_0 a_1 + a_0^3 ) B_1^2 B^4 E_0^2 \ee^4 . $$
From this we can conclude that
$$ \left| \frac{\zeta}{\nu} (u' , v'' ) \right| \mik \frac{3}{2} B_1 B  \sqrt{E_0} \ee \mbox{ for every $(u' , v'' ) \in \mathcal{D}$} ,$$
as required, by choosing $\ee$ to be small enough (and smaller than $\ee_1$ -- we call $\ee_r$ this threshold).

\end{proof}

\begin{rem}\label{rdecpsidv}
For the estimates $ | r \psi | \lesssim_{\aaa} \sqrt{E_0} \ee r^{\bar{\aaa}} $ and $ | r^2 \partial_v \psi | \lesssim_{\aaa} \sqrt{E_0} \ee r^{\bar{\aaa}} $ we choose $\bar{\aaa}$ to be:
\begin{equation}\label{baracond}
\bar{\aaa} = \dfrac{\aaa}{4} ,
\end{equation}
for $\aaa$ the fixed constant (that was chosen so that $\psi$ is $L^2$ integrable with respect to $v$) from the assumptions of Theorem \ref{dr} (actually any $\bar{\aaa}$ satisfying $\bar{\aaa} < \dfrac{\aaa}{2}$ will do).
\end{rem}

\subsection{A Priori Pointwise Boundedness for $T \psi$}
Using the results of the previous subsection, we will now prove that $T\psi$ is bounded as well in all of the domain of outer communications. 

We have the following equation for the vector fields of the different coordinate systems:
\begin{equation}\label{tpartialv}
 \partial_v = T + D Y .
\end{equation} 
We can now prove the following:
\begin{thm}[\textbf{Boundedness For $T\psi$}]\label{dvaway}
There exists some $\ee_r := \ee_r (\aaa ) > 0$ such that if $0 < \ee < \ee_r$ and the bootstrap assumptions \eqref{A1}, \eqref{A2}, \eqref{A3} hold true for $\psi$ a spherically symmetric solution of \eqref{nwef}, for some constant $C_{\aaa} := C(\aaa )$, for any $\tau_0$, $\tau_1$, $\tau_2$ with $\tau_1 < \tau_2$ and for any given $\aaa > 0$, then we have that:
\begin{equation}\label{dvaway1}
\| T \psi (\tau ) \|_{L^{\infty} (\Sigma_{\tau} )} \lesssim \sqrt{E_0} \ee ,
\end{equation}
for all $\tau$.
\end{thm}
\begin{proof}
By Theorem \ref{dr} we have that $Y\psi$ is bounded by $\lesssim \sqrt{E_0} \ee$. In the proof of the same Theorem we also showed that we have as well:
$$ \| \partial_v \psi \|_{ L^{\infty} (\Sigma_{\tau} )} \lesssim \sqrt{E_0} \ee \mbox{ for any $\tau$}.$$
The result follows now from equation \eqref{tpartialv}. 
\end{proof}

\subsection{Improved A Priori Pointwise Decay for $\psi$}\label{improvedpsi}

The decay estimate \eqref{dw} that we obtained for $\psi$ can be improved by using the boundedness results for $Y \psi$ of the previous subsections. This was actually used already in the proof of the boundedness of $Y \psi$ in the form of estimate \eqref{deenull}. We prove it below.
\begin{thm}[\textbf{Decay In $L^{\infty}$ For $\psi$}]\label{dwee}
There exists some $\ee_r := \ee_r (\aaa ) > 0$ such that if $0 < \ee < \ee_r$ and the bootstrap assumptions \eqref{A1}, \eqref{A2}, \eqref{A3} hold true for $\psi$ a spherically symmetric solution of \eqref{nwef}, for some constant $C_{\aaa} := C(\aaa )$, for any $\tau_0$, $\tau_1$, $\tau_2$ with $\tau_1 < \tau_2$ and for any given $\aaa > 0$, then we have that:
\begin{equation}\label{dwe} 
\|\psi \|_{L^{\infty} (\Sigma_{\tau})} \lesssim_{\aaa} \dfrac{ \sqrt{E_0} \ee}{(1+ \tau )^{3/5 - 3\aaa/10}} ,
\end{equation}
for all $\tau$.
\end{thm}
\begin{rem}
One of the important consequences of Theorem \ref{dwee} is that we can conclude that $\psi$ is $L^2$ integrable with respect to $\tau$. This was used in the proof of Theorem \ref{dr} in a bootstrap setting and will be used later as well when we will establish the existence of an almost conservation law on the horizon and the blow up of derivatives of order 2 and higher again on the horizon.
\end{rem}
\begin{proof}
We will only consider the case that $r'$ is close to the horizon since for $r'$ far away from the horizon we have the estimate \eqref{pdec1} which is even better than the desired one here.

As we did in the proof of Theorem \ref{dpsi} we have the following inequality for some $\alpha_0 > 0$ that will be chosen later by applying Stokes' Theorem in the region $\rrr(0,\tau) \cap r' \mik r \mik r' + \tau^{-\aaa_0}$ for $r' \meg M$:
$$ \int_{\mathbb{S}^2} \psi^2 (r' , \omega) d\omega \lesssim \int_{\mathbb{S}^2} \psi^2 (r' + \tau^{-\alpha_0} , \omega) d\omega + \int_{\Sigma_{\tau} \cap \{ r' \mik r \mik r' + \tau^{-\alpha_0} \}} \psi \partial_{\rho} \psi d\mi_{g_{\si}} \lesssim $$ $$ \lesssim \dfrac{ E_0\ee^2}{ (1+\tau )^{2-\aaa-2\alpha_0}} + \left( \int_{ \Sigma_{\tau} \cap \{ r' \mik r \mik r' + \tau^{-\alpha_0} \}} \psi^2 d\mi_{g_{\si}} \right)^{1/2} \times $$ $$ \times \left( \int_{ \Sigma_{\tau} \cap \{ r' \mik r \mik r' + \tau^{-\alpha_0} \}} (\partial_{\rho} \psi )^2 d\mi_{g_{\si}} \right)^{1/2} \lesssim $$ $$ \lesssim \dfrac{ E_0 \ee^2 }{ (1+\tau )^{2-\aaa-2\alpha_0}} + \dfrac{E_0\ee^2 }{ ( 1 + \tau )^{1-\aaa/2 + \alpha_0 /2} }  ,$$
where for the term $\int_{\mathbb{S}^2} \psi^2 (r' + \tau^{-\alpha_0}, \omega ) d\omega$ we used estimate \eqref{p1} and where in the last step we used Hardy's inequality \eqref{hardy} (for the first term in parenthesis), the decay of the $T$-Flux \eqref{end} (for the first term in parenthesis again after using Hardy) and the boundedness of $Y \psi$ (for the second term in parenthesis).

We can easily compute the optimal $\aaa_0$:
$$ 2-\aaa-2\alpha_0 = 1 - \frac{\aaa}{2} + \frac{\alpha_0}{2} \Rightarrow \alpha_0 = \frac{2}{5} - \frac{\aaa}{5}. $$
This now gives us the following decay:
\begin{equation}\label{impdec}
\int_{\mathbb{S}^2} \psi^2 (r' , \omega) d\omega \lesssim E_0 \ee^2 (1+\tau )^{6/5-3\aaa/5} .
\end{equation}
This is the desired decay because of the assumption of spherical symmetry.
\end{proof}

\section{Improving Bootstrap Assumptions And Closing Estimates}\label{bre}
In this section we will try to prove the necessary decay results for the nonlinear term using a bootstrap method, i.e. we will show that the bootstrap assumptions \eqref{A1}, \eqref{A2}, \eqref{A3} hold always true. 

\begin{prop}[\textbf{Bootstrap Results}]\label{boot}
Assume that the following bounds hold for some constants $C_{\aaa} = C(\aaa) > 0$ and $\ee > 0$ (satisfying $\ee < \ee_r$ for $\ee_r := \ee_r (\aaa)$ given in Section \ref{drbound}), for any $\tau_0$, $\tau_1$, $\tau_2$ with $\tau_1 < \tau_2$, and for any given $\aaa > 0$ such that $\dfrac{3}{5} - \dfrac{3\aaa}{10} > \dfrac{1}{2}$:
\begin{equation}\label{boot1a}
\int_{S_{\tau_0 }} |F|^2 d\mi_{g_S} \mik C_{\aaa} E_0 \ee^2  (1+\tau_0 )^{-2+\aaa} ,
\end{equation}
\begin{equation}\label{boot2a}
\int_{\tau_1}^{\tau_2} \int_{S_{\tau'}} |F|^2 d\mi_{g_{\so}} \mik C_{\aaa} E_0 \ee^2  (1+\tau_1 )^{-2+\aaa} ,
\end{equation}
\begin{equation}\label{boot3a}
\int_{\tau_1}^{\tau_2} \int_{N_{\tau'}} r^{3-\aaa} |F|^2 d\mi_{g_{\nnn}} \mik C_{\aaa} E_0  \ee^2 (1+\tau_1 )^{-2+\aaa} .
\end{equation}
Then we have that:
\begin{equation}\label{boot1}
\int_{S_{\tau_0 }} |F|^2 d\mi_{g_S} \lesssim E_0^2 \ee^4  (1+\tau_0 )^{-2+\aaa} ,
\end{equation}
\begin{equation}\label{boot2}
 \int_{\tau_1}^{\tau_2} \int_{S_{\tau'}} |F|^2 d\mi_{g_{\so}} \lesssim E_0^2 \ee^4  (1+\tau_1 )^{-2+\aaa} ,
 \end{equation}
\begin{equation}\label{boot3} 
\int_{\tau_1}^{\tau_2} \int_{N_{\tau'}} r^{3-\aaa} |F|^2 d\mi_{g_{\nnn}} \lesssim  E_0^2  \ee^4 (1+\tau_1 )^{-2+\aaa} .
 \end{equation}
\end{prop}
\begin{proof}
(a) First we will prove \eqref{boot1}.

Recall that the nonlinearity has the following form:
$$ F = A(\psi ) \left( D (Y\psi )^2 + 2T\psi \cdot Y\psi \right) ,$$
so we can compute:
\begin{equation}\label{pnonlin}
|F|^2 = A^2 (\psi ) \left( D (Y\psi )^2 + 2T\psi Y\psi \right)^2 \mik a_0 \left( 2 D^4 (Y\psi )^2 + 8 (T\psi )^2 (Y\psi )^2 \right) =
\end{equation}
 $$ = a_0 \left[ (Y\psi )^2 \left( 2 D^2 (Y\psi )^2 + 8 (T\psi )^2 \right) \right] \mik a_0 \left[ 8 (Y\psi )^2 \left( D^2 (Y\psi )^2 + (T\psi )^2 \right) \right] .$$
We then have for any $\tau$: 
\begin{equation}\label{boot11}
\int_{S_{\tau}} |F|^2 d\mi_{g_S} \mik 8 a_0 \| Y \psi \|_{L^{\infty} ( S_{\tau} )}^2  \int_{S_{\tau}} \left( D^2 (Y \psi )^2 + ( T \psi )^2 \right) d\mi_{g_S} . 
\end{equation}
Notice that since $D^2 \mik D$ and the integral is taken over the spacelike surface $S_{\tau}$, then it is implied that the following is true:
$$ \int_{S_{\tau}} D^2 (Y \psi )^2 + ( T \psi )^2 d\mi_{g_S} \mik \int_{S_{\tau}} J^T_{\mu} [\psi ] n^{\mu}  d\mi_{g_S} .$$

Because of our assumptions \eqref{boot1a}, \eqref{boot2a}, \eqref{boot3a} we can apply Theorem \ref{dr} which gives us that:
$$ \| Y \psi \|_{L^{\infty} ( S_{\tau} )}^2 \lesssim E_0 \ee^2 .$$

Moreover since we have that:
$$ \int_{S_{\tau}} J^T_{\mu} [\psi ] n^{\mu} d\mi_{g_S} \mik \int_{\si_{\tau}} J^T_{\mu} [\psi ] n^{\mu} d\mi_{g_{\si}} , $$
we can apply Theorem \ref{endec} which gives us that
$$  \int_{S_{\tau}} J^T_{\mu} [\psi ] n^{\mu} d\mi_{g_S} \lesssim E_0 \ee^2 (1+ \tau )^{-2+\aaa} ,$$ 
and, finally, we can conclude by the previous inequality \eqref{boot11} that we have in the end the desired estimate:
$$ \int_{S_{\tau}} |F|^2 d\mi_{g_{\si}} \lesssim E_0 \ee^2 \cdot E_0 \ee^2 (1+ \tau )^{-2+\aaa} . $$ 

(b) Now we will prove estimate \eqref{boot2}.

The proof is similar to the proof of estimate \eqref{boot1}. We use again the pointwise bound \eqref{pnonlin} and we get:
\begin{equation}\label{boot22}
\int_{\tau_1}^{\tau_2} \int_{S_{\tau'}} |F|^2 d\mi_{g_{\so}} \mik 8 a_0 \| Y \psi \|_{L^{\infty} ( \so (\tau_1 , \tau_2) )}^2  \int_{\tau_1}^{\tau_2} \int_{S_{\tau'}} \left( D^2 (Y \psi )^2 + ( T \psi )^2 \right) d\mi_{g_{\so}} .
\end{equation}
Again because of all three of our assumptions \eqref{boot1a}, \eqref{boot2a}, \eqref{boot3a} we can bound $\| Y \psi \|_{L^{\infty} ( \so (\tau_1 , \tau_2) )}^2$ by $E_0 \ee^2$.

For the term $\int_{\tau_1}^{\tau_2} \int_{S_{\tau'}} \left( D^2 (Y \psi )^2 + ( T \psi )^2 \right) d\mi_{g_{\so}}$ we use the Morawetz estimate \eqref{degix1} for some $0 < \eta < 1$ and we have that:
$$ \int_{\tau_1}^{\tau_2} \int_{S_{\tau'}} \left( D^2 (Y \psi )^2 + ( T \psi )^2 \right) d\mi_{g_{\so}} \lesssim \int_{\si_{\tau_1}} J^T_{\mu} [\psi ] n^{\mu} d\mi_{g_{\si}} + $$ $$ + \int_{\tau_1}^{\tau_2} \int_{S_{\tau'}} |F|^2 d\mi_{g_{\so}} + \int_{\tau_1}^{\tau_2} \int_{N_{\tau'}} r^{1+\eta} |F|^2 d\mi_{g_{\nnn}} \lesssim E_0 \ee^2 (1+ \tau_1 )^{-2+\aaa} ,$$
where in the last step we used our assumptions and the fact that because of them we can apply Theorem \ref{endec} which gives us that $\int_{\si_{\tau_1}} J^T_{\mu} [\psi ] n^{\mu} d\mi_{g_{\si}}$ decays with rate $-2+\aaa$.

The proof finishes by using this last bound and the bound for $Y\psi$ (mentioned earlier) in \eqref{boot22}.

In the proof given above for estimate \eqref{boot2} the importance of the factor $D^2$ in front of $(Y\psi )^2$ should be highlighted, if we had just $D$ we wouldn't be able to apply the Morawetz estimate and the bootstrap wouldn't close.

(c) Finally we deal with the nonlinear term close to future null infinity, i.e. with the proof of \eqref{boot3}.

Note that here we will use the "standard" null coordinates $(u , v ) = \left( t-r^{*} , t+r^{*} \right)$ under which the nonlinearity takes the following form:
$$ \Box_g \psi = \frac{4}{D} A(\psi ) \partial_u \psi \partial_v \psi .$$

We will express the nonlinearity in terms of $\ph = r\psi$. We have:
$$ \partial_u \ph = r \partial_u \psi - D \psi, \quad \partial_v \ph = r \partial_v \psi + D \psi \Rightarrow $$ $$ \Rightarrow (\partial_u \psi )^2 \lesssim \dfrac{1}{r^2} \left( (\partial_u \ph )^2 + D^2 \psi^2 \right) , \quad (\partial_v \psi )^2 \lesssim \dfrac{1}{r^2} \left( (\partial_v \ph )^2 + D^2 \psi^2 \right) \Rightarrow $$  
\begin{equation}\label{rnonlin}
 r^{3-\aaa} (\partial_u \psi )^2 (\partial_v \psi )^2 \lesssim \dfrac{1}{r^{1+\aaa}} \left( (\partial_u \ph )^2 (\partial_v \ph )^2 + D^2 \psi^2 (\partial_u \ph )^2 + D^2 \psi^2 (\partial_v \ph )^2 + D^4 \psi^4 \right) . 
 \end{equation}

We will treat separately the four different terms in \eqref{rnonlin} and we will ignore $4A$ for now since it is bounded.

We start with the last term that involves only $\psi$. We have:
\begin{equation}\label{psinonlin}
 \int_{\tau_1}^{\tau_2} \int_{N_{\tau'}} \dfrac{D^4}{D^2 r^{1+\aaa}}  \psi^4 d\mi_{g_{\nnn}} \mik \| r^{1-\aaa} \psi^2 \|_{L^{\infty} (\nnn (\tau_1 , \tau_2 ))} \int_{\tau_1}^{\tau_2} \int_{N_{\tau'}} \dfrac{\psi^2}{r^2} d\mi_{g_{\nnn}} .
 \end{equation}

Because of our assumptions we can apply Theorem \ref{pdec} and get the estimate:
\begin{equation}\label{psinonlin1}
 \| r^{1-\aaa} \psi^2 \|_{L^{\infty} (\nnn (\tau_1 , \tau_2 ))} \lesssim E_0 \ee^2 (1+\tau_1 )^{-2+\aaa} .
\end{equation}

On the other hand again because of our assumptions we can apply Theorem \ref{endec} and Hardy's inequality \eqref{hardy} and get the following estimate as well:
\begin{equation}\label{psinonlin2}
 \int_{\tau_1}^{\tau_2} \int_{N_{\tau'}} \dfrac{\psi^2}{r^2} d\mi_{g_{\nnn}} \lesssim \int_{\tau_1}^{\tau_2} \int_{\si_{\tau'}} J^T_{\mu} [\psi] n^{\mu} d\mi_{g_{\rrr}} \lesssim 
\end{equation} 
  $$ \lesssim \int_{\tau_1}^{\tau_2} \dfrac{E_0 \ee^2}{(1+\tau' )^{2-\aaa}} d\tau' \lesssim E_0 \ee^2 (1+\tau_1 )^{-1+\aaa} .$$ 

So in the end for the term $ \dfrac{D^4}{D^2 r^{1+\aaa}} \psi^4$ we get better decay than required by inserting estimates \eqref{psinonlin1}, \eqref{psinonlin2} in \eqref{psinonlin}:
\begin{equation}\label{psinonlinf}
\int_{\tau_1}^{\tau_2} \int_{N_{\tau'}} \dfrac{D^4}{D^2 r^{1+\aaa}}  \psi^4 d\mi_{g_{\nnn}} \lesssim E_0^2 \ee^4 (1+ \tau_1 )^{-3+2\aaa} .
\end{equation} 

We now look at the term $\dfrac{D^2}{D^2 r^{1+\aaa}} \psi^2 (\partial_v \ph )^2$:
\begin{equation}\label{psidvnonlin}
\int_{\tau_1}^{\tau_2} \int_{N_{\tau'}} \dfrac{D^2}{D^2 r^{1+\aaa}} \psi^2 (\partial_v \ph )^2 d\mi_{g_{\nnn}} \mik 
\end{equation}
$$ \mik \| r^{1-\aaa} \psi^2 \|_{L^{\infty} (\nnn (\tau_1 , \tau_2 ))} \int_{\tau_1}^{\tau_2} \int_{N_{\tau'}} \dfrac{(\partial_v \ph )^2}{r^2}  d\mi_{g_{\nnn}} \lesssim $$ $$ \lesssim E_0 \ee^2 (1+\tau_1 )^{-2+\aaa} \int_{N_{\tau'}} \dfrac{(\partial_v \ph )^2}{r^2}  d\mi_{g_{\nnn}} \lesssim E_0^2 \ee^4 (1+\tau_1 )^{-2+\aaa} , $$
where we used estimate \eqref{psinonlin1} and estimate \eqref{pintegint2}.

We now look at the term $\dfrac{D^2}{D^2 r^{1+\aaa}} \psi^2 (\partial_u \ph )^2$. We write the volume form in detail and we use the spherical symmetry assumption.

\begin{equation}\label{psidunonlin}
\int_{\tau_1}^{\tau_2} \int_{N_{\tau'}} \dfrac{D^2}{D^2 r^{1+\aaa}} \psi^2 (\partial_u \ph )^2 d\mi_{g_{\nnn}} =  
\end{equation}
$$ = \int_{v_{R_0}}^{\infty}  \int_{u_{\tau_1} (v) }^{u_{\tau_2} (v) } \int_{\mathbb{S}^2}  \dfrac{D^2}{D^2 r^{1+\aaa}} \psi^2 (\partial_u \ph )^2 D r^2 d\gamma_{\mathbb{S}^2} dudv \lesssim $$ $$ \lesssim \int_{v_{R_0}}^{\infty}  \int_{u_{\tau_1} (v) }^{u_{\tau_2} (v) }  r^{1-\aaa} \psi^2 (\partial_u \ph )^2  dudv \lesssim  \int_{v_{R_0}}^{\infty}  \int_{u_{\tau_1} (v) }^{u_{\tau_2} (v) }  (\partial_u \ph )^2 du \sup_u \left(  r^{1-\aaa} \psi^2 \right)  dv \lesssim $$ $$ \lesssim \sup_v \left( \int_{u_{\tau_1} (v) }^{u_{\tau_2} (v) }  (\partial_u \ph )^2 du \right) \cdot \int_{v_{R_0}}^{\infty} \sup_u \left(  r^{1-\aaa} \psi^2 \right)  dv . $$
Now we examine separately the terms $\int_{v_{R_0}}^{\infty} \sup_u \left(  r^{1-\aaa} \psi^2 \right)  dv$ and $\int_{u_{\tau_1} (v) }^{u_{\tau_2} (v) }  (\partial_u \ph )^2 du$. 

Because of our assumptions we can apply estimate \eqref{rdecpsi1} and have the following estimate:
$$ \left( r (u,v) \right)^{1-\aaa} \psi^2 (u,v) \lesssim E_0 \ee^2 (1+v)^{-1 - \aaa + 2\bar{\aaa}} \mbox{ for all $u,v$.} $$
Under condition \eqref{baracond} ($\bar{\aaa} = \aaa/4$) we get then that $r^{1-\aaa} \psi^2$ is integrable with respect to $r$ (and here more specifically with respect to $v$ since in this region $v \approx r$).

From this we get that:
\begin{equation}\label{psidunonlin1}
\int_{v_{R_0}}^{\infty} \sup_u \left(  r^{1-\aaa} \psi^2 \right)  dv \lesssim E_0 \ee^2 .
\end{equation}

For the term $\int_{u_{\tau_1} (v) }^{u_{\tau_2} (v) }  (\partial_u \ph )^2 du$ we work as follows.

\begin{equation}\label{intuph}
 \int_{u_{\tau_1}}^{u_{\tau_2}} (\partial_u \ph )^2 du = \int_{u_{\tau_1}}^{u_{\tau_2}} r^2 (\partial_u \psi )^2 du - \int_{u_{\tau_1}}^{u_{\tau_2}} D \partial_u (r \psi^2 ) du =
 \end{equation}
$$ = \int_{u_{\tau_1}}^{u_{\tau_2}} r^2 (\partial_u \psi )^2 du - D r\psi^2 |_{u_{\tau_1}}^{u_{\tau_2}} + \int_{u_{\tau_1}}^{u_{\tau_2}} \partial_u D \cdot r\psi^2 du . $$
For the last term of \eqref{intuph} we can notice that:
$$ \int_{u_{\tau_1}}^{u_{\tau_2}} \partial_u D \cdot r\psi^2 du \mik 0 ,$$
since
$$ \partial_u D =  - D \cdot D' = -D \dfrac{2M}{r^2} \left( 1-\dfrac{M}{r} \right) \mik 0 \mbox{ for $r \meg M$}. $$
Hence going back to \eqref{intuph} we get that:
\begin{equation}\label{intuph1}
 \int_{u_{\tau_1}}^{u_{\tau_2}} (\partial_u \ph )^2 du \mik \int_{u_{\tau_1}}^{u_{\tau_2}} r^2 (\partial_u \psi )^2 du - D r\psi^2 |_{u_{\tau_1}}^{u_{\tau_2}} .
\end{equation}
For the first term of the right hand side of \eqref{intuph1} we have:
$$ \int_{u_{\tau_1}}^{u_{\tau_2}} r^2 (\partial_u \psi )^2 du \lesssim \int_{\mathcal{U}_{\tau_1}^{\tau_2}} J^T_{\mu} [\psi ] n^{\mu} d\mi_{g_{\mathcal{U}_{\tau_1}^{\tau_2}}} ,$$
where $\mathcal{U}_{\tau_1}^{\tau_2}$ is the null line between $u_{\tau_1}$ and $u_{\tau_2}$. 

We apply Stokes' theorem to the current $J^T$ in the following region defined for any given $\bar{v} \in [v_{R_0} , \infty )$: 
$$U_{\tau_1}^{\tau_2} (\bar{v} ) = [u_{\tau_1} (\bar{v} ) , u_{\tau_2} (\bar{v} ) ] \times [\bar{v} , \infty ) \times \mathbb{S}^2 \subset \nnn (\tau_1 , \tau_2 ) $$
and we have that:
\begin{equation}\label{stokesunull}
\int_{\mathcal{U}_{\tau_1}^{\tau_2}} J^T_{\mu} [\psi ] n^{\mu} d\mi_{g_{\mathcal{U}_{\tau_1}^{\tau_2} (\bar{v} )}} =
\end{equation}
$$  = \int_{v \meg \bar{v} , u = u_{\tau_2}}  J^T_{\mu} [\psi ] n^{\mu} d\mi_g - \int_{v \meg \bar{v} , u = u_{\tau_1}}  J^T_{\mu} [\psi ] n^{\mu} d\mi_g + $$ $$ + \int_{\mathcal{I}^{+}}  J^T_{\mu} [\psi ] n^{\mu} d\mi_{g_{\mathcal{I}^{+}}} + \int_{U_{\tau_1}^{\tau_2}} F \cdot T\psi d\mi_{g_{U_{\tau_1}^{\tau_2}}} .$$

We note that all the terms in the right hand side of \eqref{stokesunull} can be bounded by the quantity
\begin{equation}\label{jtnonlin}
\int_{\si_{\tau_1}} J^T_{\mu} [\psi ] n^{\mu} d\mi_{g_{\si}} + \int_{\tau_1}^{\tau_2} \int_{S_{\tau'}} |F|^2 d\mi_{g_{\so}} +  \int_{\tau_1}^{\tau_2} \int_{N_{\tau'}} r^{1+\eta} |F|^2 d\mi_{g_{\nnn}} ,  
\end{equation}
for some $0 < \eta < 1$, since the terms $$\int_{v \meg v_{R_0} , u = u_{\tau_2}}  J^T_{\mu} [\psi ] n^{\mu} d\mi_g, \quad \int_{v \meg v_{R_0} , u = u_{\tau_2}}  J^T_{\mu} [\psi ] n^{\mu} d\mi_g$$ are of the form $\int_{\si_{\tau_j}} J^T_{\mu} [\psi ] n^{\mu} d\mi_{g_{\si}}$ for $\tau_j$ being either $\tau_1$ or $\tau_2$, the term at null infinity is obviously bounded by \eqref{jtnonlin} (see the proof of Proposition \ref{deg2}) and for the nonlinear term $\int_{U_{\tau_1}^{\tau_2}} F \cdot T\psi d\mi_{g_{U_{\tau_1}^{\tau_2}}}$ we use Cauchy-Schwarz with weight $r^{1+\eta}$ for some $0 < \eta < 1$ and then the Morawetz estimate of Proposition \ref{degixsn} (with $\eta$ as before).

On the other hand we also have:
$$ Dr\psi^2 |_{u_{\tau_1}}^{u_{\tau_2}} \lesssim E_0 \ee^2 (1+\tau_1 )^{-2 + \aaa} , $$
by Theorem \ref{pdec} as we noticed before.

So in the end we have by our assumptions that:
\begin{equation}\label{dunonlin}
\int_{u_{\tau_1}}^{u_{\tau_2}} (\partial_u \ph )^2 du \lesssim E_0 \ee^2 (1+\tau_1 )^{-2 + \aaa} .
\end{equation}
Inserting estimates \eqref{psidunonlin1} and \eqref{dunonlin} into \eqref{psidunonlin} we get in the end that:
\begin{equation}\label{psidunonlinf}
\int_{\tau_1}^{\tau_2} \int_{N_{\tau'}} \dfrac{D^2}{D^2 r^{1+\aaa}} \psi^2 (\partial_u \ph )^2 d\mi_{g_{\nnn}} \lesssim E_0^2 \ee^4 (1+ \tau_1 )^{-2+\aaa} . 
\end{equation}

Finally we turn to the term $\dfrac{1}{D^2 r^{1+\aaa}} (\partial_u \ph )^2 (\partial_v \ph )^2$. Working as before we have:
$$  \int_{\tau_1}^{\tau_2} \int_{N_{\tau'}} \dfrac{1}{D^2 r^{1+\aaa}} (\partial_u \ph )^2 (\partial_v \ph )^2 d\mi_{g_{\nnn}}  = $$ $$ = \int_{v_{R_0}}^{\infty}  \int_{u_{\tau_1} (v) }^{u_{\tau_2} (v) } \int_{\mathbb{S}^2} \dfrac{1}{D^2 r^{1+\aaa}} (\partial_u \ph )^2 (\partial_v \ph )^2 D r^2 d\gamma_{\mathbb{S}^2} du dv \lesssim $$ $$ \lesssim \int_{v_{R_0}}^{\infty}  \int_{u_{\tau_1} (v) }^{u_{\tau_2} (v) } (\partial_u \ph )^2 du \sup_u \left( r^{1-\aaa} (\partial_v \ph )^2 \right) dv \lesssim $$ $$ \lesssim \sup_v \left( \int_{u_{\tau_1} (v) }^{u_{\tau_2} (v) } (\partial_u \ph )^2 du \right) \cdot \int_{v_{R_0}}^{\infty} \sup_u \left( r^{1-\aaa} (\partial_v \ph )^2 \right) dv . $$

We examine separately the terms $\int_{u_{\tau_1}}^{u_{\tau_2}} (\partial_u \ph )^2 du$ and $\int_{v_{R_0}}^{\infty} \sup_u \left(  r^{1-\aaa} (\partial_v \ph )^2 \right)  dv$. From \eqref{dunonlin} we have that the first one is bounded by
$$ E_0 \ee^2 (1+\tau_1 )^{-2 + \aaa} .$$

For the second term we have the pointwise bounds:
$$ r^{1-\aaa} (\partial_v \ph )^2 \lesssim r^{1-\aaa} \psi^2 + r^{3-\aaa} (\partial_v \psi )^2 , $$
and since 
$$ | r\psi | \lesssim \sqrt{E_0} \ee r^{\bar{\aaa}} \Rightarrow r^{1-\aaa} \psi^2 \lesssim \sqrt{E_0} \ee \dfrac{r^{2\bar{\aaa} - \aaa}}{r} , $$
and 
$$ |r^2 \partial_v \psi | \lesssim \sqrt{E_0} \ee r^{\bar{\aaa}} \Rightarrow r^{3-\aaa} (\partial_v \psi )^2 \lesssim \sqrt{E_0} \ee \dfrac{r^{2\bar{\aaa} - \aaa}}{r} , $$
we can conclude that
$$ \sup_u  r^{1-\aaa} (\partial_v \ph )^2 \lesssim E_0 \ee^2 \dfrac{r^{2\bar{\aaa} - \aaa}}{r} , $$
which is integrable in $r$ as 
$$ \bar{\aaa} = \dfrac{\aaa}{4} \Rightarrow 2\bar{\aaa} - \aaa = - \dfrac{\aaa}{2} .$$
This shows that
$$ \int_{v_{R_0}}^{\infty} \sup_u \left(  r^{1-\aaa} (\partial_v \ph )^2 \right)  dv \lesssim E_0 \ee^2 . $$

Hence we conclude that:
\begin{equation}\label{psidudvnonlinf}
\int_{\tau_1}^{\tau_2} \int_{N_{\tau'}} \dfrac{1}{D^2 r^{1+\aaa}} (\partial_u \ph )^2 (\partial_v \ph )^2 d\mi_{g_{\nnn}} \lesssim E_0^2 \ee^4 (1+\tau_1 )^{-2 + \aaa} .
\end{equation}

So in the end because of the pointwise bound \eqref{rnonlin} and the boundedness of $A$ we have that:
$$ \int_{\tau_1}^{\tau_2} \int_{N_{\tau'}} r^{3-\aaa} |F|^2 d\mi_{g_{\nnn}} \lesssim \int_{\tau_1}^{\tau_2} \int_{N_{\tau'}} \dfrac{D^4}{D^2 r^{1+\aaa}}  \psi^4 d\mi_{g_{\nnn}} + $$ $$+\int_{\tau_1}^{\tau_2} \int_{N_{\tau'}} \dfrac{D^2}{D^2 r^{1+\aaa}} \psi^2 (\partial_v \ph )^2 d\mi_{g_{\nnn}} + \int_{\tau_1}^{\tau_2} \int_{N_{\tau'}} \dfrac{D^2}{D^2 r^{1+\aaa}} \psi^2 (\partial_u \ph )^2 d\mi_{g_{\nnn}} + $$ $$ +\int_{\tau_1}^{\tau_2} \int_{N_{\tau'}} \dfrac{1}{D^2 r^{1+\aaa}} (\partial_u \ph )^2 (\partial_v \ph )^2 d\mi_{g_{\nnn}}. $$

Gathering together estimates \eqref{psinonlinf}, \eqref{psidvnonlin}, \eqref{psidunonlinf} and \eqref{psidudvnonlinf} we arrive at the desired estimate \eqref{boot3}.

\end{proof}

\section{Global Well-Posedness}\label{tgwpt}
Now we combine the local theory with the boundedness estimates derived in the previous sections.

First we close our bootstrap assumptions by choosing $\ee > 0$ such that it is $\ee < \ee_r$ (where $\ee_r$ was the threshold for $\ee$ chosen in Theorem \ref{dr}) and such that
$$ C' E_0^2 \ee^4 \mik C E_0 \ee^2 ,$$
for $C'$ being the constant that shows up in the conclusion of Proposition \ref{boot}.

For spherically symmetric initial data $ (\ee \psi_0 , \ee \psi_1 )$ that are compactly supported and of size $\mik \ee$ in the $H^s (\Sigma_0 )$ norm for all $s>5/2$ we consider the initial value problem for equation \eqref{nwef}. We note that we choose our spacelike null foliation by setting $R_0$ to be such that $supp (\psi_0 ) \subset \{ r \mik R_0 \}$ for an $R_0$ that is big enough for Proposition \ref{rw} to hold.

We set up an iteration scheme as follows: let $\Psi_0 = 0$ and solve the set of linear wave equations given below recursively:
$$ \Box_g \Psi_n = F (\Psi_{n-1}), \quad \Psi_n |_{\Sigma_0} = (\ee \psi_0 , \ee \psi_1 ) \mbox{ for $n \meg 1$, $n \in \mathbb{N}$} .$$

We can't apply a local well-posedness theorem for this problem because of the null part of the foliation. But by the estimates that we derived in the previous sections, we get that $\Psi_n$, $T \Psi_n$, $Y \Psi_n$ are uniformly bounded for all $\tau$ and for all $n$. 

By the local well-posedness result of Theorem \ref{lwp} we know that there exists a unique solution $\psi$ to \eqref{nwef} up to a "time" $\mathcal{T}$ over the spacelike foliation $\widetilde{\Sigma}_{\tau}$. Since this solution is obtained by a Picard iteration scheme as well we actually have that:

$$ \Psi_n \rightarrow \psi \mbox{ in $C^{\infty} \left( \cup_{ \tau \in [0, \mathcal{T})} \Sigma_{\tau} \right)$} .$$

Now by using the estimates on the spacelike-null foliation and the fact that we start with data that are compactly supported, we get that every $\Psi_n$, $T\Psi_n $, $Y \Psi_n $ is uniformly bounded over all the spacelike foliations $\widetilde{\Sigma}_{\tau}$. This property will be inherited in the limit $\psi$ and then we can apply the continuation criterion \eqref{contlwp} and conclude that we can't have $\mathcal{T} < \infty$, hence our solution $\psi$ is actually a global one.

\begin{figure}[H]
\centering
\includegraphics[width=7cm]{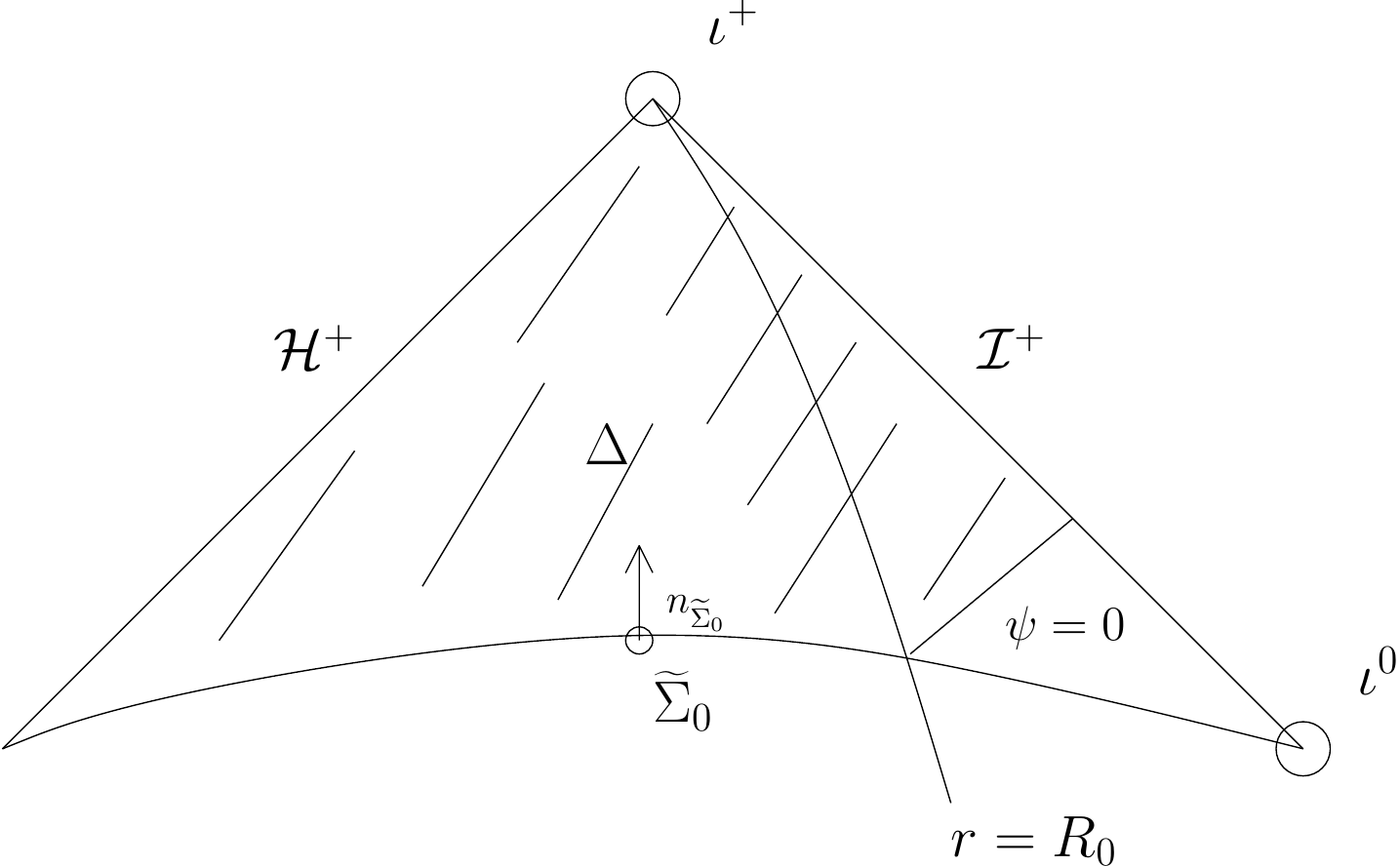}
\caption{By the finite speed of propagation we have that $supp (\psi ) \subset \Delta$. The area $\Delta$ is covered also by the the spacelike-null foliation $\{ \Sigma_{\tau} \}_{\tau \meg 0}$, while the two different foliations coincide up to $r \mik R_0$.}
\end{figure}

\section{Almost Conservation Law On The Horizon $\mathcal{H}^{+}$}\label{acwothh}
Having closed all our bootstrap assumptions, we proved in the previous section that we have global well-posedness for solutions of \eqref{nwef} for compactly supported data $(\ee \psi_0 , \ee \psi_1 ) \in C^{\infty}_0 \times C^{\infty}_0$ satisfying:
\begin{equation}\label{assum}
\| \ee \psi_0 \|_{H^s (\widetilde{\Sigma}_0 )} \mik \sqrt{E_0} \ee , \quad \| \ee \psi_1 \|_{ H^{s-1} (\widetilde{\Sigma}_0 )} \mik \sqrt{\widetilde{E}_0} \ee ,
\end{equation}
for $\ee$ as chosen before and for some $s > 5/2$.

In this class of initial data we will now show that on the future event horizon $\mathcal{H}^{+}$ we can obtain an almost conservation law.

Define the following quantity:
$$ H (\tau ) = Y \psi (\tau , M ) + \frac{1}{M} \psi (\tau , M ) .$$
Then we have:

\begin{thm}\label{acl}
For all spherically symmetric solutions $\psi$ of \eqref{nwef} with compactly supported initial data satisfying assumption \eqref{assum} we have that for all $\tau$:
\begin{equation}\label{aclh}
|H(\tau ) - H(0) | \mik O( \ee^2 )
\end{equation}
\end{thm}
\begin{proof}

By the properties of our solution $\psi$ we have that for the given initial data of \eqref{assum} the following bounds are satisfied for all $\tau$:
\begin{equation}\label{assumgwp1}
| \psi (\tau , M ) | \lesssim \dfrac{\sqrt{E_0} \ee}{(1+ \tau )^{3/5 - \aaa}} ,
\end{equation}
\begin{equation}\label{assumgwp2}
| T\psi (\tau , M )| \lesssim \sqrt{E_0} \ee ,
\end{equation}
\begin{equation}\label{assumgwp3}
| Y \psi (\tau , M )| \lesssim \sqrt{E_0} \ee ,
\end{equation}
for any $\aaa > 0$ (note that this $\aaa$ here is related to the $\aaa$ used in the previous sections by a factor of $3/10$, for our purposes this is not important). We fix $\aaa$ so that $\psi$ is $L^2$ integrable with respect to the $v$ variable.

Note that in this section we will use the notation $T = \partial_v$ (since we will always be working with the Eddington--Finkelstein coordinates $(v,r)$). 

In the spherically symmetric situation, on the event horizon (set $r=M$) we have the following equation:
\begin{equation}\label{nwhss}
TY \psi + \frac{1}{M} T \psi = A(\psi ) T \psi \cdot Y \psi \Rightarrow T \left( Y \psi + \frac{1}{M} \psi \right) = A(\psi ) T \psi \cdot Y \psi .
\end{equation}
We integrate \eqref{nwhss} with respect to $v$ and we have:
\begin{equation}\label{inwhss}
Y \psi (\tau , M) + \frac{1}{M} \psi (\tau , M) = Y \psi (0 , M) + \frac{1}{M} \psi (0 , M) + \int_0^{\tau} A(\psi ) T \psi \cdot Y \psi dv .
\end{equation}
We proceed by repeatedly integrating by parts the right hand side of \eqref{inwhss}.
$$ \int_0^{\tau} A(\psi ) T \psi \cdot Y \psi dv = A(\psi ) \psi \cdot Y \psi |_0^{\tau} - \int_0^{\tau} A(\psi ) \psi \cdot ( TY \psi ) dv - \int_0^{\tau} A' (\psi ) \psi \cdot T \psi \cdot Y \psi dv = $$ $$ = A(\psi ) \psi \cdot Y \psi |_0^{\tau} + \frac{1}{M} \int_0^{\tau} A(\psi ) \psi \cdot T \psi dv - \int_0^{\tau} \left( A^2 (\psi ) + A' (\psi ) \right) \psi \cdot T \psi \cdot Y \psi dv ,$$
where we used the equation \eqref{nwhss} for the second term of the first line.

We look at the following term:
$$ \frac{1}{M} \int_0^{\tau} A(\psi ) \psi \cdot T \psi dv = \frac{1}{2M} A(\psi ) \psi^2 |_0^{\tau} - \frac{1}{2M} \int_0^{\tau} A' (\psi ) \psi^2 \cdot T \psi dv ,$$
and by estimating in absolute value and using the triangle inequality we have:
$$ \left| \frac{1}{M} \int_0^{\tau} A(\psi ) \psi \cdot T \psi dv \right| \lesssim \dfrac{a_0 E_0 \ee^2}{M} + \dfrac{a_1 \sqrt{E_0} \ee }{M} \int_0^{\tau} \dfrac{E_0 \ee^2}{(1+ v )^{6/5 - 2\aaa}} dv \lesssim $$ $$ \lesssim \dfrac{a_0 E_0 \ee^2}{M} + \dfrac{a_1 E_0^{3/2} \ee^3 }{M} ,$$
by using estimates \eqref{assumgwp1}, \eqref{assumgwp2} and the boundedness of $A$ and $A'$.

Now we investigate the following term:
$$ \int_0^{\tau} \left( A^2 (\psi ) + A' (\psi ) \right) \psi \cdot T \psi \cdot Y \psi dv := \int_0^{\tau} B(\psi ) \psi \cdot T \psi \cdot Y \psi dv = $$ $$ = B(\psi ) \psi^2 \cdot Y \psi |_0^{\tau} - \int_0^{\tau} B' (\psi ) \psi^2 \cdot T \psi \cdot Y \psi dv - $$ $$ - \int_0^{\tau} B(\psi ) \psi^2 \cdot ( TY \psi ) dv - \int_0^{\tau} B(\psi ) \psi \cdot T \psi \cdot Y \psi dv \Rightarrow $$ $$ \Rightarrow \int_0^{\tau} B(\psi ) \psi \cdot T \psi \cdot Y \psi dv = \frac{1}{2} B(\psi ) \psi^2 \cdot Y \psi |_0^{\tau} - \frac{1}{2} \int_0^{\tau} B' (\psi ) \psi^2 \cdot T \psi \cdot Y \psi dv  - $$ $$ - \frac{1}{2} \int_0^{\tau} B(\psi ) \psi^2 \cdot ( TY \psi ) dv .$$ 
For the second term of the right hand side of the above equation we have:
$$ \left| \int_0^{\tau} B' (\psi ) \psi^2 \cdot T \psi \cdot Y \psi dv \right| \lesssim (a_0^2 + a_1 ) E_0 \ee^2 \int_0^{\tau} \dfrac{E_0 \ee^2}{( 1 + v )^{6/5 - 2\aaa}} dv \lesssim (a_0^2 + a_1 ) E_0^2 \ee^4 .$$
For the third term we use again the equation \eqref{nwhss}:
$$ \int_0^{\tau} B(\psi ) \psi^2 \cdot ( TY \psi ) dv = \int_0^{\tau} A(\psi ) B(\psi ) \psi^2 \cdot T \psi \cdot Y \psi dv - \frac{1}{M} \int_0^{\tau} B(\psi ) \psi^2 \cdot T \psi dv .$$
We examine these last terms separately. For the first one we have:
$$ \left| \int_0^{\tau} A(\psi ) B(\psi ) \psi^2 \cdot T \psi \cdot Y \psi dv\right| \lesssim $$ $$ \lesssim a_0 (a_0^2 + a_1 ) E_0 \ee^2 \int_0^{\tau} \dfrac{E_0 \ee^2}{( 1 + v )^{6/5 - 2\aaa}} dv  \lesssim a_0 (a_0^2 + a_1 ) E_0^2 \ee^4 ,$$
where we used the estimates \eqref{assumgwp1}, \eqref{assumgwp2}, \eqref{assumgwp3} and the boundedness of $A$ and $A'$.

We turn now to the second one:
$$ \left| \frac{1}{M} \int_0^{\tau} B(\psi ) \psi^2 \cdot ( TY \psi ) dv\right| \lesssim $$ $$ \lesssim (a_0^2 + a_1 ) \sqrt{E_0} \ee \int_0^{\tau} \dfrac{E_0 \ee^2}{( 1 + v )^{6/5-2\aaa}} dv \lesssim (a_0^2 + a_1 ) E_0^{3/2} \ee^3 .$$
So in the end we have that:
$$ \left| \int_0^{\tau} B(\psi ) \psi \cdot T \psi \cdot Y \psi dv \right| \lesssim 2(a_0^2 + a_1 ) E_0^{3/2} \ee^3 + (a_0 +1 ) (a_0^2 + a_1 ) E_0^2 \ee^4 .$$

Gathering together all the above computations we have in the end the following:
\begin{equation}\label{acons}
\left| \left( Y \psi (\tau , M) + \frac{1}{M} \psi (\tau , M) \right) - \left(Y \psi (0 , M) + \frac{1}{M} \psi (0 , M) \right)  \right| \lesssim  
\end{equation}
$$ \lesssim \dfrac{a_0 E_0 \ee^2}{M} + \left( 2(a_0^2 + a_1 ) E_0^{3/2} + \dfrac{a_1 E_0^{3/2}}{M} \right) \ee^3 +  (a_0 +1 ) (a_0^2 + a_1 ) E_0^2 \ee^4 .$$
\end{proof}

The reason for characterising estimate \eqref{acons} as an almost conservation law is that we can choose $\ee > 0$ to be small enough (smaller always than the $\ee$ that we get for the global well-posedness theorem) so that the right hand side of \eqref{acons} is always smaller than $\ee^{1+\eta}$ for some $\eta > 0$. This result combined with the decay of $\psi$ forces $Y \psi$ to be almost conserved on the horizon.

\section{Asymptotic Blow-Up For Higher Derivatives On The Horizon $\mathcal{H}^{+}$}\label{asbufhd}
In this section we will further examine the asymptotic behaviour of spherically symmetric solutions of \eqref{nwef}. In particular we will prove a result that is related to the almost conservation law of the previous section, specifically we will show that all 2nd or higher derivatives with respect to $r$ (always in the Eddington--Finkelstein coordinate system $(v,r)$) of solution that come from initial data satisfying a positivity condition on the future event horizon blow up along the future event horizon $\mathcal{H}^{+}$.

\begin{thm}\label{inst}
For a spherically symmetric global solution $\psi$ of \eqref{nwef} that was obtained in the previous sections (i.e. we start with data that are less than $\ee > 0$ in some $H^s$ norm for $s > 5/2$, for $\ee$ chosen appropriately) with initial data that satisfy:
\begin{equation}\label{posblow}
 \psi_0 (M) > 0 , \quad Y \psi (0 , M ) > 0 ,
\end{equation}
 we have that on the event horizon $r=M$ the following asymptotic behaviour can be observed:
$$ | ( Y^k \psi ) (\tau , M )| \rightarrow \infty \mbox{ as $\tau \rightarrow \infty$} ,$$
for $k \meg 2$. 
\end{thm}
\begin{rem}
If we don't assume the positivity condition \eqref{posblow}, then it should be expected that we will have in general non-decay for the second derivative $Y^2 \psi$ and blow-up for the third and higher derivatives $Y^k \psi$, $k \meg 3$, in analogy with the linear case as it was shown in \cite{aretakis2012}. We won't pursue this direction here.
\end{rem}
\begin{proof}
First we differentiate with respect to $r$ equation \eqref{nwef} and we evaluate the equation \eqref{nwef} on the horizon:
\begin{equation}\label{nwefr}
T Y^2 \psi + \frac{1}{M} TY \psi - \frac{1}{M^2} T \psi + \frac{1}{M^2} Y \psi = A' (\psi ) (Y \psi )^2 \cdot T \psi + 
\end{equation}
$$ + A(\psi ) ( TY \psi ) \cdot Y \psi + A(\psi ) ( Y^2 \psi ) \cdot T \psi . $$

We integrate \eqref{nwefr} with respect to $v$ for some $\tau_0$ big enough that will be chosen later, and we have:
\begin{equation}\label{intr}
\int_{\tau_0}^{\tau} ( T Y^2 \psi ) dv = \frac{1}{M} Y \psi (\tau_0) - \frac{1}{M} Y \psi (\tau) + \frac{1}{M^2} \psi (\tau ) - \frac{1}{M^2} \psi (\tau_0) - 
\end{equation}
$$ - \frac{1}{M^2} \int_{\tau_0}^{\tau} ( Y \psi ) dv + \int_{\tau_0}^{\tau} A' (\psi ) (Y \psi )^2 \cdot T \psi dv + $$ $$ + \int_{\tau_0}^{\tau} A(\psi ) ( TY \psi ) \cdot Y \psi dv + \int_{\tau_0}^{\tau} A(\psi ) ( Y^2 \psi ) \cdot T \psi dv .$$
We look first at the last term of \eqref{intr}:
$$ \int_{\tau_0}^{\tau} A(\psi ) ( Y^2 \psi ) \cdot T \psi dv = A(\psi ) \psi  \cdot ( Y^2 \psi ) |_{\tau_0 }^{\tau} - \int_{\tau_0}^{\tau} A' (\psi ) \psi \cdot T \psi \cdot ( Y^2 \psi ) dv - $$ $$ - \int_{\tau_0}^{\tau} A(\psi ) \psi \cdot ( T Y^2 \psi ) dv .$$
Going back to \eqref{intr} we have:
\begin{equation}\label{intr2}
\int_{\tau_0}^{\tau} ( T Y^2\psi ) (1+A(\psi ) \psi ) dv = \frac{1}{M} Y \psi (\tau_0) - \frac{1}{M} Y \psi (\tau) + \frac{1}{M^2} \psi (\tau ) - \frac{1}{M^2} \psi (\tau_0) + 
\end{equation}
$$ + A(\psi ) \psi (\tau) ( Y^2 \psi ) (\tau) - A(\psi ) \psi (\tau_0) ( Y^2 \psi ) (\tau_0) - \frac{1}{M^2} \int_{\tau_0}^{\tau} ( Y \psi ) dv + $$ $$ + \int_{\tau_0}^{\tau} A' (\psi ) (Y \psi )^2 \cdot T \psi dv + \int_{\tau_0}^{\tau} A(\psi ) ( TY \psi ) \cdot Y \psi dv - \int_{\tau_0}^{\tau} A' (\psi ) \psi \cdot T \psi \cdot ( Y^2 \psi ) dv .$$
We use the boundedness of $A'$ and we get that:
\begin{equation}\label{intr5}
\int_{\tau_0}^{\tau} ( T Y^2\psi ) (1+A(\psi ) \psi ) dv \mik \frac{1}{M} Y \psi (\tau_0) - \frac{1}{M} Y \psi (\tau) + \frac{1}{M^2} \psi (\tau ) - \frac{1}{M^2} \psi (\tau_0) +
\end{equation}
$$ + A(\psi ) \psi (\tau)  ( Y^2 \psi ) (\tau) - A(\psi ) \psi (\tau_0) ( Y^2 \psi ) (\tau_0) - \frac{1}{M^2} \int_{\tau_0}^{\tau} ( Y \psi ) dv + $$ $$ + \int_{\tau_0}^{\tau} A' (\psi ) (Y \psi )^2 \cdot T \psi dv + \int_{\tau_0}^{\tau} A(\psi ) ( TY \psi ) \cdot Y \psi dv + \frac{a_1}{2} \int_{\tau_0}^{\tau}  T (\psi^2 ) \cdot ( Y^2 \psi ) dv .$$
For the last term we have:
$$\frac{a_1}{2} \int_{\tau_0}^{\tau}  T (\psi^2 ) ( Y^2 \psi ) dv = \frac{a_1}{2} \psi^2 ( Y^2 \psi ) |_{\tau_0}^{\tau} - \frac{a_1}{2} \int_{\tau_0}^{\tau}  \psi^2 \cdot (T Y^2 \psi ) dv .$$
Then we have that:
\begin{equation}\label{intr6}
\int_{\tau_0}^{\tau} ( T Y^2\psi ) \left( 1+A(\psi ) \psi + \frac{a_1}{2} \psi^2 \right) dv \mik \frac{1}{M} Y \psi (\tau_0) - \frac{1}{M} Y \psi (\tau) + \frac{1}{M^2} \psi (\tau ) - \frac{1}{M^2} \psi (\tau_0) +
\end{equation}
$$ + A(\psi ) \psi (\tau) ( Y^2 \psi ) (\tau) - A(\psi ) \psi (\tau_0) ( Y^2 \psi ) (\tau_0) + \frac{a_1}{2} \psi^2 (\tau ) ( Y^2 \psi ) (\tau ) - \frac{a_1}{2} \psi^2 (\tau_0 ) ( Y^2 \psi ) (\tau_0 ) - $$ $$-\frac{1}{M^2} \int_{\tau_0}^{\tau} ( Y \psi ) dv  + \int_{\tau_0}^{\tau} A' (\psi ) (Y \psi )^2 \cdot T \psi dv + \int_{\tau_0}^{\tau} A(\psi ) ( TY \psi ) \cdot Y \psi dv  .$$

We use estimate \eqref{assumgwp1} and we choose an appropriate $\tau_0$ so that for any $\tau' \meg \tau_0$ we have that:
$$ 1 + A(\psi ) \psi (\tau') + \frac{a_1}{2} \psi^2 (\tau' ) \meg \frac{3}{4} .$$
Then we have that:
$$ \int_{\tau_0}^{\tau} ( T Y^2\psi ) \left( 1+ A(\psi ) \psi + \frac{a_1}{2} \psi^2 \right) dv \meg \frac{3}{4} \int_{\tau_0}^{\tau} ( T Y^2 )\psi dv = \frac{3}{4} ( Y^2 \psi ) (\tau) - \frac{3}{4} ( Y^2 \psi ) (\tau_0 ) ,$$
which going back to \eqref{intr6} and using again \eqref{assumgwp1} allows us to absorb the terms $$ \psi (\tau) ( Y^2 \psi ) (\tau) \mbox{ and } \frac{a_1}{2} \psi^2 (\tau ) Y^2 \psi (\tau )$$
 to the left hand side of \eqref{intr6}:
\begin{equation}\label{intr3}
( Y^2 \psi ) (\tau) \mik \frac{3}{2} ( Y^2 \psi ) (\tau_0 ) + \frac{2}{M} Y \psi (\tau_0) - \frac{2}{M} Y \psi (\tau) + \frac{2}{M^2} \psi (\tau ) - \frac{2}{M^2} \psi (\tau_0) + 
\end{equation}
$$ - 2\psi (\tau_0) ( Y^2 \psi ) (\tau_0) - \frac{a_1}{2} \psi^2 (\tau_0 ) ( Y^2 \psi ) (\tau_0 ) - $$ $$ - \frac{2}{M^2} \int_{\tau_0}^{\tau} ( Y \psi ) dv + 2\int_{\tau_0}^{\tau} A' (\psi ) (Y \psi )^2 \cdot T \psi dv + 2\int_{\tau_0}^{\tau} A(\psi ) ( TY \psi ) \cdot Y \psi dv .$$
Now we look at the term $\frac{2}{M^2} \int_0^{\tau} ( Y \psi ) dv$. We use the smallness of the initial data and the positivity of $\psi_0 (M)$ and $Y \psi (0,M)$. From Theorem \ref{acl} we have:
$$ Y \psi (\tau , M ) + \frac{1}{M} \psi (\tau , M ) = Y \psi (0, M) + \frac{1}{M} \psi (0,M) +\ee_1 = C_M \ee + \ee_1 ,$$
for $C_M = c + \frac{c}{M}$ where $0 \mik c \mik 1$ and for $\ee_1$ some number (might be positive or negative) which is much smaller in absolute value than $\ee$. By estimate \eqref{assumgwp1} we have:
$$ \left| \int_{\tau_0}^{\tau} \psi dv \right| \mik \int_{0}^{\tau} |\psi| dv \lesssim \sqrt{E_0} \ee (1+ \tau)^{2/5 + \aaa} ,$$
for $\aaa > 0$ that we have chosen to be small enough (so that $\psi$ is $L^2$ integrable with respect to $v$) which gives us here that $2/5 + \aaa < 1$.

By the last two inequalities we have that:
$$\int_{\tau_0}^{\tau} ( Y \psi ) dv = \int_{\tau_0}^{\tau} \left( Y \psi + \frac{1}{M} \psi - \frac{1}{M} \psi \right) dv \gtrsim $$ $$ \gtrsim \ee (\tau - \tau_0 ) + \ee_1 (\tau - \tau_0 ) - \ee (1+ \tau)^{2/5 + \aaa} \gtrsim \ee \tau \Rightarrow $$
\begin{equation}\label{assum1}
-\int_{\tau_0}^{\tau} ( Y \psi )  dv\lesssim - \ee \tau ,
\end{equation}
which is true by all the above assumptions asymptotically on $\mathcal{H}^{+}$. 

Next we look at the term $2\int_{\tau_0}^{\tau} A(\psi ) ( TY \psi ) \cdot Y \psi$ using the equation on the horizon \eqref{nwhss}:
$$ 2\int_{\tau_0}^{\tau} A(\psi ) ( TY \psi ) \cdot Y \psi dv = 2\int_{\tau_0}^{\tau} A^2 (\psi ) T \psi \cdot (Y \psi )^2 dv - \frac{2}{M} \int_{\tau_0}^{\tau} T \psi \cdot Y \psi dv  $$
$$ \Rightarrow \left|\int_{\tau_0}^{\tau} A(\psi ) ( TY \psi ) \cdot Y \psi dv \right| \lesssim E_0^{3/2} \ee^3 (\tau - \tau_0 ) + E_0 \ee^2 (\tau - \tau_0 ) $$
\begin{equation}\label{assum2}
2\int_{\tau_0}^{\tau} A(\psi ) ( TY \psi ) \cdot Y \psi dv \lesssim E_0 \ee^2 \tau , 
\end{equation}
asymptotically on $\mathcal{H}^{+}$. 

Finally we also have:
$$\left| \int_{\tau_0}^{\tau} A' (\psi ) (Y \psi )^2 T \psi dv \right| \lesssim E_0^{3/2}\ee^3 (\tau - \tau_0 ) ,$$ 
which implies that asymptotically on $\mathcal{H}^{+}$ we have:
$$ 2 \int_{\tau_0}^{\tau} A' (\psi ) (Y \psi )^2 T \psi dv \lesssim E_0^{3/2} \ee^3 \tau .$$
By the smallness of $\ee$ we have that:
\begin{equation}\label{assum3}
-\frac{2}{M^2} \int_{\tau_0}^{\tau} ( Y \psi ) dv + 2 \int_{\tau_0}^{\tau} A' (\psi ) (Y \psi )^2 \cdot T \psi dv + 2\int_{\tau_0}^{\tau} A(\psi ) ( TY \psi ) \cdot Y \psi dv \mik -C_{\ee} \tau ,
\end{equation}
for some $C_{\ee} > 0$.

Now we go back to \eqref{intr3} and putting together all the above estimates we have:
$$ ( Y^2 \psi ) (\tau) \mik G(\psi, Y \psi, Y^2 \psi) (\tau_0 ) + \textbf{C}_{\ee} \dfrac{1}{(1+\tau)^{3/5-\aaa}} - C_{\ee} \tau $$
\begin{equation}\label{intr4}
( Y^2 \psi ) (\tau) \mik - c_{\ee} \tau ,
\end{equation}
for some appropriate constant $c_{\ee} = c(\ee) > 0$ asymptotically on $\mathcal{H}^{+}$ and some function $G$ that depends on the values of $\psi$, $Y \psi$, $Y^2 \psi$ at $\tau_0$. This last inequality \eqref{intr4} finishes the proof of the Theorem for the case $k=2$. For $k > 2$ the proof is similar.

\end{proof} 

\section{Remarks On Other Nonlinearities}\label{othernon}
As we mentioned in Section \ref{Main}, we carried all computations so far for the nonlinear term being the null form $A(\psi) g^{\alpha \beta} \partial_{\alpha} \psi \partial_{\beta} \psi$. In this Section we will deal with the other nonlinear terms that are present in \eqref{nw}, and we will also look at some different nonlinearities for which we have some different results. 

We won't try to optimize the results though, so we will just mention what else can be included in the nonlinearity $F$ apart from the nonlinear null form term. 

Let us look for now at the nonlinear equation
$$ \Box_g \psi = |\psi |^l . $$
To make things easier, we will use (and of course close) the bootstrap assumption:
\begin{equation}\label{bassum}
\int_{\si_{\tau}} |F|^2 d\mi_{g_{\si}} \lesssim E_0 \ee^2 (1+\tau )^{-3+\aaa} ,
\end{equation}
for all $\tau$. This will actually imply the other two bootstrap assumptions of Theorem \ref{boot} (while this assumption itself is better than the first of Theorem \ref{boot}).

Let 
$$ l = l(\aaa) = l_0 + l_1 \mbox{ with } l_0 \meg \dfrac{3-\aaa}{6/5 - 3\aaa/5} \mbox{ and } l_1 > \dfrac{1}{2-2\aaa} , $$
for any given $\aaa$ such that 
$$\dfrac{3}{5} - \dfrac{3\aaa}{10} > \dfrac{1}{2}.$$

Then the bootstrap for \eqref{bassum} for the term $O(|\psi |^l )$ can easily close since
$$ |\psi (\tau) |^{2l_0} \lesssim E_0^{l_0} \ee^{2l_0} (1+\tau )^{3-\aaa} ,$$
and $|\psi |^{2l_1}$ is integrable with respect to $r$. One can then easily check that the continuation criterion is satisfied (note that now we have to control additionally $\psi$ in $L^{\infty}$ but this is true of course as well). The only difference is that the equations in Section \ref{drbound} are not the same, but similar estimates can be derived by using the decay for $\psi$ and the properties of $\partial_u r$.

Now we examine a different class of nonlinear equations that have the following form:
\begin{equation}\label{nwwnc}
\Box_g \psi = (\psi )^{2n} + \chi (T \psi )^{2n} + \chi (Y\psi )^{2n} \mbox{ for $n \in \mathbb{N}$, $n \meg 2$,} 
\end{equation}
for $\chi := \chi (r) \in C^{\infty}_0 ([M,\infty))$ a smooth cut-off that is 1 in $[M, M+c/2]$ and 0 in $[M+c , \infty)$ for some constant $c < \infty$.

We include here for completeness an argument of Aretakis from \cite{aretakis2013} that shows that we have finite time blow-up on the horizon $\mathcal{H}^{+}$ for (spherically symmetric) solutions of \eqref{nwwnc}. We have that on the horizon equation \eqref{nwwnc} becomes:
\begin{equation}\label{nwwnch}
T\left( Y \psi + \frac{1}{M} \psi \right) = \frac{1}{2} \left[ (\psi )^{2n} + (T\psi )^{2n} + (Y \psi )^{2n} \right] .
\end{equation}
By the trivial inequality $a^{2n} + b^{2n} \meg C_n (a+b)^{2n}$ for any $a,b \in \rr$ we get that \eqref{nwwnch} gives us:
$$ (\psi )^{2n} + (T\psi )^{2n} + (Y \psi )^{2n} \meg C_n \left( Y \psi + \frac{1}{M} \psi \right)^{2n} \Rightarrow $$
$$ \Rightarrow T \left( Y \psi + \frac{1}{M} \psi \right) \meg C_n \left( Y \psi + \frac{1}{M} \psi \right)^{2n} \Rightarrow $$
\begin{equation}\label{nwwnch1}
\Rightarrow T H(\tau) \meg C_n H^{2n} (\tau) \meg 0  ,
\end{equation}
in the notation of Section \ref{acwothh}. So if at $\tau = 0$ we have that $ H(0) = \eta > 0$ we get $H(\tau ) > 0$ for all $\tau$.

All we have to do now is to integrate the inequality \eqref{nwwnch1} and arrive at:
$$ T \left( -\dfrac{1}{(2n-1) H^{2n-1} (\tau )} \right) \meg C_n \Rightarrow \dfrac{1}{\eta^{2n-1}} - \dfrac{1}{H^{2n-1} (\tau)} \meg C_n (2n-1) \tau \Rightarrow $$ $$ \Rightarrow \dfrac{1}{H^{2n-1} (\tau)} \mik - C_n (2n-1) \tau + \frac{1}{\eta^{2n-1}} , $$
which shows that $H(\tau) \rightarrow \infty$ as $\tau \rightarrow \dfrac{1}{C_n (2n-1) \eta^{2n-1}}$ (i.e. the solution of \eqref{nwwnch1} blows up in $C^1$).

But we just saw that spherically symmetric solutions of the equation
\begin{equation}
\Box_g \psi = (\psi )^{2n}
\end{equation}
behave well for $n$ big enough. So in the end, we can state the following conjecture that asks for a reinterpretation of the Aretakis instability phenomenon as a shock formation mechanism in a certain sense.
\begin{conj}\label{shockconj}
Spherically symmetric solutions of equation \eqref{nwwnc} have \textbf{first} singularities on the event horizon. 
\end{conj}

\section{Acknowledgements}
First and foremost I would like to thank Stefanos Aretakis for all his support, his advice, and all the time that he spent for this work. I would also like to thank Georgios Moschidis for many stimulating discussions concerning this paper. Finally, I'd like to thank Mihalis Dafermos, Yakov Shlapentokh-Rothman and Shiwu Yang for sharing with me many useful insights, as well as Jonathan Luk for helpful comments on a preliminary version of this manuscript.

\appendix
\section{The Extremal Reissner-Nordstr\"{o}m Solution In Null Coordinates}\label{ngrn}
The extremal Reissner-Nordstr\"{o}m black-hole spacetime is a spherically symmetric solution of the Einstein--Maxwell system of equations:
$$ R_{\alpha \beta} - \frac{1}{2} R g_{\alpha \beta} = 2 T_{\alpha \beta} ,$$
$$ \nabla^{\alpha} F_{\alpha \beta} = 0 , \quad dF = 0 ,$$
where
$$ T_{\alpha \beta} = 2 \left( F^{\rho}_{\alpha} F_{\beta \rho} - \frac{1}{4} g_{\alpha \beta} F^{ab} F_{ab} \right) .$$

Under the assumption of spherical symmetry the Einstein--Maxwell equations reduce to a number of simpler equations. For a detailed and elementary treatment of this topic (which can be generalized to the setting of the Einstein equations being coupled to any matter field) we refer to the lecture notes of Schlue \cite{volkerlectures}, and for the more interested reader we mention the works of Christodoulou \cite{DC93}, \cite{DC91}, Dafermos \cite{MD03}, \cite{MD12}, Dafermos and Rodnianski \cite{MDIR05} and Kommemi \cite{JK10b}. In spherical symmetry we are looking at our spacetime $\mathcal{M}$ as follows: define the quotient space
$$\mathcal{Q} = \mathcal{M} \diagup SO(3) , $$
which is an $1+1$-dimensional manifold with boundary. The metric of $\mathcal{M}$ can be split into two parts
$$ g_{\mathcal{M}} = g_{\mathcal{Q}} + g_{\mathbb{S}^2} ,$$
where $g_{\mathbb{S}^2} = r^2 (x) \gamma_{\mathbb{S}^2}$ for $x \in \mathcal{Q}$ and for $r : \mathcal{Q} \rightarrow \rr_{\meg 0}$ the area radius function.

We now consider a double null coordinate system for our solution, under which the metric takes the form:
$$ g = -\Omega^2 dudv + r^2 \gamma_{\mathbb{S}^2} .$$
We note that $\mathcal{Q}$ contains a subset $\mathcal{D} = \{ (u,v) \in [0,U] \times [0,\bar{V}] \}$ with the future event horizon $\mathcal{H}^{+}$ being on $\{ (u,v) | u=U, v \meg 0 \}$ and the future null infinity $\mathcal{I}^{+}$ on $\{ (u,v) | v=\bar{V} , u \meg 0 \}$ for $\bar{V} \mik \infty$. Moreover the function $\Omega$ (which is not arbitrary) is a strictly positive $C^1$ function on $\mathcal{D}$ and the variables $u$ and $v$ increase towards the future (i.e. $\nabla u$ and $\nabla v$ are future pointing).

$\mathcal{D}$ has the following Penrose diagram:
\begin{figure}[H]
\centering
\includegraphics[width=6cm]{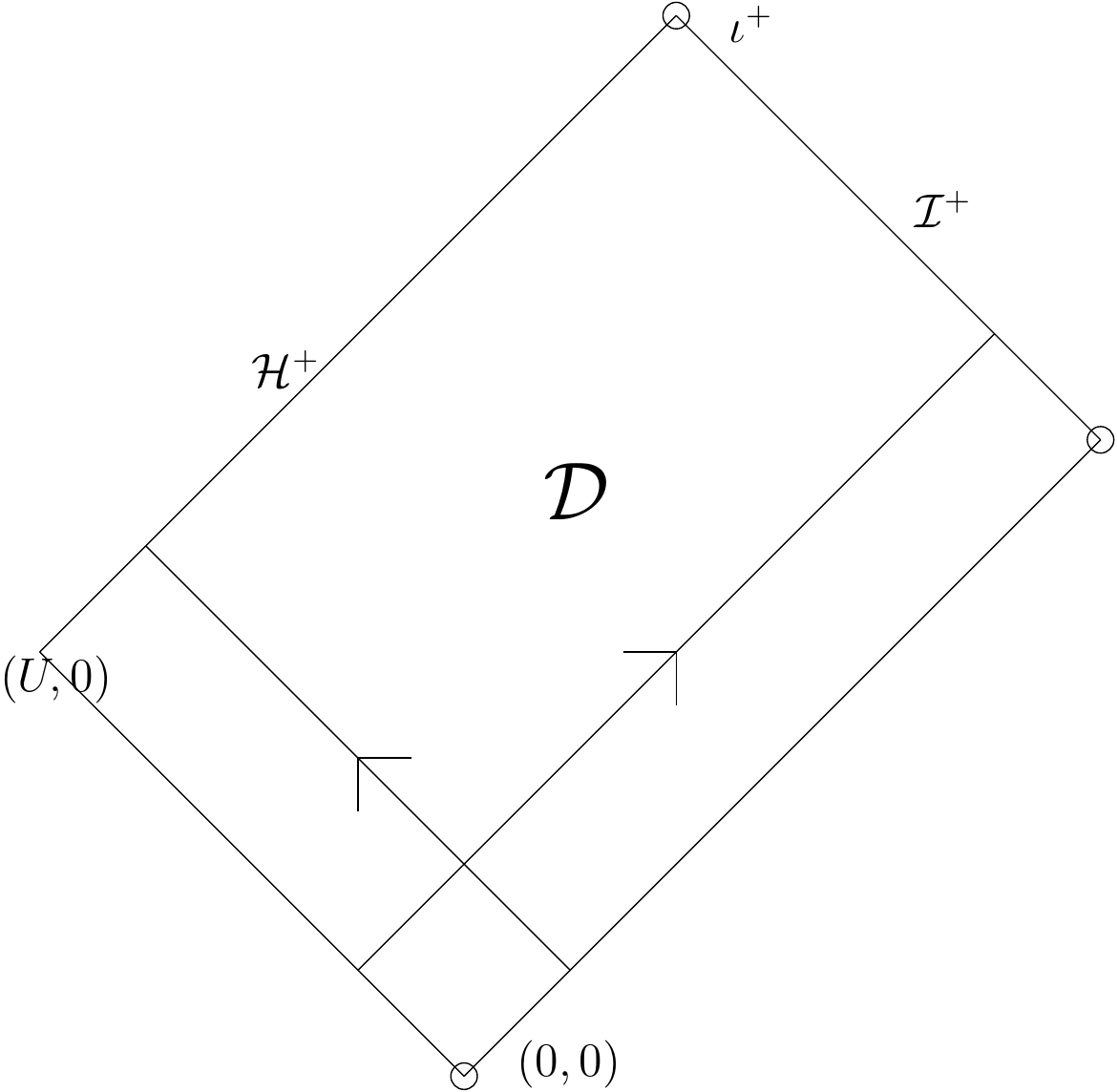}
\end{figure}

We will actually be more specific here and define the $v$ variable to be $v = t + r^{*}$. On the other hand we renormalize the $u$ variable  to be regular on the future event horizon (i.e. $U < \infty$). \footnote{Recall that in the null coordinate system that we already used we considered $u = t- r^{*}$ and $v$ as above, in this case $u$  becomes $-\infty$ on the horizon.}

In the same null coordinate system the Maxwell field takes the following form:
$$ F = \dfrac{\Omega^2 E}{2r^2} du \wedge dv + B \sin \theta d\theta \wedge d\phi ,$$
for real constants $E, B$ (the electric and the magnetic charge respectively). By denoting $e = \sqrt{E^2 + B^2}$ we can compute the energy momentum tensor:
\begin{equation}\label{mtmaxwell}
 T_{uu} = T_{vv} = 0, \quad T_{uv} = \dfrac{\Omega^2 M^2}{4 r^4} ,
 \end{equation}
since in the extremal case we have that $e^2 = M^2$ for the mass function $M$ which is defined by:
$$ \left( 1 - \frac{M}{r} \right)^2 = -\frac{4}{\Omega^2} \partial_u r \partial_v r .$$

The Einstein--Maxwell system gives us the following equations for the area radius function $r$ (combined with the computation \eqref{mtmaxwell}):
\begin{equation}\label{ruv}
\partial_{uv} r = -\frac{1}{r} \partial_u r \partial_v r - \frac{\Omega^2}{4r} + r T_{uv} , 
\end{equation}
\begin{equation}\label{ruu}
\partial_u \left( \dfrac{\partial_u r}{\Omega^2} \right) = 0 ,
\end{equation}
\begin{equation}\label{rvv}
\partial_v \left( \dfrac{\partial_v r}{\Omega^2} \right) = 0 . 
\end{equation}
Substituting the mass equation in \eqref{ruv} we get that:
\begin{equation}\label{ruv2}
\partial_{uv} r = -\dfrac{\Omega^2}{2r^2} \left( M - \frac{M^2}{r} \right) .
\end{equation}
We can get the extremal Reissner-Nordstr\"{o}m solution by starting with $\partial_v r \meg 0$, $\partial_u r < 0$ initially. 

Note that here we have chosen $\partial_v r = D$, which implies that we have:
\begin{equation}\label{ruomega}
-\frac{4}{\Omega^2} \partial_u r = 1 .
\end{equation}
From equation \eqref{ruomega} we have that everywhere in $\mathcal{D}$ the following holds:
\begin{equation}\label{ruminus}
\partial_u r < 0 .
\end{equation}
Moreover from equation \eqref{ruv2} and the initial condition $\partial_u r < 0$ we can see that $\partial_u r$ remains always negative and increases in absolute value with $v$. 

Finally we rewrite \eqref{ruv2} in the following form using \eqref{ruomega}:
\begin{equation}\label{ruv1}
\partial_{uv} r = \dfrac{2\partial_u r}{ r^2} \left( M - \frac{M^2}{r} \right) .
\end{equation}

\section{The d'Alembertian On The Extremal Reissner-Nordstr\"{o}m Spacetime}
We record here the form that the d'Alembertian operator $\Box_g$ takes under the different coordinates systems that we use in this work for the extremal Reissner-Nordstr\"{o}m black hole spacetime.

In the Eddington--Finkelstein $(v,r)$ coordinates we have:
\begin{equation}\label{wave}
\Box_g \psi = D \partial_{rr} \psi + 2\partial_{vr} \psi + \frac{2}{r} \partial_v \psi + \left( D' + \frac{2D}{r} \right) \partial_r \psi + \slashed{\Delta} \psi ,
\end{equation}
where $\slashed{\Delta} = \frac{1}{r^2}\Delta_{\mathbb{S}^2}$. 

In the null coordinates $(u = t-r^{*} , v = t+r^{*} )$ we have:
\begin{equation}\label{wavenull}
\Box_g \psi = -\frac{4}{Dr} \partial_{uv} (r\psi ) - \frac{D'}{r} \psi + \slashed{\Delta} \psi .
\end{equation}
\bibliographystyle{acm}
\bibliography{bibliography}
\end{document}